\documentclass[11pt,reqno]{amsart}
\usepackage{xifthen,setspace}
\usepackage{pkgfile}

\newtheorem{assumption}{Assumption}
\usepackage{longtable}

\newtheorem{definition}{Definition}

\newcommand{\e}{{\varepsilon}}

\newcommand{\DD}{{\mathbb{D}}}

\newcommand{\dd}{{\delta}}
\newcommand\old[1]{}
\usepackage{booktabs}

\title{Competitive Erosion is Conformally Invariant}

\author{Shirshendu Ganguly and Yuval Peres}
\address{Department of Statistics \\ University of California\\ Berkeley, USA.}
\email{sganguly@berkeley.edu}
\address{ Microsoft Research\\ WA, USA.}
\email{peres@microsoft.com}

\def\R{\mathbb{R}}
\def\Z{\mathbb{Z}}
\def\div{\mathop{\mathrm{div}}}
\begin{document}
\maketitle

%\spacing{1.05}
%%Term definitions
%
%
%
%% Use the terms
%
%%Print the glossary

\begin{abstract}
We study a graph-theoretic model of interface dynamics 
called {\bf competitive erosion}.
Each vertex of the graph is occupied by a particle, which can be either red or blue.  New red and blue particles are emitted alternately from their respective  bases and perform random walk. On encountering a particle of the opposite color they remove it and occupy its position. We consider competitive erosion on discretizations of `smooth', planar, simply connected, domains. The main result of this article shows that at stationarity, with high probability, the blue and the red regions are separated by the level curves of the Green function, with Neumann boundary conditions, which are orthogonal circular arcs on the disc and hyperbolic geodesics on a general, simply connected domain. This establishes \emph{conformal invariance} of the model.   
\end{abstract}

\begin{figure}[h] 
\centering
\begin{tabular}{ccc}
 \includegraphics[width=.33\textwidth]{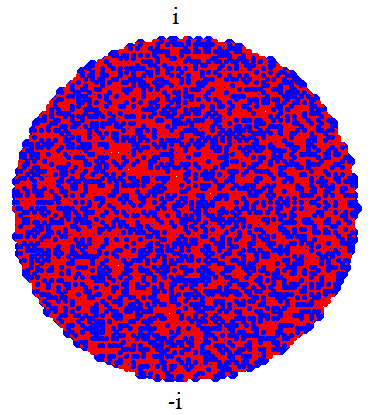}&
 \includegraphics[width=.33\textwidth]{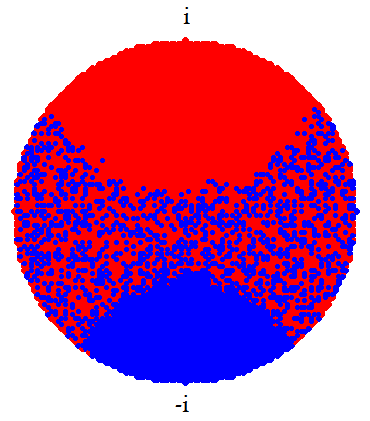} &
\includegraphics[width=.33\textwidth]{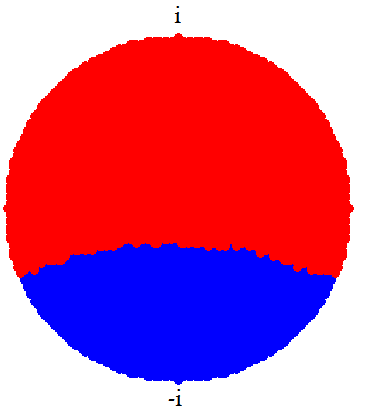} \\
$t=0$ & $t=1666$ & $t=5000$
\end{tabular}
\caption{ Competitive erosion on the unit disc, spontaneously forms an interface. $i$ and $-i$ form the red source and blue source respectively.
In the initial state each vertex of a mesh of size $\frac{1}{50}$ is independently colored blue with probability $1/3$ and red otherwise. After $5000$ time steps of competitive erosion, the red and blue regions have separated with the interface being an orthogonal circular arc.
}
\label{formation23}
\end{figure}
\tableofcontents

\section{Interface dynamics}\label{model}

In 2003, Propp introduced \emph{competitive erosion}: a graph-theoretic model of interface dynamics maintained in equilibrium by equal and opposing forces.
The model has the following underlying data:
	\begin{itemize}
	\item A finite connected graph $G=(V,E)$ with vertex set $V$ and edge set $E$.
	\item Probability measures $\mu_1$ and $\mu_2$ on $V$.  
	\item An integer $0 \leq k \leq |V|-1$.
	\end{itemize}
\emph{Competitive erosion} is a discrete-time Markov chain $\{{\rm{CE}}(t)\}_{t \geq 0}$ on the space $\left\{S\subset V:|S|=k \right\}.$
One time step is defined by
\begin{equation}\label{dynamics}
	{\rm{CE}}(t+1) = \{{\rm{CE}}(t) \cup \{X_t\}\} \setminus \{Y_t\}, 
	\end{equation}
where $X_t$ is the first site in $V\setminus{\rm{CE}}(t)$ visited by a simple random walk, whose starting point has distribution $\mu_1$; and $Y_t$ is the first site in ${\rm{CE}}(t) \cup \{X_t\}$ visited by an independent simple random walk, whose starting point has distribution $\mu_2$. We will think of every vertex in $V$ as colored either `blue' or `red' and ${\rm{CE}}(t)$ will denote the set of blue vertices, and accordingly we rename the starting measures $\mu_1$ and $\mu_2$ as $\mu_{\rm{B}}$ and $\mu_{\rm{R}}$ respectively.
If the blue and red sources $\mu_{\rm{B}}$ and $\mu_{\rm{R}}$ have well-separated supports, one expects that the dynamics separates the graph into coherent red and blue territories. 

Note that \emph{Competitive erosion} can be thought of as a competing  version of Internal Diffusion Limited Aggregation (IDLA), which is a fundamental growth model; first proposed by Meakin and Deutch in 1986,  as a model of industrial chemical processes such as electropolishing, corrosion and etching.  Competing particle systems, modeling co-existence of various species etc, have been the subject of intense study in physical sciences as well as mathematics:  competing versions of the Richardson model were considered in \cite{cr2,cr1}, processes modeling annihilation between competing species were studied in a series of works, see for e.g., \cite{bra1,bra2}. 
The study of \emph{competitive erosion} was initiated in \cite {cyl} by the authors along with Lionel Levine and James Propp, where the reader can find a detailed introduction of the process. The underlying graph considered in \cite{cyl} was the cylinder (the product of a path and a cycle). However, the original question asked by Propp (\cite{pcpropp}) was in the setting where the underlying graph is a discrete approximation to a general smooth simply connected planar domain. The process was predicted to exhibit \emph{conformal invariance}. Confirming this is the goal of the article. More on comparison with IDLA is presented in Section \ref{idla1}.

\subsection{Conformal Invariance}As mentioned above,
we consider the case, when the underlying graph is a discretization of a smooth, bounded, planar, simply connected domain $\U,$ i.e., $V=\U_n:=\frac{1}{n}\mathbb{Z}^2\cap \U$ and $k=\lfloor{\alpha |\U_n|\rfloor} $ for some fixed $\alpha \in (0,1/2]$. When the domain is the unit disc $\DD,$ thought of as a subset of the complex plane $\C,$ simulations (Figure ~\ref{formation23}) show that, if the measures $\mu_{\rm{B}}$ and $\mu_{\rm{R}}$ are delta masses  on $-i,i,$ then after running the process for some time, the blue and red regions are separated  and the blue region seems to converge to a subset of $\D$ having $\alpha$ fraction of the total area and bounded by an orthogonal circular arc. Let us call the latter region $\DD_{(\alpha)}$. It is well known that the boundary of   $\DD_{(\alpha)}$ is a level set of the Green function with Neumann boundary conditions on $\DD$ with source and sink at $-i$ and $i$ respectively. For a general domain $\U,$ with two points $x_{\rm{B}}$ and $x_{\rm{R}}$ on the boundary, one  can obtain the corresponding $\U_{(\alpha)}$ by looking at the level set enclosing $\alpha$ fraction of the area, of the corresponding Green function on $\U$ with source and sink at $x_{\rm{B}}$ and $x_{\rm{R}}$ respectively. Using conformal invariance of the Green function,  $\U_{(\alpha)}$ is related to $\DD_{(\alpha)}$ via a suitable conformal map, (see Figure \ref{conftrans}). Precise definitions are provided in the next subsection.  It was predicted by Propp in 2003, that for any reasonably regular domain $\U,$  and the competitive erosion chain on $\U_n$ with $k=\lfloor{\alpha |\U_n|\rfloor},$ and $\mu_{\rm{B}}, \mu_{\rm{R}}$  being delta masses at $x_{\rm{B}}$ and $x_{\rm{R}}$ respectively, the blue region should ``converge" to $\U_{(\alpha)},$ as $n$ grows large.

Our main result confirms a version of Propp's prediction. More formally, for technical reasons, we need $\mu_{\rm{B}}$ and $\mu_{\rm{R}}$ to be spread out and supported in the interior of $\U,$ instead of being point masses. In this article, we take them to be uniform measures on all lattice points inside small `blobs' of radii $\frac{\dd}{4}$  lying inside $\U$,and also at distance $\dd$ from the points $x_{\rm{B}}$ and $x_{\rm{R}}$ respectively (see Figure \ref{blob1} i.)
The main theorem then involves choosing $\dd $  arbitrarily small  which can be thought of as approximating  the delta masses at $x_{\rm{B}},x_{\rm{R}}.$ Proving Propp's prediction when the source measures are indeed point masses on the boundary is left as an interesting technical challenge. (For a more elaborate discussion on the associated difficulties see Section \ref{concl1}.)     
\subsection{Formal definitions and setup}\label{para}
In this section we collect all the necessary definitions and notations required for the statement of the main result.
For any two points $x,y \in \C,$ (also thought of as $\R^2$) we use $d(x,y)$ to denote the Euclidean distance between them.  For sets $A, B \subset \C,$ 
let $$d(A,B)=\inf_{x\in A, y \in B}d(x,y).$$
 $\DD$ will be used to denote the unit disc centered at the origin in the complex plane. $B(x,r )$ denotes the open euclidean ball of radius $r$ with center $x$.
For  a bounded, simply connected, planar domain, $\U\subset \C$ we will say ``$\U $ is smooth",  to mean that  the boundary of $\U$ i.e., $\overline \U \cap \overline{(\U^c)},$ where $\overline \U$ is the closure of $\U$, is an analytic curve (equivalently, the conformal map from $\U$ to $\DD$ has a conformal extension across the boundary, see \cite[Prop 3.1]{pom}).

\emph{Throughout the rest of the  article, all our domains will be smooth and hence we will not recall all these properties and simply speak of ``a domain".}

Henceforth for any domain $\U\subset \C,$  we use $\partial \U$  to denote the boundary of $\U,$\footnote{We also need the notion of boundaries of  subsets of vertices in finite of graphs. This will be defined later in \eqref{bdry1}.}
Also for a domain $\U$ and points $x_{\rm{B}},x_{\rm{R}} \in \partial \U,$ let,  
\begin{align}
\label{confmap1}
\phi&:\DD \rightarrow \U,\\
\nonumber
\psi&: \U  \rightarrow \DD,
\end{align}
be conformal maps, such that $\phi \circ \psi,$ $\psi \circ\phi$ are the identity maps on the respective domains and
$\psi(x_{\rm{B}}) = -i,\, \psi(x_{\rm{R}}) =  i.$ 
The existence of such maps is guaranteed by the Riemann Mapping Theorem, (see for  e.g., : \cite[Chapter 6]{ahl1}). In fact, there exists  more than one pair  $(\phi,\psi),$ since a conformal map between domains has three degrees of freedom and here we have fixed the value at only two points.  However, we choose a particular pair $(\phi,\psi),$ assumed to be fixed throughout the rest of the article.
For any $\beta \in \mathbb{R},$
define,
 \begin{equation}\label{area12}
 \U_{\ge\beta}:=\Bigl\{z\in  \U:\frac{64}{\pi}\log{\Bigl|\frac{\psi(z)-i}{\psi(z)+i}\Bigr|}\ge \beta\Bigr\}.
 \end{equation} 
Also $\U_{<\beta}$ is defined in the obvious way.  
The constant $\frac{64}{\pi}$ is not important. It falls out of some natural integrals involving the Brownian motion heat kernel appearing later in the article. 
Thus for the disc, $\DD_{\beta}$ is a region enclosed by an orthogonal circular arc (geodesic with respect to the hyperbolic metric) symmetric with respect to $i$ and $-i$ and is the level set of the Green function, as mentioned before. For a general domain $\U$, we transfer the regions via conformal maps (using conformal invariance of the Green function) and the boundaries still remain  geodesics as they are conformally invariant, (see Figure  \ref{conftrans}).
\begin{figure}[hbt]
\centering
\includegraphics[scale=.6]{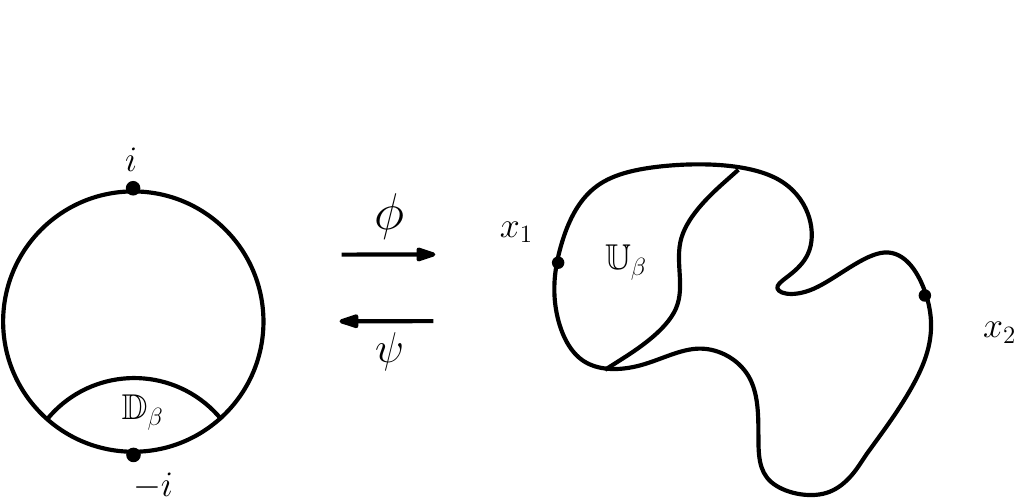}
\caption{Hyperbolic geodesics are circular arcs on the disc. They are invariant under conformal maps. However, since conformal maps are not area preserving $\DD_{(\alpha)}$ can get mapped to $\U_{(\alpha')}$ for some $\alpha'\neq \alpha $ .} 
\label{conftrans}
\end{figure}
However, for our purposes, we need an area parametrization of the regions $\U_{\ge\beta}.$ Given $\alpha\in (0,1)$, let $\beta=\beta(\alpha)$ be such that, 
\begin{equation}\label{areaparam}
{\rm{area}}(\U_{\ge\beta})=\alpha\,\, \rm{ area}(\U).
\end{equation} 
 Let 
 \begin{equation}\label{area13}
 \U_{(\alpha)}:= \U_{\ge\beta(\alpha)}.
 \end{equation}
 We now  fix $\alpha \in (0,1/2],$  throughout the rest of the article. 
Given $\U,$ we take $\U_{n} = \U \cap \frac1n \Z^2,$ as our vertex set. As the edges of our graph, we take the usual nearest-neighbor edges of $\U_n$ thought of as a subset of $\frac1n \Z^2$. However, we delete every such edge which intersects $\U^c$. Since $\U$ is assumed to be smooth, $\U_n$ will be connected for large enough $n$. See Remark \ref{graphcon123} below.

Fix $k=\lfloor{\alpha |\U_n|}\rfloor,$ and $x_{\rm{B}},x_{\rm{R}} \in \partial \U$.
The following defines the blobs, as discussed in the previous section. 
 For small enough $\dd>0,$ let $y_{\rm{B}}, y_{\rm{R}} \in \U,$ be such that for ${\rm{C}}\in \{{\rm{B}},\rm{R}\},$
 \begin{eqnarray*}
 d(x_{\rm{C}},y_{\rm{C}})&=& \dd \\
 d(y_{\rm{C}},\partial \U) &> & \frac{\dd}{2}.
  \end{eqnarray*}   
 Let $\U_{\rm{B}}=\U_{{\rm{B}},\delta}:=B(y_{\rm{B}},\frac{\dd}{4})$ and similarly let $\U_{\rm{R}}=\U_{{\rm{R}},\delta}:=B(y_{\rm{R}},\frac{\dd}{4}).$  
As discrete approximations of $\U_{\rm{B}},$  we take $$\U_{{\rm{B}},n}=B(z_{{\rm{B}},n}, \frac{\dd}{4})\cap \U_{n},$$ where $z_{{\rm{B}},n}\in \frac{1}{n}\Z^2$ is the closest lattice point to $y_{\rm{B}}$.  
Define $\mu_{\rm{B}}=\mu_{{\rm{B}},\delta,n}$ to be the uniform measure on $\U_{{\rm{B}},n}$.
We similarly define $\U_{{\rm{R}},n}$ by choosing $z_{2,n},$ and correspondingly define $\mu_{\rm{R}}$ as the uniform measure on $\U_{{\rm{R}},n}.$
\begin{figure}[hbt]
\centering
\includegraphics[scale=.8]{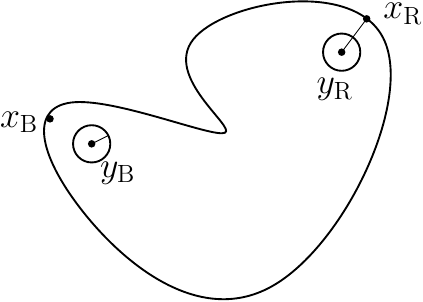}
\caption{$y_{\rm{B}}$'s and $y_{\rm{R}}$ are points at distance $\dd$ from $x_{\rm{B}}$ and $x_{\rm{R}}$ respectively. They are at distance at least $\frac{\dd}{2}$ from $\partial \U$. The blobs are discs of radius $\frac{\dd}{4}$ centered at $y_{\rm{B}}$'s and $y_{\rm{R}}$.}
\label{blobdef234}
\end{figure}
Thus the blobs are the sets $\U_{\rm{B}}$ and $\U_{\rm{R}}$ and $\mu_{\rm{B}}$, and $\mu_{\rm{R}}$ are uniform measures on their discrete approximations $\U_{{\rm{B}},n}$ and $\U_{{\rm{R}},n}$ respectively.  Note that the points $y_{\rm{B}}$ and $y_{\rm{R}}$  were just required to satisfy certain  properties, and other than that were completely arbitrary. 
\begin{remark}\label{graphcon123}
The smoothness assumption on $\U$ allows us to choose $y_{\rm{B}}$ and $y_{\rm{R}}$.  This is formally proved in Corollary \ref{ias1} (see Figure  \ref{blobdef234}).
The connectedness of $\U_n$ for large enough $n,$ follows since near the boundary $\U$ looks like a half plane locally, see \eqref{localhalf}.  
\end{remark}

For the formal statement of the theorem let 
\begin{equation} \label{goodset11} \cG_{\e}=\cG_{\e,n} := \{  S \subset \U_n \,:|S|=\lfloor{\alpha |\U_n|\rfloor},\,\, \U_{(\alpha-\e)} \cap \U_n \subset S \subset \U_{(\alpha+\e)}\cap \U_n \}. \end{equation}
Note that, because the conformal maps  $\phi$ and $\psi$ are Lipschitz (see Section \ref{commen1} i.), and $\alpha$ is fixed, for all small enough $\e,$ the boundaries of $\U_{(\alpha)}$ and $\U_{(\alpha \pm \e)}$ are $\Theta(\e)$ away from each other where the constant in $\Theta(\cdot),$ depends only on $\alpha, \U.$  

Thus $\cG_\e$ denotes the set of all configurations where each vertex in $\U_{(\alpha-\e)}$ is colored blue and each vertex in $\U \setminus \U_{(\alpha-\e)}$ is colored red. Our main result informally states that for any $\e$ as $n$ goes to infinity (the mesh size goes to zero) the equilibrium measure of the competitive erosion chain concentrates on $\cG_{\e}$ provided that the blob size $\delta$ is small enough.

 \begin{thm}\label{mainresult}(Main Result) For a domain $\U\subset\C,$ consider the competitive erosion chain on $\U_n$ with blob radii $\frac{\dd}{4}$ and $k=\lfloor{\alpha |\U_n|}\rfloor$. Then, given $\e>0,$ for $\dd< \dd_0(\e),$ there exists a positive constant $D=D(\e,\dd,\U),$ such that  for all large enough $n=2^m,$    
\begin{equation}\label{weak}
\pi(\cG_{\e})\ge 1-e^{-Dn},
\end{equation} 
where $\pi=\pi_{\dd,n}$, is the stationary measure of the chain ${\rm{CE}}(t)$.
\end{thm}

See Figures \ref{formation23} and \ref{blob1} ii .
Note that for brevity, we suppress the centers of the blobs in the notation $\pi_{\dd,n}$. 
Informally, the above theorem says the following: consider the competitive erosion chain on the discretization of any smooth domain $\U$ with points $x_{\rm{B}},x_{\rm{R}}$ on the boundary. Suppose that $\alpha$ fraction of the vertices are blue. If the blobs from which the blue and red random walks start are small enough and close enough to $x_{\rm{B}}$ and $x_{\rm{R}},$ then as the mesh size goes to zero, at stationarity, the blue region looks like the set $\U_{(\alpha)},$  in the sense that on the  `$\U_{(\alpha)}$ side' of  a small band around the boundary of $\U_{(\alpha)}$, the vertices are all blue and similarly on the `other side',  all vertices are red. As discussed above, the sets $\U_{(\alpha)}$ for various domains $\U$ can be obtained from each other via conformal maps. Thus Theorem \ref{mainresult} establishes that competitive erosion is \emph{conformally invariant}.
Note that the mesh size goes to $0$ at dyadic scales. This is for technical convenience, and the reasons are elaborated in Section \ref{commen1} iii. Implicit in the statement of Theorem~\ref{mainresult} is the claim that competitive erosion has a unique stationary distribution.  A version of this was proved for the cylinder graph in \cite[Section 2.3]{cyl}, and the same proof carries over here, owing to the smoothness assumption on $\U.$ We omit the details. 

\begin{figure}[hbt]
\centering
\includegraphics[scale=.18]{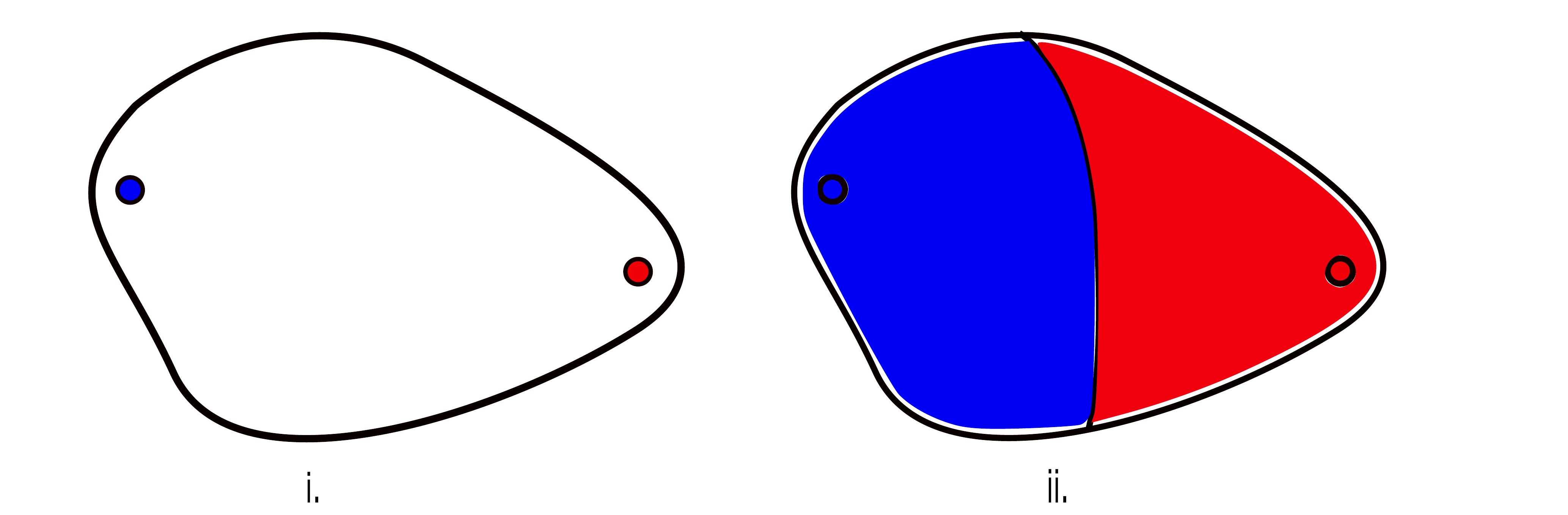}
\caption{On the left: uniform measure on the two blobs close to the points $x_{\rm{B}},x_{\rm{R}}$ act as ``smooth'' approximations of the corresponding delta masses. On the right: The red and blue territories have separated out after running the competitive erosion chain with the blue region being ``close" to $\U_{(\alpha)}$ if the blobs are small enough. }
\label{blob1}
\end{figure}

\subsubsection{Statistical version of Theorem \ref{mainresult}}\label{quant1}
Notice that one can think of Theorem \ref{mainresult} as a convergence result in terms of the Hausdorff metric.  Our proof strategy at a very high level, consists of broadly two steps: 
\begin{enumerate}
\item Prove a version of the main result for the weaker $L_{1}$ metric (see Theorem \ref{quantmainresult} below), which we call a statistical version.
\item Bootstrap to get Theorem  \ref{mainresult}, by using certain robust estimates about IDLA on $\U_n.$
\end{enumerate}
We now prepare to state the statistical version of Theorem \ref{mainresult}. 
Following our convention of coloring the vertices with blue ($\rm{B}$) and red (${\rm{R}}$) color, define, 
\begin{align}\label{statespace}
\Omega:=\Omega_n &:=\{\sigma\in \{
{\rm{B}},{\rm{R}}\}^{\U_n}, 
|\{x \in \U_n : \sigma(x)={\rm{B}}\}|=\lfloor{\alpha|\U_n|}\rfloor\}\\
\nonumber
\Omega':=\Omega'_n &:=\{\sigma\in \{{\rm{B}},{\rm{R}}\}^{\U_n}, 
|\{x \in \U_n : \sigma(x)={\rm{B}}\}|=\lfloor{\alpha|\U_n|}\rfloor+1\}.
\end{align}
Thus  the competitive erosion chain $\{{\rm{CE}}(t)\}$, can also be thought of as a Markov chain $\{\varsigma_{t}\}_{t\ge0}$, on $\Omega$ in a natural way.
Now, given $\e>0$, define $\cA_{\e}:=\cA_{\e,n}$ to be the set of all configurations $\sigma\in \Omega_n$ such that, 
\begin{eqnarray}\label{optimalconf}
|\{x\in \U_{(\alpha),n}:  \sigma(x)={\rm{R}}\}|&\le & \e n^2,\\
\nonumber
 |\{x\in \U_n \setminus \U_{(\alpha),n}:  \sigma(x)={\rm{B}}\}|& \le & \e n^2,
\end{eqnarray}
where $\U_{(\alpha),n}=\U_n \cap \U_{(\alpha)}$ ($\U_{(\alpha)}$ was defined in \eqref{area13}).
Thus  if $\e$ is small, $\cA_{\e}$ is the set of all configurations where the amount of ``dust" particles of unexpected color has small density, i.e., the set of all configurations, for which the union of all squares whose centers are blue sites, is close to $\U_{(\alpha)}$ in the $L_1$ metric, (see Figure  \ref{f.dust1}).
\begin{figure}[hbt]
\centering
\includegraphics[scale=.34]{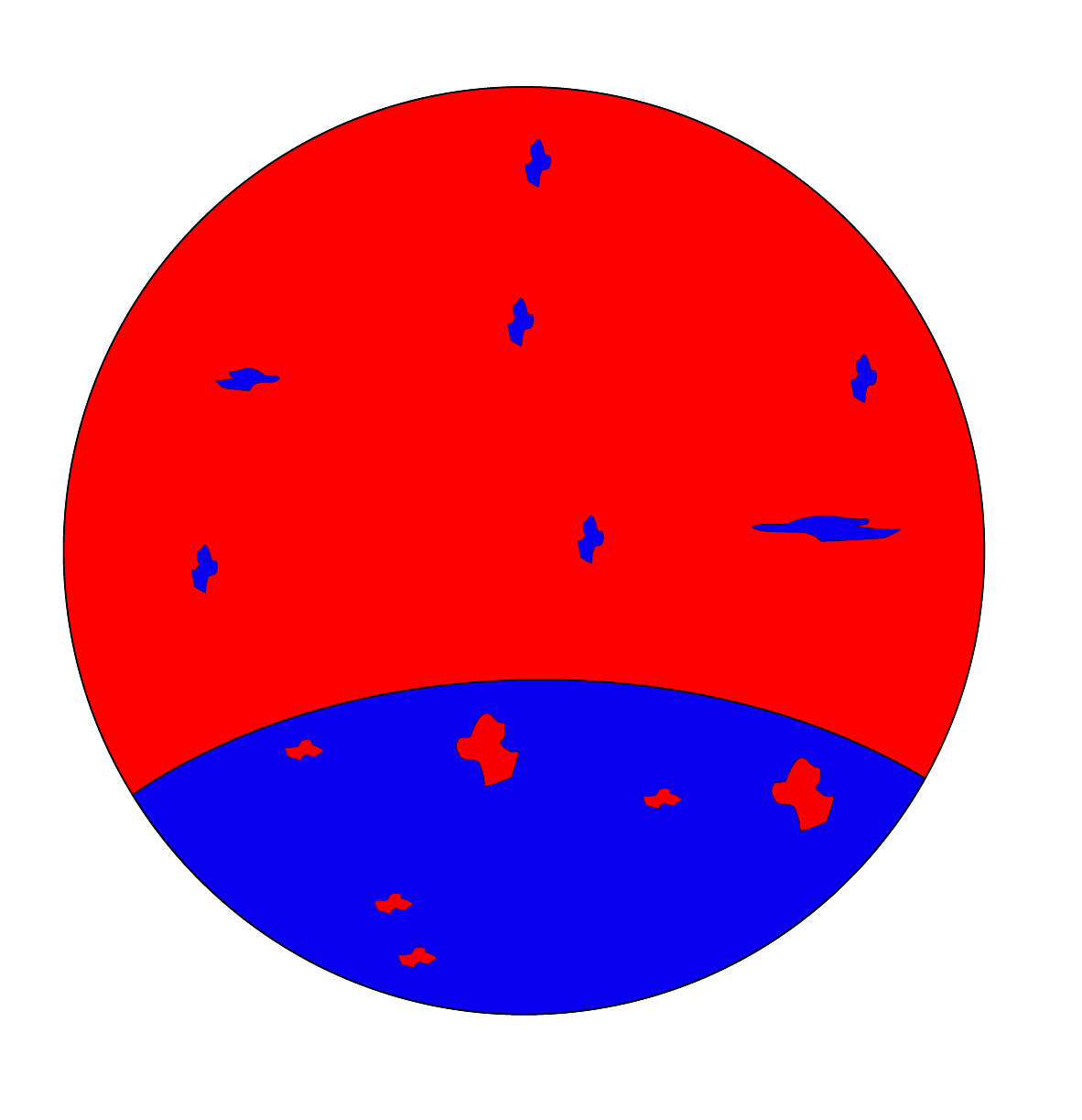}
\caption{ A typical configuration in $\cA_{\e}$ for small $\e,$ with small number of particles of unexpected color. By Theorem \ref{quantmainresult} the stationary measure of the competitive erosion chain concentrates on such configurations.}
\label{f.dust1}
\end{figure}
The next result shows that the stationary measure of $\cA_{\e}$ is large.
\begin{thm}\label{quantmainresult} Consider the setting of Theorem \ref{mainresult} and let $\e>0$. Then  there exists $\delta_0=\dd_0(\e)$ such that for all $\dd\le \dd_0$ and  $n=2^{m}>N(\dd)$, $$\pi(\cA_{\e})\ge 1-e^{{-Dn}^2},$$ where $D=D(\e,\dd, \U)>0$. 
\end{thm}

\subsection{Comparison with IDLA}\label{idla1}
Internal diffusion limited aggregation (IDLA) is a fundamental model of a 
random interface moving in a \emph{monotone} (outward) fashion. 
IDLA on a graph involves only one species with an ever-growing territory $I(t)$
where $ I(t+1)\setminus I(t)$ is the first site outside $I(t),$ visited by a simple random walk whose starting point has  a certain specified distribution.
Competitive erosion can be viewed as a symmetrized version of IDLA: whereas $I(t)$ and its complement  play asymmetric roles, ${\rm{CE}}(t)$ and its complement play symmetric roles in \eqref{dynamics}.

IDLA on a finite graph is only defined up to the finite time $t$ when $I(t)$ is the entire vertex set. For this reason,  IDLA is usually studied on an infinite graph and the theorems about IDLA are limit theorems: asymptotic shape \cite{lbg}, order of the fluctuations \cite{ag,ag2,jls}, and distributional limit of the fluctuations \cite{gff}. Scaling limit of IDLA  with multiple sources was proved in \cite{levinescaling}.  In contrast, competitive erosion on a finite graph is defined for all times, so it is natural to ask about its stationary distribution.   
To appreciate the difference in character between IDLA and competitive erosion, note that
the stationary distribution of the latter assigns tiny but positive probability to configurations that look very different to the final figure in Figure  \ref{formation23}; thus competitive erosion will occasionally form these exceptional configurations. We end this discussion by mentioning that in this article, we will  crucially use estimates for IDLA  on the graph $\U_n$ in the proof of Theorem \ref{mainresult}. 

\subsection{Remarks about Theorems \ref{mainresult}, \ref{quantmainresult} and Section \ref{para}}\label{commen1}
\begin{itemize}

\item [i.] Since $\U$ is smooth, using Schwarz reflection,  $\phi$ and hence $\psi$ can be extended conformally  across the boundary onto some neighborhoods of $\overline \DD$ and $\overline \U.$ In particular this implies that $|\phi'|$ and $|\psi'|$ are bounded away from $0$ and $\infty$ on $\U$ and $\DD$ respectively. See \cite[Prop 3.1]{pom}. This bi-Lipschitz nature of the maps will be used in several distortion estimates throughout the rest of the article.

\item [ii.] Note that the blob sources $\U_{\rm{B}}$ and $\U_{\rm{R}},$ lie entirely in the interior of $\U$ and the measures $\mu_{\rm{B}}$ and $\mu_{\rm{R}}$  should be thought of as `regular' approximations to  point masses at points $x_{\rm{B}},x_{\rm{R}}$ as $\dd \rightarrow 0.$ Also, by choice $ |\U_{{\rm{B}},n}|=|\U_{{\rm{R}},n}|.$  This will be technically convenient.

\item [iii.] Lastly, we discuss the choice of the dyadic mesh sizes in the statements of Theorems \ref{mainresult} and \ref{quantmainresult}. The technical core of the proofs of the theorems rely on a convergence theory of the discrete Green function for random walk to that of Reflected Brownian motion developed by the authors in \cite{tc}. This depends on the convergence of the random walk on $\U_n$ under proper scaling of time to Reflected Brownian motion on $\overline\U$. The proof of the latter, appears in \cite{BC} and subsequent local CLT estimates have been obtained in \cite{wtf}. However, both the above papers assume dyadic discretization in their proofs.  
Since the proof of our results rely heavily on the aforementioned convergence results, we assume the dyadic discretization in our statements as well.   We end by mentioning that it has been pointed out (via personal communication) by the authors of \cite{BC} and \cite{wtf}, that the convergence holds true even when the mesh size at the $n^{th}$ step is $\frac{1}{n}$ instead of $\frac{1}{2^{n}}$ and the latter was chosen for technical convenience.   
\end{itemize}

\section{Additional remarks and future directions}\label{concl1}

Note that Theorem \ref{mainresult} proves closeness in the Hausdorff distance of the set of blue particles to the set  $\U_{(\alpha)}$ and the set of red vertices to the set $\U \setminus \U_{(\alpha)}$, with high probability under equilibrium. This is done by first proving closeness in the weaker $L_1$ topology of the indicator of the blue set to the indicator of $\U_{(\alpha)}$ in Theorem \ref{quantmainresult}. Then we strengthen it by comparing the competitive erosion chain to IDLA.  

Moreover, it is not too hard to see that  the probability bound of $e^{-cn^2}$ in Theorem \ref{quantmainresult} is tight up to constants in the exponential scale. For the purpose of illustration assume that the domain is a square and that the red and blue blobs are located near the middle points of the bottom side and the top side respectively. Now starting from a configuration in $\cG_{\e},$ with probability at least $e^{-cn^2},$  in the next $\Theta(n^2)$ rounds of the competitive erosion chain, the blue random walks will hit the left half of the interface, while the red random walks hits the right half of the interface. Thus there would be a blue cluster growing into the red region on the left half of the domain whereas there would be a red cluster growing into the blue region on the right hand side of the domain.
Since the process runs for $\Theta(n^2)$ rounds, the number of blue particles encroaching the `red region' is $\Theta(n^2)$ and hence by definition the configuration reaches a state outside $\cA_{\e}$ for all small enough $\e.$ This implies that such configurations will have probability at least $e^{-cn^2}.$

Notice that for  Theorem \ref{mainresult} the bound is only exponential in $n$ and not $n^2$. 
A lower bound of exponential in $n^2$ for $\pi(\cG^c_{\e})$ follows from the above discussion. Another possible approach to proving  a lower bound is by considering the probability of forming thin long  red tentacles in the blue region. 
Note that in each of  the next $n$ steps of the chain the blue random walk can hit a particular point on the interface with probability  $O(\frac{1}{n^n}),$ and then build a tentacle of length $O(n)$ growing out of it, with probability  $\frac{1}{4^{n^2}}.$  Thus, this strategy which yields optimal bounds for fluctuations of IDLA on $\Z^2$ (see \cite{jls}),  does not help improve the bound from the preceding paragraph.

An important ingredient in the proof of the  bound in Theorem \ref{mainresult} is a bound on the rate of growth of the IDLA cluster on $\U_n$ (Theorem \ref{IDL1}) in a non-standard setting where some of the initial sites are occupied.  Thus, a possible approach to reducing the gap between the upper and lower bounds of $\pi(\cG^c_{\e})$ could be to use the sharp techniques to bound the growth of IDLA clusters appearing in \cite{jls} based on harmonic functions.

The next remark concerns the mixing time of the competitive erosion chain. The proof of Theorem \ref{quantmainresult} in fact says:  the chain reaches the set $\cA_{\e}$ starting from any configuration in $O(n^2)$ time (see Lemma \ref{hitting}). Clearly this is optimal since if one starts from a configuration where a significant fraction of the vertices in $\U_{(\alpha),n}$ are red, then it will indeed take at least $\Theta(n^2)$ rounds of the chain to make all of those vertices blue. Thus in some `macroscopic' sense the mixing time of the chain is $\Theta(n^2).$ Finding the order of the actual mixing time of the  chain is an interesting question and is left as an open problem.

Another natural technical challenge would be to prove a version of Theorem \ref{mainresult} in the case where the sources are points on the boundary instead of blobs with a nontrivial interior. This is the setting in which Propp originally made his predictions.  The main issue lies in understanding the convergence of the Green function for such source distributions as the mesh size goes to $0,$ since the random walk started from the source would typically visit the source $\log n$ times where $\frac{1}{n}$ is the grid size, while for Reflected Brownian motion one has to pass to the notion of local time. When the sources have positive measure one does not encounter this technical difficulty.
Also, in principle one can study the same process when the starting distributions are certain specified distributions on $\U_n$ and not just point masses.  The interfaces in that case would be the level curves of the appropriate Green function; see \cite[Sec 1.3]{cyl} for further discussion. 
 We list below other possible directions for further research.

\begin{enumerate}
\item 
\textbf{Fluctuations}.
A natural next step would be to find the order of magnitude of  fluctuations of the interface. The corresponding results for internal DLA have been proved in \cite{jls,gff} where the authors prove $O(\log(n))$ and $O(1)$ bounds on the maximum and typical deviations along with explicit description of the scaling limits.
\item
\textbf{General Domains and Higher dimensions}.  The main results of this article show that  the interface of the blue and red territories stabilize along a level curve of the corresponding dipole Green function which is harmonic except at the sources and has  Neumann boundary condition. Some form of this  fact should hold true in a general class of domains which need not even be simply connected or planar. For the cylinder graph this was shown in \cite{cyl}.
Proof of Theorem \ref{mainresult} is rather delicate, but the proof of Theorem \ref{quantmainresult} is robust and  should be generalizable to such settings. The main ingredients involved include a convergence theory for the Green function for random walk as the mesh size becomes smaller and an understanding of the Green function in the continuum. These programs in the setting of this paper were carried out by the authors in \cite{tc} and  the  generalizations are natural further directions to pursue. 
\item
\textbf{Multiple particles}. Another natural extension is when there are more than two kinds of particles. One concrete definition of the model in the setting of three colors which has the property of conservation of mass is the following: Let the colors be red, blue, and green and pick three points $x_1,x_2$ and $x_3$ on the boundary of the domain. Red, blue and green particles will be emitted from $x_1,x_2$ and $x_3$ respectively.  Start by emitting a  red particle which walks till it encounters a blue or a green particle and occupies its position by removing it. In the next round the particle emitted has the same color as the particle removed in the first round. Subsequently in every round, the particle emitted  is of the same color as the particle which was replaced in the last round to ensure conservation of mass.  There is no conjecture for the form of the separating curves currently. Any progress along this direction would be rather interesting and thus proving a theorem analogous to Theorem \ref{mainresult} for \emph{Competitive Erosion} in the above setting remains an intriguing challenge. 
\end{enumerate}

\section{Main ideas and organization of the article}\label{sop}
To prove  Theorem \ref{quantmainresult}, i.e., that the stationary distribution concentrates on the set $\cA_{\e}$, we identify a \emph{Lyapunov function} ${\rm{W}}(\cdot)$ on the state space. That is a function which attains its global maximum in $\cA_{\e},$ and increases in expectation in one step of the process $\{\varsigma_t\}$ when starting outside $\cA_{\e}$, i.e., that if $\sigma \notin \cA_{\e},$ then,
	\begin{equation}\label{e.drift}
	 \E_{\sigma}({\rm{W}}(\varsigma_1) - {\rm{W}}(\varsigma_0)) \geq a >0,
	 \end{equation}
 (where $\E_{\sigma}(\cdot)$ denotes the expectation with respect to the competitive erosion  process $\{\varsigma_t\}$, started from $\sigma$).
For more on Lyapunov functions see \cite{fmm}. 
To construct such a function ${\rm{W}}(\cdot),$ we proceed by defining the following discrete Green function: for any $x\in \U_n,$
 \begin{equation}\label{olddef}
 G_n(x)=\frac{2n^2}{|\U_{{\rm{B}},n}|}\int_{0}^{\infty}[\mathbb{P}_x(X(t)\in \U_{{\rm{B}},n})-\mathbb{P}_x(X(t)\in \U_{{\rm{R}},n})]\,dt,\footnote{For technical reasons, in the formal definition, we will introduce a centering constant (see \eqref{dgf}).}
 \end{equation}
 where $\mathbb{P}_x(\cdot)$ is the measure induced by the  continuous time random walk $X(\cdot),$ on the graph $\U_n$ started from $x,$ (the waiting times of the random walk are mentioned explicitly in the next section). Thus, the above function is the expected difference, in the amount of time the random walk spends in $\U_{{\rm{B}},n}$ and in $\U_{{\rm{R}},n}$ respectively.
 Now, for any $\sigma\in \Omega \cup \Omega',$ we define the weight function, 
\begin{equation}\label{wf}
 {\rm{W}}(\sigma): = \sum_{x \in S_{\rm{B}}(\sigma)} G_n(x),
\end{equation}
where for $\sigma \in \Omega \cup \Omega',$
\begin{align}\label{region}
S_{\rm{B}}(\sigma)&:=\{x \in \U_n: \sigma(x)={\rm{B}}\},\\
\nonumber
S_{\rm{R}}(\sigma)&:=\{x \in \U_n: \sigma(x)={\rm{R}}\}.
\end{align}
When the underlying $\sigma$ is clear from context, we will often denote the above sets by $S_{\rm{B}}$ and $S_{\rm{R}}$ respectively.
Now, recall the map $\psi$ from \eqref{confmap1}.
Our proofs rely on the key technical  result (Theorem \ref{convergence}) which states that: 
up to translation by a constant, $G_{n}(x)$ is close to the function 
$\frac{64}{\pi}\log{\bigl|\frac{\psi(x)-i}{\psi(x)+i}\bigr|},$  if the blob sizes are small and $n$ is large.
The proof of the above, involves a convergence theory for the discrete Green function with  Neumann boundary conditions to its continuous counterpart, and appears in a separate paper by the authors \cite{tc}. The proofs rely on local central limit theorems for convergence of random walk on $\U_n$ to Reflected Brownian motion on $\U.$   Some of these estimates were obtained very recently in \cite{BC,wtf,thes}.
Note that from the above discussion it becomes clear  why ${\rm{W}}(\sigma)$ is maximized if all the blue vertices are in $\U_{(\alpha),n}$ (see \eqref{area13}) and hence this  justifies the construction of the function ${\rm{W}}(\cdot).$  A related discussion about interfaces of the blue and the red territory appears in a  paper  studying competitive erosion on the simpler cylinder graph, by the authors along with Lionel Levine and James Propp \cite[Sec 4.3]{cyl}.

The proof of  \eqref{e.drift} is the most technically challenging part of this paper. The proofs are based on considering the graph $\U_n$ as an electrical network. At a high level, the proof proceeds by interpreting  $\E_{\sigma}({\rm{W}}(\varsigma_1) - {\rm{W}}(\varsigma_0))$,  as the potential drop when certain vertices of the network are identified/glued.  Rayleigh's monotonicity principle says that the voltage drop is always non-negative. However, for our purposes we need to make the monotonicity principle quantitative to prove a strictly positive lower bound. There are several geometric case considerations along the way to make the argument work. This constitutes a bulk of the paper and the result is stated as  Theorem \ref{negdrift123}.  
 
The proof of Theorem \ref{negdrift123} is broken down into several parts. Sections \ref{efn} and \ref{build100} are devoted to this. Section \ref{roadmap} provides a detailed roadmap for the proof and what the various ingredients and steps of the proof are.
Assuming Theorem \ref{negdrift123}, we use Azuma's inequality to argue that the process ${\rm{W}}(\varsigma_t),$ which more or less behaves like a sub-martingale, attains a state near its maximum in time $O(n^2)$ and takes an exponential in $n^2$ time to drift away from the maximum. Thus on average, it spends all but an exponentially tiny fraction of time, in a neighborhood of its maximum.
These ingredients together with a general estimate relating local times to stationary measure (Lemma~\ref{hitstation}) establish  Theorem~\ref{quantmainresult}. This is done in Section \ref{mht}. 
In Section \ref{TP} we state some properties of the Green function including asymptotic conformal invariance. 

To prove Theorem \ref{mainresult} from Theorem \ref{quantmainresult} we show that starting erosion from a configuration $\sigma\in \cA_{\e},$ the remaining dust particles (red particles in $\U_{(\alpha-\e ),n}$) get wiped off quickly. This part of the proof is delicate and uses IDLA estimates on $\U_n$ starting from the random environment $\sigma.$  Proofs of such estimates constitute the remaining technical challenge and is presented in Section \ref{IDLArobust}. The proof of the main result appears in Section \ref{pmr}. 
We finish the proofs of some lemmas stated and used in Section \ref{mht} in Appendix \ref{pol}.
Some basic geometric facts used throughout the article are proved in Appendix \ref{RWE}. 

\subsection{More notations, conventions and remarks}\label{techass}

We already introduced a list of notations and definitions in Section \ref{para} needed for the statements of the main results. In this section we summarize the notation used in the remainder of the article. We also introduce some conventions that will be followed in the sequel.

Even though we look at the Markov chain $\{\varsigma_t\}$ at integer times, 
for later purposes we need to keep track of the full process: i.e., we consider, 
\begin{equation}\label{fullpro1}
\{\varsigma_{t}\},\,\,\,\, t=0,1/2,1,3/2,2,\ldots,
\end{equation}
where $\varsigma_t \in \Omega$  or $\varsigma_t \in \Omega',$ depending on whether $t$ is an integer or a half integer. 
Clearly, a single integer time step of the chain consists of a step from $\Omega$ to $\Omega'$ followed by a step from $\Omega'$ back to $\Omega$.
Throughout the article, by random walk on $\U_n$, we will mean the 
continuous time random  walk   with  waiting times given by exponential random variables with mean $\frac{1}{2n^2}$, unless specifically mentioned otherwise. 
This time change is done to  ensure that the random walk heat kernel converges to that of Reflected Brownian motion; a fact, which will be used heavily throughout the article.
However, sometimes  considering discrete time random walk will be more convenient. It will be explicitly mentioned when we do so, and hence should not cause any confusion.

For any space, and a subset $A,$  we denote by $\mathbf{1}(A)(\cdot),$  the indicator function of the set $A,$ defined on the underlying space. 
  Since we  deal with several processes (the competitive erosion chain, random walk on $\U_n,$ Reflected Brownian motion on $\U$ etc), to avoid introducing too much notation, we use $\P(\cdot)$ and $\E(\cdot)$ to denote the corresponding probability measure and expectation associated  with  the underlying process, suppressing dependence on the process, in the notation. Following the same conventions, for any process and a subset $A$ of the corresponding state space, 
$\tau(A)$ will denote the hitting time of that set. 
The above convention poses no danger of confusion as the setting will always be clear from context, or explicitly mentioned.
Also for any set $A$ as above, and any  time $t,$  we  use,  
\begin{equation}\label{stopnot} 
\tau(A,t)
\end{equation} to denote the first hitting time of $A$ by the corresponding process, at or after time $t.$ Thus $\tau(A)=\tau(A,0).$
This notation will come in handy, since we  deal with several hitting times one after the other, for  e.g., : for two sets $A_1$ and $A_2$, the event $\{\tau (A_1)\le K_1\} \cap \{\tau(A_2,\tau(A_1))\le K_2\}$ will be used to denote the intersection of the events that $A_1$ is hit within the first $K_1$ steps and then from $\tau(A_1)$ onwards, the set $A_2$ is hit within the next  $K_2$ steps. 

Recall that the parameter $\delta$ was used to define the size of the blobs. We suppress the $\delta$ dependence in all our notations
for brevity. This will not cause any confusion, since we will not be using $\delta$ to denote any other quantity.
For any $x\in \U$ and $r>0,$ we  denote $B(x,r)\cap \frac{1}{n}\Z^2,$ by $B_n(x,r),$ where as mentioned before, $B(x,r)$ denotes the open euclidean ball of radius $r$ centered at $x.$ 
We will often use the same letter for a constant whose value may change from line to line. Also, $O(\cdot)$ and $\Theta(\cdot)$ are used to denote their usual meaning.  For easy reference, below we summarize some of the notations  (already defined or to be defined later), that will be used frequently throughout the article.

\begin{longtable}{@{}lll@{}}

  {\bf Notation}
  & {\bf \quad Defined in Section}
  & {\quad\quad\bf Short Informal Description}\\
\hline  
  $\U\, (\D)$
  & \quad\quad\quad\quad\quad \ref{para}
  &\quad  \quad General smooth domain (Unit Disc).\\

  $\U_n \,(\D_n)$
  & \quad\quad\quad\quad\quad \ref{para}
  &\quad  \quad  Graph obtained by discretizing the above\\ &&\quad  \quad domains.\\

  ${\U}_{\ge\beta}$ (${\U}_{\ge\beta,n}$)
  & \quad\quad\quad\quad\quad \ref{para} (\ref{glos1})
  &\quad  \quad  Region enclosed by Hyperbolic Geodesics, \\&&\quad\quad (Lattice version).\\

  ${\U}_{(\alpha)}$ (${\U}_{(\alpha),n}$)
  & \quad\quad\quad\quad\quad \ref{para} (\ref{quant1})
  &\quad \quad Area parametrization of $\U_{\ge\beta}$, \\&&\quad\quad (Lattice version). \\

    $x_{\rm{C}} \,(y_{\rm{C}})$
  & \quad\quad\quad\quad\quad \ref{para}
  &\quad \quad Source for color $\rm{C},$ (center of blob).   \\
 
  $\U_{\rm{C}}\, (\U_{{\rm{C}},n})$
  & \quad\quad\quad\quad\quad \ref{para}
  &\quad \quad Continuous (Discrete) blobs.\\

  $\mu_{\rm{C}}$
  & \quad\quad\quad\quad\quad \ref{para}
  &\quad \quad Uniform measure on $\U_{{\rm{C}},n}$.\\

  $B(x,r) \,(B_n(x,r))$
  & \quad\quad\quad\quad\quad \ref{para} (\ref{techass})
  &\quad \quad Euclidean ball (Lattice Ball).\\

   $\mathcal{G}_{\e}$    & \quad\quad\quad\quad\quad \ref{para} &\quad \quad  Set of all configuration where all the sites \\&&\quad\quad outside an  $\e$ band around   ${\rm{Geo}}_{\beta(\alpha)},$\\&&\quad\quad are of the expected color. \\

   $\{\varsigma_t\}$
  & \quad\quad\quad\quad\quad \ref{quant1}
  &\quad \quad Markov Chain on the space of colorings.\\

  $\cA_{\e}$    &  \quad\quad\quad\quad\quad\ref{quant1} &\quad \quad  Set of all configurations which has at most \\&&\quad\quad $\e n^2$ vertices with unexpected color in the \\&&\quad\quad red and blue regions.\\
  
   $S_{{\rm{C}}}(\sigma)$ & \quad\quad\quad\quad\quad \ref{sop} &\quad \quad Subset of vertices of color ${\rm{C}}$. \\

   ${\rm{Geo}}_{\beta}$
  & \quad\quad\quad\quad\quad \ref{mht}
  &\quad \quad Hyperbolic Geodesics.\\

  $G_n(\cdot)$ &   \quad\quad\quad\quad\quad\ref{mht} &\quad \quad Discrete Green function with\\&&\quad \quad Neumann boundary conditions.\\

  $\Gamma_{\e}$  & \quad\quad\quad\quad\quad \ref{glos1} &\quad\quad Set of all configurations where the  weight \\&&\quad\quad function is within $2\e n^2$ of its maximum.\\

  $\U^{(i),n}$
  & \quad\quad\quad\quad\quad \ref{glos1}
  & \quad\quad The domain $\U$ is subdivided into $O(n)$ shells\\&&\quad \quad of width $O(\frac{1}{n})$ each. \\

  $\Omega_{(\e)}$  & \quad\quad\quad\quad\quad\ref{glos1} &\quad \quad Set of all configurations where at most $\e n$\\&&\quad \quad different shells (defined above)  have sites \\&&\quad \quad of  unexpected  color.   \\

  $G_*(\cdot)$ &   \quad\quad\quad\quad\quad\ref{cinvariance} &\quad \quad Continuous Green function with\\&&\quad \quad Neumann boundary conditions. \\

  $R_{{\rm{C}}}$ & \quad\quad\quad\quad\quad \ref{efn} &\quad \quad Subset of color ${\rm{C}}$ connected by a \\ &&\quad\quad monochromatic path to $\U_{{\rm{C}},n}.$ \\

 $\sT_{({\rm{C}},\sigma)}$ &   \quad\quad\quad\quad\quad\ref{efn} &\quad \quad Exit time of $R_{{\rm{C}}}(\sigma)$.\\

$\mathring R_{{\rm{C}}}(\sigma)$ &  \quad\quad\quad\quad\quad\ref{geomdef} &\quad \quad   `Simply connected version' of $R_{{\rm{C}}}$.\\

 $\partial_{in} \text{ and } \partial_{out}$ &  \quad\quad\quad\quad\quad \ref{geomdef} &\quad \quad Inner and outer vertex boundary of subsets\\&&\quad \quad of vertices in $\U_n$.\\
 
   ${\rm{Ind}}_{\rm{C}}(\cdot)$& \quad\quad\quad\quad\quad \ref{pota100} &\quad\quad  Normalized indicator on $\U_{{\rm{C}},n}.$\\

${\rm{Ind}}(\cdot)$ &\quad\quad\quad\quad\quad \ref{pota100} &\quad\quad Difference of the indicators corresponding \\&& \quad \quad to the two colors.\\

  $G_{{\rm{C}},n}$ &   \quad\quad\quad\quad\quad\ref{stopdef1002} &\quad \quad Green function for random walk stopped\\&&\quad \quad on exiting $R_{\rm{C}}$. \\
   $G^*_{{\rm{C}},n}$ &   \quad\quad\quad\quad\quad\ref{weak102} &\quad \quad A variant of the above where the walk is also\\&&\quad \quad  killed  on hitting the inner boundary of $\mathring R_{{\rm{C}}}$.\\
 
 $\cE$ &   \quad\quad\quad\quad\quad\ref{weak102} &\quad \quad Energy of the flow induced by $G_n.$ \\
 $\cE_{{\rm{C}}}$ ($\cE^*_{{\rm{C}}}$)&   \quad\quad\quad\quad\quad\ref{weak102} &\quad \quad Energy of the flows corresponding to  \\
 && \quad \quad $G_{{\rm{C}},n}$ and $G^*_{{\rm{C}},n}$ respectively. \\
   
\hline
\end{longtable}

\section{Proof of Theorem \ref{quantmainresult}}\label{mht}

The goal of this section is to prove Theorem \ref{quantmainresult},  assuming some results whose proofs constitutes the technical bulk of this paper and are provided later. 

\begin{enumerate}
\item We first state the key technical result, Theorem \ref{negdrift123} establishing \eqref{e.drift}, and then some bounds for the weight function  and the Green function to be used in the proof of Theorem \ref{quantmainresult}. 
\item We then state and prove certain hitting time estimates for the competitive erosion chain in Subsection \ref{ht1000}, which are consequences of the above and a sub-martingale concentration inequality (Lemma \ref{lemm:azuma1}).  
\item  The proof of Theorem \ref{quantmainresult} then follows from a result relating hitting times of sets,  and stationary measure (see Lemma \ref{hitstation}).
\end{enumerate}

For technical convenience, we introduce a centering constant in the expression for the discrete Green function in \eqref{olddef} and hence work with the following slightly altered definition :
 For any $x\in \U_n,$ 
 \begin{equation}\label{dgf}
 G_n(x)=\frac{2n^2}{|\U_{{\rm{B}},n}|}\int_{0}^{\infty}\mathbb{P}_x(X(t)\in \U_{{\rm{B}},n})-\mathbb{P}_x(X(t)\in \U_{{\rm{R}},n})\,dt-c,
 \end{equation}
 where $c=c(\dd)$ is a centering constant we introduce to ensure that if $\dd$ is small and $n$ is large then the Green function is close to the function $
\frac{64}{\pi} \log\left|\frac{\psi(x)-i}{\psi(x)+i}\right|,$
(see Section \ref{TP}  Theorem \ref{convergence} and Lemma \ref{unifcon}). The centering constant  $c(\dd),$  can be written in terms of the Reflected  Brownian motion heat kernel on $\U$\footnote{If $\cB_t$ denotes Reflected Brownian motion, then $c$ is chosen such that the integral of the continuous Green function,
\begin{equation}\label{discdef}
G(z)=\frac{2}{\rm{ area}(\U_{\rm{B}})}\int_{0}^{\infty}[\P_{z}(\cB_{t} \in \U_{\rm{B}})-\P_{z}(\cB_{t} \in \U_{\rm{R}})]\,dt-c,
\end{equation} along $\partial \U$ is $0$.  For more details see \cite{tc}.},
but the explicit form of $c$ is not important in any of the arguments.
Also, let us define the  one parameter family of  hyperbolic geodesics  for any $\U,$ which form the boundary of the sets $\U_{\ge\beta}$ defined in \eqref{area12}: For any $\beta \in \mathbb{R},$ 
\begin{equation}\label{def2}
{\rm{Geo}}_{\beta}(\U):=\{z\in \U: \frac{64}{\pi}\log\left|\frac{\psi(z)-i}{\psi(z)+i}\right|=\beta \}.
\end{equation}
We will often suppress the dependence on $\U$ and denote the above, by ${\rm{Geo}}_{\beta}$, since the underlying domain will be clear from context. For more details regarding the hyperbolic metric see \cite{ahl1}.
\subsection{The main technical theorem and properties of the weight function.}\label{glos1} We will now state the key Theorem \ref{negdrift123}.
We start by defining the following `good set'. 
First, we divide our domain $\U$ according to the geodesics defined in \eqref{def2}.
Let $$a_{-k_1}>\ldots >a_{-1}>a_0>a_1>a_2>\ldots >a_{k_2},$$  be the maximal set of real numbers such that $d({\rm{Geo}}_{a_{i}}, {\rm{Geo}}_{a_{i+1}}) =\frac{100}{n},$ for all $i=-k_1,\ldots k_2-1$ and $a_0=\beta(\alpha)$ (see \eqref{area13}). 
Note that $k_1$ and $k_2$ are functions of the domain $\U,\beta(\alpha), n$ and that they grow linearly in $n.$ Moreover, since away from the points $x_{\rm{B}},x_{\rm{R}}$ on the boundary,  $\log \left|\frac{\psi(z)-i}{\psi(z)+i}\right|$ is a bounded function, with bounded non-zero derivative, as long as $d({\rm{Geo}}_{a_i},\{x_{\rm{B}},x_{\rm{R}}\}) \ge c$ for some $c>0$ (independent of $n$), 
\begin{equation}\label{order6}
|a_i-a_{i-1}|=\Theta (\frac{1}{n}),
\end{equation}
 where the constant in the $\Theta(\cdot)$ notation depend on $c$ and $\U$.  Let,
\begin{equation}\label{newnot1}
\U^{(i),n}:= \U_{\ge a_{i+1},n}\cap \U_{<a_{i},n} \text { (Figure \ref{f.dust10}.)}
\end{equation}
where for any $a\in \R,$  we let  $\U_{\ge a,n}:=\U_{\ge a}\cap \U_n$ and $\U_{<a,n}:=\U_{< a}\cap \U_n$ ($\U_{\ge a}$ and  $\U_{< a}$ were defined in \eqref{area12}).
Now given $\e>0$, define, 
\begin{equation}\label{optimalcnf1234}
\Omega_{(\e)}:=\Omega_{(\e),n}, 
\end{equation}
to be the set of all configurations $\sigma\in \Omega,$ such that 
there exists at most $\e n$ many negative $i's$ such that $\U^{(i),n}\cap \rm{S}_{\rm{R}} \neq \emptyset,$ (see \eqref{region} for definition of $\rm{S}_{\rm{R}}$)   and at most $\e n$ many positive $i's$ such that $\U^{(i),n}\cap \rm{S}_{\rm{B}}\neq \emptyset.$ Thus in words, $\Omega_{(\e)}$ consists of  configurations in $\Omega$ where at most $2\e n$ many shells $\U^{(i),n}$ contain vertices of  unexpected color.
\begin{figure}[hbt]
\centering
\includegraphics[scale=.64]{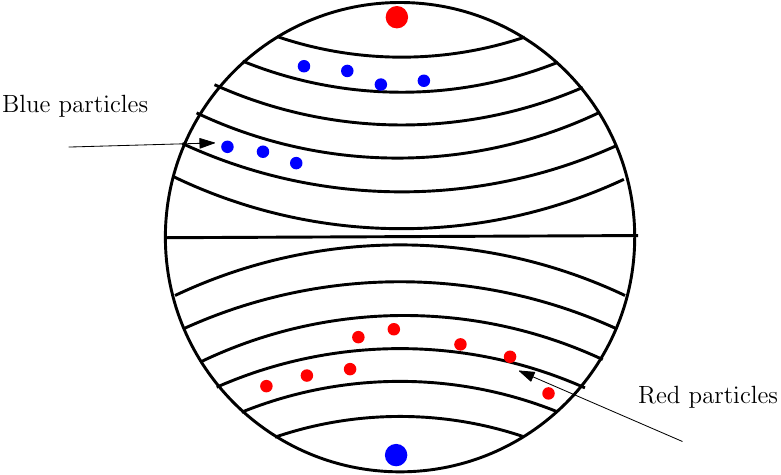}
\caption{The shells on the disc $\D,$ bounded by the hyperbolic geodesics are the $\D^{(i),n}$ 's for various values of $i.$ The figure illustrates a typical configuration in $\Omega_{(\e)}$ for small $\e,$ where only a few shells contain particles of unexpected color. }
\label{f.dust10}
\end{figure}
With the above definitions, we now state the key technical result in this paper establishing \eqref{e.drift}. 
\begin{thm}\label{negdrift123} In the setting of Theorem \ref{quantmainresult}, given $\e>0$, there exists  $a=a(\e) >0$ such that for all small enough $\dd$, and all  $n=2^{m}$ large enough,  (depending on $\dd$)
if $\sigma \notin \Omega_{(\e)}$  then, 
$$\E_{\sigma}({\rm{W}}(\varsigma_1)-{\rm{W}}(\varsigma_0))>a,$$
where $\{\varsigma_t\}$ is the competitive erosion chain.
\end{thm}
This shows that the weight function has a positive drift, if the configuration is outside the set $\Omega_{(\e)}$.
As mentioned before, the proof of the above is quite involved and  presented later (see  Section \ref{roadmap}).  Next, for a fixed domain $\U,$ we state a few useful properties of the Green function (see \eqref{dgf}) and the weight function ${\rm{W}}(\cdot)$,  whose proofs are deferred to  the Appendix.
The following result proves a uniform upper bound of the weight function independent of $\dd$. 
\begin{lem}\label{unibounded}
There exists a constant $D=D(\U)>0$ such that for all $\dd$ small enough, and $n=2^{m}>N(\dd)$, for all $\sigma \in \Omega,$
$|{\rm{W}}(\sigma)|\le  Dn^2.$
\end{lem}
For the next result which shows that the weight function is uniformly close to its maximum value, on the set $\cA_{\e}.$ Let, 
\begin{equation}\label{opt122}
{\rm{W}}_{\max}:=\sup_{\sigma \in \Omega} {\rm{W}}(\sigma).
\end{equation}
\begin{lem}\label{closelemm} For $\e>0$, there exists positive  constants $\dd_0(\e) ,\zeta(\e)$ such that for $\dd\le \dd_0$ and all large enough $n=2^{m}>N(\dd),$
\begin{equation}\label{close}
\inf_{\sigma\in \cA_{\e}}{\rm{W}}_{\sigma}\ge {\rm{W}}_{\max}-\zeta n^2,
\end{equation}
where $\zeta$ goes to $0$ with $\e$. 
\end{lem}
The next lemma reflects the `logarithmic singularity' of the Green function by determining the rate of blow up of the latter, as $\dd$ goes to $0$.   
\begin{lem}\label{maxbnd1} There exists a constant $C>0,$  such that for all small enough $\dd,$ and all large enough $n=2^m>N_0(\delta),$
\begin{equation}\label{logorder}
C^{-1} |\log(\dd)| \le \sup_{z\in\U_n}|G_{n}(z)|\le C |\log(\dd)|.
\end{equation}
\end{lem}

Given a number $\e>0,$ let,
\begin{equation}\label{def1}
\Gamma_{\e}=\Gamma_{\e,n}:=\{\sigma\,\in \Omega: {\rm{W}}({\sigma})\ge {\rm{W}}_{\max}-2\e n^2\}.
\end{equation}
\begin{lem}\label{rem2} Given $\e>0,$  there exists $0< \e_2<\e < \e_1$,  such that for all $\dd<\dd_0(\e),$ and $n=2^m\ge N(\dd),$ $$ \Gamma_{\e_2}\subset \cA_{\e}\subset \Gamma_{\e_1}.$$ 
Moreover, $\cG_{\e}\subset \Omega_{(c\e)},$ for some $c=c(\U)$, and there exists $\e_3>0$, such that $\Omega_{(\e_3)} \subset 
\cA_{\e}.$ Also, $\e_1,\e_2,\e_3$ all converge to $0$ as $\e$ goes to $0.$ 
\end{lem}

Thus, the above lemma strengthens Lemma \ref{closelemm}. It says that not only does the weight function stay close to its maximum on $\cA_{\e}$, but $\cA_{\e}$ is roughly a level set of the weight function, i.e., the actual level sets $\Gamma_{\e}$'s can be approximated by the  sets $\cA_{\e}$'s and vice versa.
The proofs of Lemmas \ref{unibounded}, \ref{closelemm}, \ref{maxbnd1}, and  \ref{rem2} are provided in Appendix \ref{pol}. 

\subsection{Hitting time estimates}\label{ht1000}
For any arbitrary  Markov chain $\chi(\cdot)$ on the state space $\cM$ and for a particular subset $A\subset \cM$, a general approach to showing that the stationary measure of $A$ is large, consists of the following two steps:
 \begin{enumerate} 
 \item Show that starting from any state in $\cM$, the chain $\chi(\cdot)$ hits  a subset of the set $A$ (thought of as the set of all states in $A$ which are away from the boundary), quickly.
 \item Show that once such a subset has been reached, the process $\chi(\cdot)$ stays inside $A$ for a long time. 
 \end{enumerate} 
The following result formalizes the above by choosing $A_1\subset A$ as the interior and $A_2=\cM \setminus A$. 
\begin{lem}\cite[Prop 1.4]{hit12}\label{hitstation} 
Let $\chi(\cdot)$ be an irreducible Markov chain on a finite state space $\mathcal{M}$. Suppose $A_1, A_2 \subset \mathcal{M}$. Let $t_1,t_2$ be the maximum expected hitting time of $A_1$ and $A_2$ respectively for $\chi(\cdot),$ i.e.,
\begin{align}
\label{hyp1}
t_1 &= \max_{x \in A_2}\E_{x}(\tau(A_1)), \\
\label{hyp2}
t_2 &= \min_{x \in A_1}\E_{x}(\tau(A_2)).
\end{align}
Then, $$\nu(A_2)\le \frac{t_1}{t_1+t_2},$$
where $\nu$ is the stationary distribution of the Markov chain $\chi(\cdot)$ on $\mathcal{M}.$
 \end{lem}
The next result proves a uniform bound on the hitting time of the set $\cA_{\e},$ for the competitive erosion chain, and is the first step, in carrying out the above approach for the proof of Theorem \ref{quantmainresult}. However for later applications (see Lemma \ref{linehit2}), we work with sets of the form $\cA_{\e_1}\cap \Omega_{(\e_2)},$ instead of just $\cA_{\e}$.
\begin{lem}\label{hitting} Given $\e_1>0$ and $\e_2>0 $, there exists $\dd_0(\e_1,\e_2),$ such that for all $\dd\le \dd_0$, there exists $C,d>0,$ such that for all  $n=2^{m}>N(\dd),$ for all  $\sigma\in \Omega,$ 
\begin{align}
\label{ht1}
 \mathbb{P}_{\sigma}(\tau(\cA_{\e_1}\cap \Omega_{(\e_2)}) \ge Cn^2) & \le  e^{-{d}n^2},\\
\nonumber
 \E_{\sigma}\left(\tau(\cA_{\e_1}\cap \Omega_{(\e_2)})\right)&\le Cn^{2},
\end{align} 
where $\dd$ is the parameter appearing in Section \ref{para}.
 \end{lem}
In the next result, we show that once the process $\{\varsigma_{t}\}$ has hit $\cA_{\e},$ it tends to stay ``close" to the set for a long time. By Lemma \ref{rem2}, this is implied by the next result.
\begin{lem}\label{goodbad} Given $\e_1>0,$ there exists $\dd_{0}(\e_1),$ such that for all $\dd\le \dd_0,$  there exists  $d=d(\e_1,\dd)>0,$ such that for all large enough $n=2^{m}>N(\dd),$ and any $\sigma\in \Omega,$ 
 $$\inf_{\sigma\in \Gamma_{\e_1}}\mathbb{P}_{\sigma}(\tau(\Omega \setminus \Gamma_{2\e_1})>e^{dn^2})\ge 1-e^{-dn^2}.$$
 \end{lem}
The above results follow from a general lemma we state next, about hitting times for sub-martingales. This is a consequence of the Azuma-Hoeffding concentration estimates for 
sub-\\
martingales. The statement of the lemma has a few parameters and could be difficult to parse. However, for subsequent applications it would  be helpful as it would allow us to  plug in various values for the parameters. 
Let $\{\chi(t)\}_{t\ge 0}$ be a discrete time stochastic process taking values in an abstract set $\mathcal{D}.$ Also, let $g:\mathcal{D} \rightarrow \R$ be a real valued function and $\mathcal{F}_t$ be the filtration generated by the process $\chi(\cdot)$ up to time $t$. 
\begin{lem}\cite[Lemma 10]{cyl}\label{lemm:azuma1}
Let $a_1,a_2,a_3 >0$. Suppose,   
\begin{eqnarray}\label{hyp11}
|g(x)| & \le & a_1 \mbox{ for all } x  \in \mathcal{D},\\
\label{hyp22}
|g(\chi(t))-g(\chi(t-1))| & \le & a_2 \mbox{ for all } t.
\end{eqnarray}
Also, suppose that $\cB \subset \mathcal{D}$  is such that for any time $t,$ 
\begin{eqnarray}\label{submart1}
\E\left(g(\chi(t))-g(\chi(t-1))|\mathcal{F}_{t-1},\chi(t-1)\notin \cB\right)& \ge & a_3.
\end{eqnarray} 
Then, 
\begin{itemize}
\item [i.] $\E_{x}(\tau(\cB))\le \frac{2a_1}{a_3},$ for all $x\in \cD.$ Moreover, for any $a_4, T>0$ such that $a_4-a_3T <0$, we have $$P_{x}(\tau(\cB) \ge T) \le \exp\left({-\frac{{(a_4-a_3T)}^2}{4a_2^2T}}\right),$$  for all $x \in \mathcal{D}$ such that $ g(x)\ge a_1-a_4.$\\
\item [ii.]  Suppose for some $a_5>2a_2$,  we have $\cB=\{x\in \cD: g(x)\ge a_1-a_5\}$ and $\cB'=\{x\in \cD: g(x)\le a_1-2a_5\}.$   Then for all $x \in \cB,$ and all $T> \frac{2a_2}{a_3},$
$$\mathbb{P}_{x}(\tau(\cB') \ge T') \ge  1-\left[\exp\left({-\frac{a_5^2 }{32a_2^2T}}\right)+\exp\left({-\frac{a_3^2 T^2}{32a_2^2T}}\right)\right],
$$  where $T' =\exp\left({\frac{\min( a_5^2,a_3^2 T^2 )}{32a_2^2T}}\right).$
\end{itemize} 
 \end{lem}
Using the above result, the proofs of Lemmas \ref{hitting} and \ref{goodbad} follow easily.
\begin{proof}[Proof of Lemma \ref{hitting}]  The proof follows from Lemma \ref{lemm:azuma1} i.
The stochastic process we consider, is the competitive erosion chain $\{\varsigma_{t}\},$ and $g:={\rm{W}}(\sigma)$ is the weight function defined in \eqref{wf}. 
Using Lemma \ref{rem2}, let $\e_3$ be such that $\Omega_{(\e_3)}\subset \cA_{\e_1}.$
We make the following choices of parameters: 
\begin{align*}
\cB & =  \cA_{\e_1}\cap \Omega_{(\e_2)},\\
a_1 &= D n^2, \quad \quad \mbox{ appearing in Lemma \ref{unibounded}},\\
a_2 &= C \log(\dd), \quad \mbox { for a large enough universal constant } C,\\
a_3 &= \min (a(\e_2),a(\e_3)),    \quad \quad  \mbox{where } a(\cdot) \mbox{ appears in Theorem }\ref{negdrift123}, \\
a_4  &= 2Dn^2, \quad T   = \frac{(3D+1)n^2}{a_3}.
\end{align*}
 Thus, \eqref{hyp11} is satisfied by Lemma \ref{unibounded}.
That \eqref{hyp22} is satisfied by our choice of $a_2,$ follows from Lemma \ref{maxbnd1}.  The drift condition \eqref{submart1}, is satisfied by Theorem \ref{negdrift123} and our choice of $\e_3.$ Thus by Lemma \ref{lemm:azuma1} i., the result follows. 
\end{proof}
\begin{proof}[Proof of Lemma \ref{goodbad}]
The proof follows from Lemma \ref{lemm:azuma1} ii. by considering the same process as in the proof of the above lemma, and the following choice of parameters:
\begin{align*}
\cB = \Gamma_{\e_1},&\quad \cB'  =  \Gamma_{2\e_1}, \quad  a_5 = \e_1 n^2,  \quad T  = n^2,\\
a_1 & = {\rm{W}}_{\max},\quad ( \text{see } \eqref{opt122}),\\
a_2 &= C \log(\dd), \,\,\mbox { for a large enough universal constant } C,\\
a_3 &= a(\e_3),
\end{align*}
where $\e_3$ is such that $\Omega_{(\e_3)}\subset \Gamma_{\e_1}.$
Now, by the above choice of parameters $T'=\frac{1}{2}e^{\frac{\min( \e_1^2,a_3^2 )n^2}{16}},$ and by Lemma \ref{lemm:azuma1} ii. for all $\sigma \in \Gamma_{\e_1},$
$$\mathbb{P}_{\sigma}(\tau(\Gamma_{2\e_1})\ge T')\ge  1-[e^{-\frac{\e_1^2 n^2 }{16}}+e^{-\frac{a(\e)^2 n^2}{16}}].$$
Hence the proof is complete. 
\end{proof}
\begin{proof}[Proof of Theorem \ref{quantmainresult}] 
Equipped with the previous hitting time results, we finish off the proof of Theorem \ref{quantmainresult}, using Lemma \ref{hitstation}.
Notice that, by the lower containment in Lemma \ref{rem2}, it suffices to prove that for a given small enough $\e,$ for all large enough $n=2^{m}$,
\begin{equation}\label{suff1}
\pi(\Gamma_{2\e})\ge 1-e^{-cn^2},
\end{equation}
for some $c=c(\e,\dd,\U)> 0.$ 
The proof, now follows immediately from Lemma \ref{hitstation}, with the following choices of the parameters:
\begin{align*}
A_1  = \Gamma_{\e}, 
A_2  = \Omega \setminus \Gamma_{2\e},
t_1  = d_1n^2,
t_2 = e^{d_2n^2},
\end{align*}
where $c_1,c_2,d_1 ,d_2$ are chosen such that the hypotheses \eqref{hyp1} and \eqref{hyp2} of Lemma \ref{hitstation} are satisfied. Lemmas \ref{hitting} and \ref{goodbad} allow us to do that.  
Thus the proof of Theorem \ref{quantmainresult} is complete.
\end{proof}
\section{Properties of the Green function}\label{TP}  
In this section, 
we state asymptotic properties of  the discrete Green function, including asymptotic conformal invariance. 
\subsection{The Discrete Green Function}\label{dg}

The integral in $G_n$ in \eqref{dgf}, is absolutely integrable by the following  lemma.
\begin{lem}\cite[Lemma 3.1]{tc}\label{welldefined} For any domain $\U,$ there exists a constant $D=D(\U)$ such that, for all large enough $n$,
$$\sup_{x\in \U_n}|\mathbb{P}_x(X(t)\in \U_{{\rm{B}},n})-\mathbb{P}_x(X(t)\in \U_{{\rm{R}},n})|\le 2e^{-Dt}$$ 
for all $t\ge 0$, where $X(t)$ is the continuous time random walk on $\U_n$.
\end{lem}
The above is a consequence of the fact that random walk on $\U_n$ converges to the equilibrium measure exponentially fast, and the fact that $\pi_{RW}(\U_{{\rm{B}},n})=\pi_{RW}(\U_{{\rm{R}},n}),$ where $\pi_{RW}$ is the stationary measure for the random walk on $\U_n.$ The last fact follows, since by choice $|\U_{{\rm{B}},n}|=|\U_{{\rm{R}},n}|,$ (see subsection \ref{commen1} ii.), and all the vertices in $\U_{{\rm{B}},n} \cup \U_{{\rm{R}},n}$ have degree four. 

A crucial operator in what follows, is the discrete Laplacian $\Delta$, which acts on  any function $f:\U_n \rightarrow \mathbb{R}$ and produces $\Delta f: \U_n \rightarrow \mathbb{R}$ defined as: for any $x \in \U_n,$ 
 \begin{equation}\label{lapnot1}
 \Delta f(x)=f(x)-\frac{1}{d_x}\sum_{y\sim x}f(y),
\end{equation}
 where $d_x$ is the degree of the vertex $x$ and $y\sim x$ denotes that $y$ is a neighbor of $x$ in the graph $\U_n.$  The next lemma concerns the Laplacian of the Green function.
\begin{lem}\label{laplacian98} \cite[Lemma 3.2]{tc} The Laplacian of the Green function, $\Delta G_n,$ satisfies, 
\begin{equation}\label{laplacian}
\Delta G_n =\frac{1}{| \U_{{\rm{B}},n}|} (\mathbf{1}( \U_{{\rm{B}},n})-\mathbf{1}( \U_{{\rm{R}},n})),
\end{equation}
where for any subset $A\subset \U_n$, $\mathbf{1}(A)$ denotes the indicator of the set $A$.
\end{lem}
\subsection{Conformal invariance of the Green function}\label{cinvariance}
We define the following function on the domain $\U:$ 
\begin{equation}\label{laplace}
\tilde f:=\frac{16}{\rm{ area}(\U_{\rm{B}})}\bigl(\mathbf{1}(\U_{\rm{R}})-\mathbf{1}(\U_{\rm{B}})\bigr),
\end{equation}
where $\U_{\rm{B}}$ and $\U_{\rm{R}}$ are defined in Section \ref{para}.
Again $16$ is a constant that appears for the same reason that $64$ appears in \eqref{area12} and is not important.
Recall the functions $\phi$ and $\psi$ from \eqref{confmap1}.
Then, 
\begin{equation}\label{mappedsource}
\tilde f\circ\phi=\frac{16}{\rm{ area}(\U_{\rm{B}})}\bigl(\mathbf{1}(\psi(\U_{\rm{R}}))-\mathbf{1}(\psi(\U_{\rm{B}}))\bigr).
\end{equation}
For any $\U$, define the function $G_{*}:\overline \U \rightarrow \mathbb{R}$ such that for all $z \in \overline \U,$ if $y\in \overline \DD$ is such that $\phi(y)=z$, then, 
\begin{equation}\label{transformed}
G_{*}(z)=\frac{1}{\pi}\int_{|\zeta|<1}\tilde f\circ \phi(\zeta)|\phi'(\zeta)|^2\log(|(\zeta-y)(1-\bar{\zeta}y)|^2)d\xi d\eta,
\end{equation}
where $\zeta=\xi+i\eta.$ The above expression is well defined on $\overline{\U}$ since by Section \ref{commen1} i., the maps $\phi,\psi$ can be extended to neighborhoods containing $\overline \DD, \overline \U$ respectively.
Notice the dependence of $G_*$ on $\dd$ through $\tilde f$. However, for brevity, we choose to suppress the dependence on $\dd$ in the notation. 
The function $G_*,$ is obtained by integrating the Laplacian obtained in Lemma \ref{laplacian98} against the Green kernel with Neumann boundary conditions which is conformally invariant. This is discussed in \cite{tc}.  Lemma \ref{unifcon} shows that as $\dd$ goes to $0$, the function approaches the function $\log\left|\frac{\psi(z)-i}{\psi(z)+i}\right|,$ up to a multiplicative constant.
and hence is clearly conformally invariant. This establishes that the function $G_*(\cdot)$ is ``asymptotically" conformally invariant.
Even though the function $G_n$ in \eqref{dgf} is defined on the graph $\U_n,$  we use  interpolation  to think of it as a function on $\overline \U.$
We  now  state one of the  key convergence results which establishes asymptotic conformal invariance of $G_{n}(\cdot).$ 
We start with the interpolation scheme. 
For all $x\in \frac{1}{n}\mathbb{Z}^2\setminus \U_n,$  define $G_n(x)=0.$
Now, having defined $G_n(x)$ for all $x,\in \frac{1}{n}\mathbb{Z}^2,$ we define it on $ \mathbb{C},$ and hence by restriction, also  on $\overline \U.$ 
This is done by interpolating $G_n(x)$ by a sequence of harmonic extensions along simplices:
\begin{enumerate}[i.]
\item First extend it along the edges using the value on the vertices in $\frac{1}{n}\mathbb{Z}^2$.
\item Then extend it to the squares using the value on the edges.
\end{enumerate}
Thus the function is extended to the entire complex plane $\mathbb{C}$. By abuse of notation in the following result we call the extended function $G_n(\cdot)$ as well. For the next result recall the blob radius $\delta$ from Section \ref{para}.
\begin{thm}\cite[Theorem 3.1]{tc}\label{convergence} For all small enough $\dd$ depending on $\U,$
$$\lim_{{m\rightarrow \infty}\atop{n=2^m}}\sup_{z\in \overline{\U}}|G_n(z)-G_{*}(z)|=0.$$
\end{thm}
We next make a few observations about the function $G_{*}(\cdot),$ and hence by the aforementioned theorem, about the asymptotics of the function $G_n.$ 
The following is a uniform convergence result of $G_*,$ away from the points $x_{\rm{B}},x_{\rm{R}}$.
\begin{lem}\label{unifcon} Let $\U$ be as in Section \ref{para}. Given $a< 1$,  
$$\lim_{\dd \rightarrow 0}\sup_{{z\in \overline \U:\atop { d(z,x_{\rm{B}})\ge \dd^a}}\atop{d(z,x_{\rm{R}})\ge \dd^a}}\left|G_*(z)-\frac{64}{\pi}\log \left|\frac{\psi(z)-i}{\psi(z)+i}\right|\right|=0$$  
\end{lem}
Lemma \ref{maxbnd1} is now a corollary of the above results and the following: 
\begin{lem}\label{step1} There exists a constant $C>0,$ such that for any $0<a<1,$ and for all small enough $\dd$, for all $z\in \overline \U$ such that $d(z,x_{\rm{B}})\le \dd^{a},$ $$\frac{a}{C}|\log(\dd)| \le G_*(z) \le C |\log(\dd)|$$ and similarly if  $d(z,x_{\rm{R}})\le \dd^{a}$, $$-C|\log(\dd)|\le  G_*(z) \le -\frac{a}{C}|\log(\dd)|.$$ 
\end{lem}
The technical proofs of 
 Lemmas  \ref{unifcon} and \ref{step1} are provided in Appendix \ref{pol}.

\section{Roadmap for the proof of Theorem \ref{negdrift123}}\label{roadmap} The next few sections will be devoted to the proof of Theorem \ref{negdrift123}.
For the reader's convenience, we provide a detailed roadmap as to what each section does. 
The proof of Theorem \ref{negdrift123} is quite long and involved and the argument has to be divided into cases depending on $\sigma$. 
Before presenting formal arguments, we roughly describe the different cases:
\begin{enumerate}
\item The blue connected components intersecting $\U_{{\rm{B}},n}$  or the red connected components intersecting $\U_{{\rm{R}},n}$ have small diameter.
\item There is a red path connecting $\U_{{\rm{R}},n}$ to a point close to $\U_{{\rm{B}},n},$ or there is a blue path connecting $\U_{{\rm{B}},n}$ to a point close to $\U_{{\rm{R}},n}.$  
\item When none of the above cases hold. 
\end{enumerate}
The following lists what each of the subsequent sections achieve:
\begin{enumerate}[(a)]
\item In Section \ref{efn} we prove Theorem \ref{negdrift123} under the first two cases.
\item The proof of the last remaining case appears in Section \ref{build100}. However, it requires quite a bit of preparation.  
As has already been mentioned, the proof is based on electrical network theory of finite graphs, and in particular, the change in one step of the weight function can be realized as the potential drop when certain vertices in the network are glued or identified. The proof is now completed by a quantitative version of Rayleigh's monotonicity principle. 
\item In Subsection \ref{geomdef}, we define various subsets of $\U_n$ and their boundaries, which would be used in defining the glued sets, etc. 
\item In Subsection \ref{pota100}, we recall certain results from the theory of electrical networks, that our arguments will rely on. 
\item In Subsection \ref{stopdef1002}, we define certain variants for the Green function in \ref{dgf}, where the random walk $X(t)$  is replaced by a random walk killed on hitting certain sets. 
Using these, we state and prove the  key Lemma \ref{energyint},
 which realizes the change in the weight function as difference in energies of certain flows on the network $\U_n.$
\item We realize the RHS of the expression in  Lemma \ref{energyint}, as a potential drop when certain vertices of the underlying network $\U_n$ are glued. We then rely on Rayleigh's monotonicity principle, to prove a weaker version of Theorem \ref{negdrift123}, by proving that the drift is almost non-negative (Lemma \ref{weak201}) in Subsection \ref{weak102}.
\item Finally, we prove a stronger quantitative version (Lemma \ref{resquant1}) in Subsection \ref{quant102}. 
\item
Using all the above ingredients, the proof of Theorem \ref{negdrift123} follows quickly.
\end{enumerate}
\section{Proof of Theorem \ref{negdrift123} under the first two 
cases}\label{efn}
We start with some definitions. For brevity, we will suppress the $\sigma$ dependence in the notations we define subsequently, whenever there is no scope for confusion.
\begin{definition}\label{areadef2} Given $\sigma \in \Omega\cup \Omega'$ (see \eqref{statespace} for definition), 
 let $R_{\rm{B}} = R_{\rm{B}}(\sigma)$ (resp. $R_{\rm{R}}$) be the set of all points in $\U_n$ of color blue (resp. red) reachable by a monochromatic path of blue (resp. red) vertices  from a point in  $\U_{{\rm{B}},n}$ (resp. $\U_{{\rm{R}},n}$). Note that we do not  allow the endpoint to be of the opposite color and that $R_{\rm{B}}$ and $R_{\rm{R}}$ are disjoint. 
\end{definition}
\begin{figure}[hbt]
\centering
\includegraphics[scale=.25]{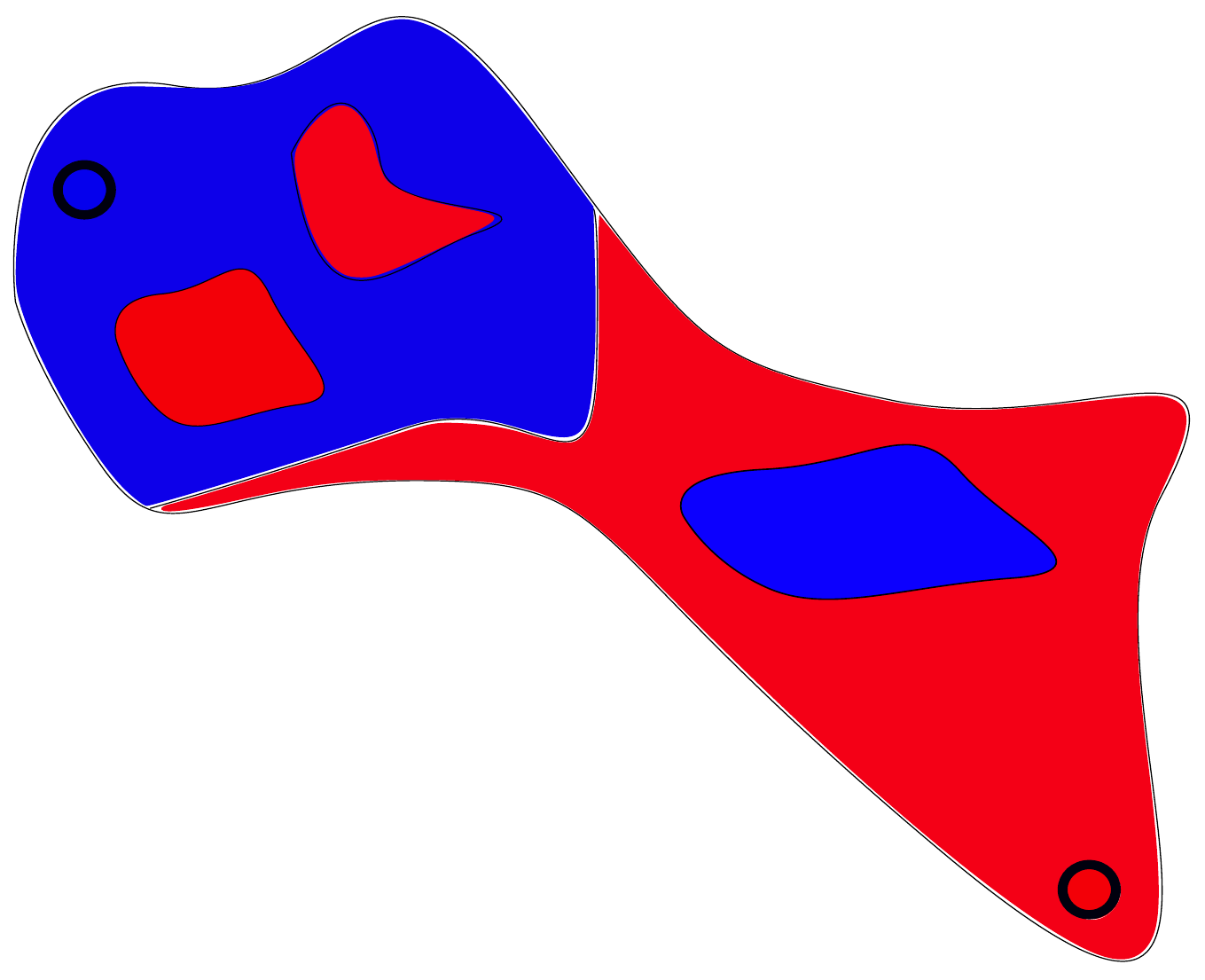}
\label{f.region}
\caption{The blue connected region containing the blue blob is  $R_{\rm{B}}$ and similarly  $R_{\rm{R}}$. Note that the two red islands and the blue island are not included in either set. It will be useful to keep in mind that  if a part of the blue blob is colored red, it is not included in the set $R_{\rm{B}}$.}
\end{figure}
For ${\rm{C}}\in \{{\rm{B}},\rm{R}\},$ and $\sigma \in \Omega \cup \Omega',$ let,
\begin{equation}\label{stoptime}
\sT_{({\rm{C}},\sigma)} :=\tau (\U_n \setminus R_{\rm{C}}(\sigma)),
\end{equation}
be the hitting time of $\U_n \setminus R_{\rm{C}}(\sigma),$ for the continuous time random walk on $\U_n$.
Let $X^{({\rm{B}})}(t)$ and $X^{({\rm{R}})}(t)$ be continuous time random walks on $\U_n$ started uniformly over $\U_{{\rm{B}},n}$ and $\U_{{\rm{R}},n}$ respectively. Thus  by definition, 
\begin{equation}\label{alternate}
\E_{\sigma}({\rm{W}}(\varsigma_{1})-{\rm{W}}(\varsigma_{0}))=\E[G_n(X^{({\rm{B}})}(\sT_{({\rm{B}},\varsigma_0)}))-G_n(X^{({\rm{R}})}(\sT_{({\rm{R}},\varsigma_{1/2})}))].
\end{equation}
For the two terms on the right hand side, the expectation in the first term is over the measure induced by the blue random walk started according to $\mu_{\rm{B}}$ (see Section \ref{para}.) Note that the expectation in the second term is over the measure induced by the red random walk started according to $\mu_{\rm{R}},$ as well as the random intermediate configuration $\varsigma_{1/2},$ (see \eqref{fullpro1}). 

We now state the first theorem towards the proof of Theorem \ref{negdrift123} which covers the first two cases mentioned in Section \ref{roadmap}.
 \begin{thm}\label{casei} Let $\U,y_{\rm{B}},y_{\rm{R}}, \dd$ be as in Section \ref{para}. There exists $D>0,$ and  
$ 0<a_2 < a_1<1,$ and $\dd_0>0,$ such that for $\dd \le \dd_0,$ and all $n=2^{m}>N(\dd),$   if $\sigma\in \Omega$ is such that, either, 
\begin{enumerate}
\item [i.]  $R_{\rm{B}}$ does not intersect $\U_n\setminus B_n(y_{\rm{B}},\dd^{a_2})$ or,
\item [ii.] $R_{\rm{B}}$  intersects  $\U_n\setminus B_n(y_{\rm{B}},\dd^{a_2}) $  and $R_{\rm{R}}$ intersects $B_n(y_{\rm{B}},\dd^{a_1}) $ 

or  switching the roles of $R_{\rm{B}}$ and $R_{\rm{R}}$  in the above two cases 

\item [iii.] $R_{\rm{R}}$ does not intersect $\U_n \setminus B_n(y_{\rm{R}},\dd^{a_2})$ or,
\item [iv.] $R_{\rm{R}}$  intersects  $\U_n \setminus B_n(y_{\rm{R}},\dd^{a_2}) $  and $R_{\rm{B}}$ intersects $B_n(y_{\rm{R}},\dd^{a_1}),$ 
\end{enumerate} 
then, 
\begin{equation}\label{subm1}
D^{-1}|\log(\dd)| \le \E_{\sigma}({\rm{W}}(\varsigma_1)-{\rm{W}}(\varsigma_0)) \le D |\log(\dd)|.
\end{equation} 
\end{thm}
The reason for the above is roughly the following: (we only discuss the first two cases.) 
\begin{itemize}
\item In case $i.,$ since the diameter of $R_{\rm{B}}$ is small, the blue random walk stops before exiting a small ball. Thus the first term on the RHS in \eqref{alternate} is roughly $|\log(\dd)|$ by Lemma \ref{step1}. Now, by standard random walk estimates one can show that the red walk is likely to exit $R_{\rm{R}}$ before coming close to $x_{\rm{B}}.$ Thus the second term is much smaller and hence the result follows. Figure  \ref{f.prooffig} provides an illustration.
\item In case $ii.,$ by hypothesis, there is a red path which connects $\U_{{\rm{R}},n}$ to a small ball around $y_{\rm{B}}$ of radius $\dd^{a}$ for some $a$. Then it becomes likely, that the blue random walk starting from $\U_{{\rm{B}},n},$ will hit this path and hence exit $R_{\rm{B}}$ before going too far from $x_{\rm{B}},$ which again makes the first term in the RHS in \eqref{alternate}, roughly $|\log(\dd)|$.  The rest of the arguments are then the same as in the previous case which shows that the second term in \eqref{alternate} is much smaller. See Figure  \ref{f.prooffig1} for an illustration. 
\end{itemize}

To carry out the above steps one needs estimates about the random walk on $\U_n$ which were obtained by the authors in \cite{tc1}. We now quote two results from there, we will need. The first one  says the following: uniformly for any point $z$ at distance $1/2$ (any fixed constant would work) from $y_{\rm{B}},$ and all configurations $\sigma \in \Omega \cup \Omega',$ the chance that the random walk starting from $z,$ does not hit $\rm{S}_{\rm{B}},$ before being at distance $\e$ from $y_{\rm{B}},$ goes to $0$ as $\e$ goes to $0$.  
\begin{figure}[hbt]
\centering
\includegraphics[scale=.5]{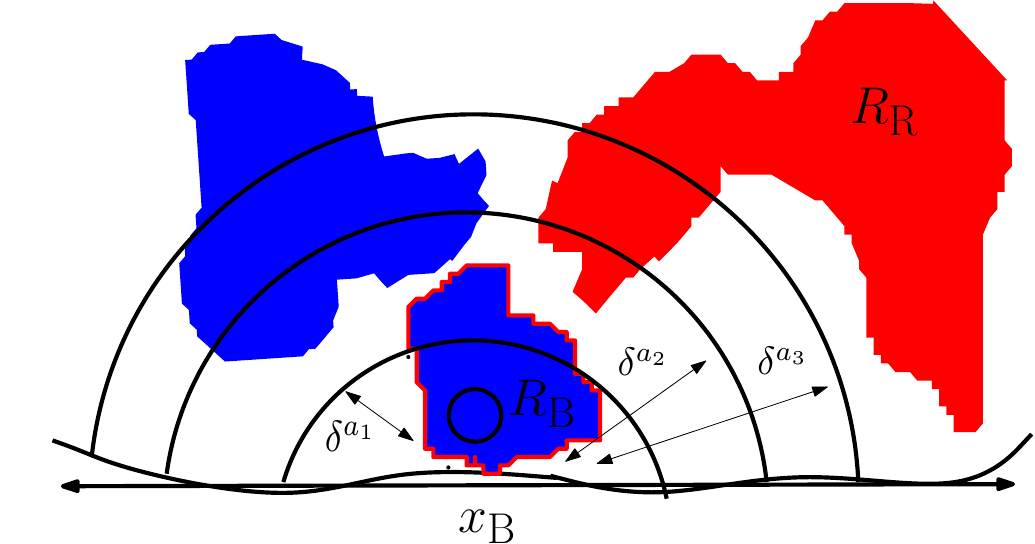}
\caption{Illustrating the case when the diameter of $R_{\rm{B}}$ is less than $\dd^{a_2}$ and hence the blue walker stops within $B_n(y_{\rm{B}},\dd^{a_2})$.}   
\label{f.prooffig}
\end{figure}
\begin{lem}\cite[Lemma 5.2]{tc1}\label{polydecay} For all $\U,y_{\rm{B}},$ as in Section \ref{para},  
$$\lim_{\e \to 0}\limsup_{n\rightarrow \infty}\sup_{\sigma \in \Omega \cup \Omega'}\left[\sup_{z\in \U_n \setminus B_n(y_{\rm{B}},\frac{1}{2})}\mathbb{P}_z\bigl (\tau(B_n(y_{\rm{B}},\e))\le \tau({\rm{S}_{\rm{B}}})\bigr)\right]=0,$$
where $\rm{S}_{\rm{B}}=\rm{S}_{\rm{B}}(\sigma)$ was defined in \eqref{region}, $\tau(\cdot)$ denotes the hitting time for the random walk on $\U_n,$ and $\mathbb{P}_{z}(\cdot)$ denotes, the random walk probability measure started from $z \in \U_n.$ 
\end{lem}
The next lemma,   roughly says that if a connected set $A,$ of a large enough diameter, is close to $\U_{{\rm{B}},n},$ then random walk starting from $\U_{{\rm{B}},n},$ is  likely to hit the set $A,$ before exiting a large enough ball.
\begin{lem}\label{capa}\cite[Lemma 5.6]{tc1} Let $0<\e_1< \e_2.$ Assume $A\subset \U_n$ is a connected set such that,  
 $d(\U_{{\mathrm{B}},n},A)\le \e_1.$ and $A \cap \bigl\{\U_n\setminus B_n(y_{\rm{B}},\e_2)\bigr\}\neq \emptyset.$ 
Then, $$\sup_{x\in \U_{{\mathrm{B}},n}}\mathbb{P}_x(\tau(\U_n\setminus B_n(y_{\rm{B}},\e_2))\le \tau(A))\le {C^{\log{\left(\frac{\e_2}{\e_1}\right)}}}, $$
for some $C=C(\U)<1$ independent of $n$.
\end{lem}
We are now ready to prove Theorem \ref{casei}.
\begin{proof}[Proof of Theorem \ref{casei}]
By symmetry it suffices to prove the theorem for the first two cases.

\textit{Case i.} Let $0<a_3<a_2<a_1< 1$ be constants to be specified later and for $i=1,2,3,$ define $\dd_i=\dd^{a_i}$ and $p(\dd_3)$ be the probability that the random walk started uniformly from $\U_{{\rm{R}},n}$ hits $B_n(y_{\rm{B}},\dd_3)$ before hitting $\rm{S}_{\rm{B}}(\sigma)$.   
We claim that, in this case there exists constants $D_1,D_2 >0,$  such that given $a_2,a_3,$ for all small enough $\dd,$ and large enough $n,$ (depending on $\delta$), 
\begin{align}\label{proofexp1}
\mathbb{E}_{\sigma}\bigl({\rm{W}}(\varsigma_{1})-{\rm{W}}(\varsigma_0)\bigr) &\ge  D_1|\log(\dd_2)|-\bigl[p(\dd_3)D_2 |\log(\dd)|+(1-p(\dd_3))D_2|\log(\dd_3)|\bigr],\\
\nonumber
&=D_1a_2|\log(\dd)|-\bigl[p(\dd_3)D_2 |\log(\dd)|+(1-p(\dd_3))D_2a_3|\log(\dd)|\bigr].
\end{align}
To see the above, recall \eqref{alternate},
$$\E_{\sigma}({\rm{W}}(\varsigma_{1})-{\rm{W}}(\varsigma_{0}))=\E[G_n(X^{({\rm{B}})}(\sT_{({\rm{B}},\varsigma_0)}))-G_n(X^{({\rm{R}})}(\sT_{({\rm{R}},\varsigma_{1/2})}))].$$
Now by hypothesis, $R_{\rm{B}}$ does not intersect $\U_n \setminus B_n(y_{\rm{B}},\dd_2)$ (see Figure \ref{f.prooffig}), and hence the blue random walk stops before exiting $B_n(y_{\rm{B}},\dd_2)$. Thus from Theorem \ref{convergence} and Lemma \ref{step1}, it follows that the first term on the RHS is at least $D_1|\log(\dd_2)|$. 
Now, on the event that the red random walk enters the ball $B_n(y_{\rm{B}},\dd_3)$ before hitting $\rm{S}_{\rm{B}}(\varsigma_{1/2}),$ which happens with probability at most $p(\dd_3)$, the second term is at most $O(|\log(\dd)|),$ by  Lemma \ref{maxbnd1}. Otherwise it is at most  $D_2|\log(\dd_3)|$ by   Theorem \ref{convergence} and Lemma \ref{unifcon}.
Since $p(\dd_3)$ goes to $0$ as $\dd$ goes to $0,$ we can suitably choose $a_3< a_2$ to complete the argument. Note that we do not need $a_1$ for this argument. It appears in the analysis of the next case.
 \begin{figure}[hbt]
\centering
\includegraphics[scale=.5]{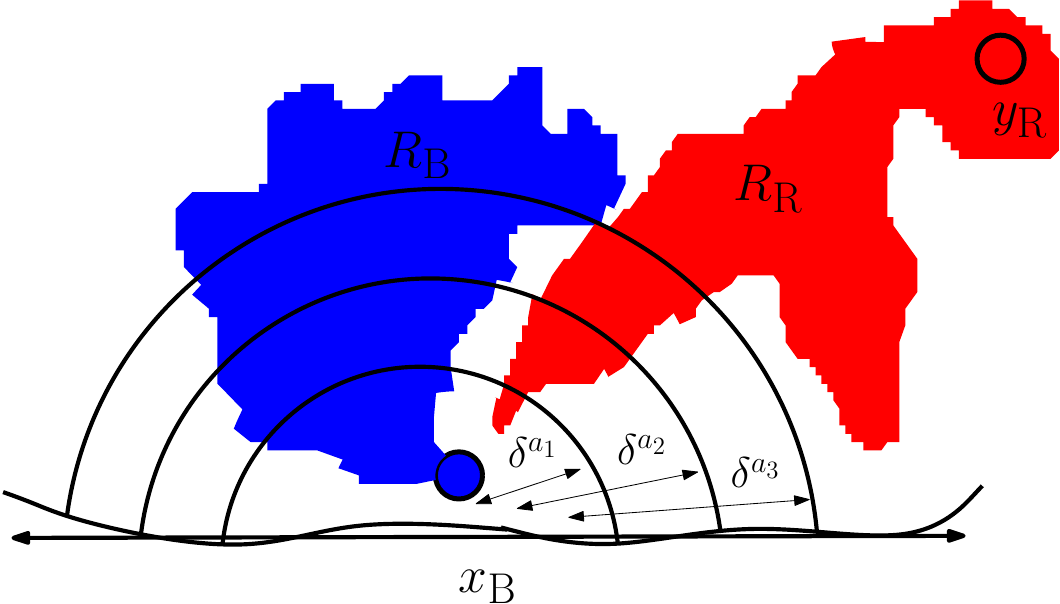}
\caption{Illustrating the case when $R_{\rm{R}}$ connects $B_n(y_{\rm{R}},\dd)$ to a point in the neighborhood of $x_{\rm{B}}$.}
\label{f.prooffig1}
\end{figure}

\textit{Case ii. }  By hypothesis, there exists a path $\sP,$ of red vertices connecting $B_n(y_{\rm{B}},\dd_1)$ and $\U_n \setminus B_n(y_{\rm{B}},\dd_2).$ Thus, it should be likely that the blue random walk started from $\U_{{\rm{B}},n},$ will hit $\sP$ before exiting $B_n(y_{\rm{B}},\dd_2).$ This is precisely the content of Lemma \ref{capa} which implies,
\begin{equation}\label{pathhit}
\inf_{x\in \U_{{\rm{B}},n}}\mathbb{P}_{x}\bigl(\tau(\sP)\le \tau(\U_n\setminus B_n(y_{\rm{B}},\dd_2))\bigr) \ge 1- C^{\log(\frac{\dd_2}{\dd_1})},
\end{equation}  for some $C=C(U)<1.$
Hence, as in the previous case, there exists constants $D_1,D_2 >0,$ such that given $a_1,a_2,a_3,$  for all small enough $\delta$ and  large enough $n$ (depending on $\delta$), 
\begin{align}\label{explain2}
\mathbb{E}_{\sigma}\bigl({\rm{W}}(\varsigma_{1})-{\rm{W}}(\varsigma_0)\bigr) &\ge D_1 |\log(\dd_2)|-  C^{\log(\dd_2/\dd_1)} D_2|\log(\dd)| \\
\nonumber
&-\bigl[p(\dd_3)D_2|\log(\dd)|+(1-p(\dd_3))D_2|\log(\dd_3)|\bigr].
 \end{align}
To see the above, note that $G_{n}(X^{({\rm{B}})}(\sT_{({\rm{B}},\varsigma_0)}))$ is at least $D_1|\log(\dd_2)|,$ if the blue random walk stops before exiting $B_n(y_{\rm{B}},\dd_2),$ which happens with chance at least $1-C^{\log(\dd_2/\dd_1)},$ by \eqref{pathhit}. Otherwise, by  Lemma \ref{maxbnd1},  it is  at least $-D_2|\log(\dd)|.$ 
We bound the term $G_{n}(X^{(\rm{R})}(\sT_{({\rm{R}},\varsigma_{1/2})})),$ exactly as in case $i.$  
The expression in \eqref{explain2} simplified becomes, $$D_1 a_2|\log(\dd)|-  C^{\log(\dd_2/\dd_1)} D_2|\log(\dd)| 
-\bigl[p(\dd_3)D_2|\log(\dd)|+(1-p(\dd_3))D_2a_3|\log(\dd)|\bigr].$$
Thus, one can suitably choose constants $a_3 < a_2 < a_1,$ such that the RHS in both \eqref{proofexp1} and \eqref{explain2} are $C_1|\log(\dd)|,$ for some $C_1>0$ (independent of $\delta,n$) 
and hence, the proof of Theorem \ref{casei} is complete.
\end{proof}
To finish the proof of Theorem \ref{negdrift123}, we have to consider Case (3) mentioned in Section \ref{roadmap}.
\textbf{Thus, throughout the sequel until the completion of the proof of Theorem \ref{negdrift123} we will work under the following assumption:} \begin{assumption}\label{ass2}: $\sigma$ does not fall in any of the four cases in the statement of Theorem \ref{casei}.
That is, both $R_{\rm{B}}$ and $R_{\rm{R}}$ extend beyond $B_n(y_{\rm{B}},\dd_2)$ and $B_n(y_{\rm{R}},\dd_2)$ respectively and also $R_{\rm{R}}$ does not reach $B_n(y_{\rm{B}},\dd_1)$ and similarly $R_{\rm{B}}$ does not reach $B_n(y_{\rm{R}},\dd_1),$ (recall $\delta_1=\delta^{a_1},\delta_2=\delta^{a_2}$ from the statement of Theorem \ref{casei}).
\end{assumption} 

Thus under Assumption \ref{ass2},
\begin{eqnarray}\label{nointer123}
B_n(y_{\rm{B}},\dd_1)\cap R_{\rm{R}}&=&\emptyset,\\
\nonumber
B_n(y_{\rm{R}},\dd_1)\cap R_{\rm{B}}&=&\emptyset. 
\end{eqnarray}
Notice that we have $\U_{{\rm{B}},n}\subset B_n(y_{\rm{B}},\dd_1),$ (similarly $\U_{{\rm{R}},n}\subset B_n(y_{{\rm{R}}},\dd_1)$) for all small enough $\dd$, since $\dd_1=\dd^{a}$ for some $a<1.$ 
\section{Completing Proof of Theorem \ref{negdrift123}}\label{build100}
The proof of Theorem \ref{negdrift123} under the above assumption,  requires quite a bit of preparation. An elaborate description of what various subsequent subsections achieve, was provided in Section \ref{roadmap}.
\subsection{Some geometric definitions}
\label{geomdef}
For technical reasons sometimes it will be convenient to work with a slightly larger set instead of $R_{\rm{B}}.$
To define it, first observe that for any small enough $\dd>0$, for all $n\ge N(\dd),$ there exists a set $A_{{\rm{B}},n}\subset \U_n$
such that $A_{{\rm{B}},n}$ is connected, and,
$$ B_n(y_{\rm{B}},\frac{\dd_1}{4})\cap \U_n \subseteq A_{{\rm{B}},n}\subset B_n(y_{\rm{B}},\dd_1)\cap \U_n.$$ Similarly, one  defines, $A_{{\rm{R}},n}$ by replacing $y_{\rm{B}}$ by $y_{\rm{R}}.$ 
These facts follow since,  locally near the boundary $\U$ looks like a half plane, see \eqref{localhalf}. We omit the details.
The exact form of the set $A_{{\rm{B}},n}$ will not be important for us as long as the above properties are satisfied and we could as well work with $B_n(y_{\rm{B}},\frac{\dd_1}{4})\cap \U_n$ if the latter is connected.

Since $\U_{{\rm{B}},n}\subset B_n(y_{\rm{B}},\frac{\dd_1}{4})\cap \U_n \subseteq A_{{\rm{B}},n},$ (and similarly $\U_{{\rm{R}},n}\subset B_n(y_{{\rm{R}}},\frac{\dd_1}{4})\cap \U_n \subseteq A_{{\rm{R}},n}$),
it immediately follows from definition, that $R_{\rm{B}}\cup A_{{\rm{B}},n}$ and $R_{\rm{R}}\cup A_{{\rm{R}},n}$ are two disjoint, connected, subsets of $\U_n.$
 Let $\U_n\setminus \{R_{\rm{B}}\cup A_{{\rm{B}},n}\}= \bigcup_{i=1}^{m} C_{{\rm{B}},i},$ where $C_{{\rm{B}},1}, \ldots C_{{\rm{B}},m},$ are the connected components of $\U_n\setminus \{R_{\rm{B}}\cup A_{{\rm{B}},n}\}$. 
 Let $\{R_{\rm{R}}\cup A_{{\rm{R}},n}\}\subseteq C_{{\rm{B}},m}.$  
 
We will often work with the following `simply connected' version of ${R}_{\rm{B}}:$ 
\begin{equation}\label{fillholes}
\mathring{R}_{\rm{B}}= R_{\rm{B}}\cup A_{{\rm{B}},n}\bigcup_{i\neq m} C_{{\rm{B}},i}.
\end{equation}
Informally, $\mathring{R}_{\rm{B}}$ is the set $R_{\rm{B}}\cup A_{{\rm{B}},n}$ along with all the  red islands filled in (see Figure  \ref{f.regdef1}).
Similarly define $\mathring{R}_{\rm{R}}.$ Observe that  
$\mathring{R}_{\rm{C}}$ is connected for ${\rm{C}}\in \{{\rm{B}},\rm{R}\}$. Note that $\mathring{R}_{\rm{C}}$ depends on $\sigma$ which we choose to suppress for the ease of notation. 
\begin{figure}[hbt]
\center
\includegraphics[scale=.6]{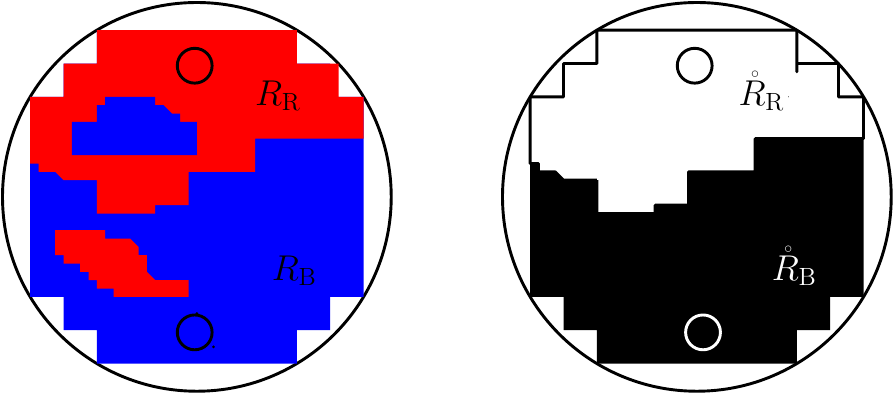}
\caption{Illustrating the difference between the sets ${R}_{\rm{C}}$ and $\mathring{R}_{\rm{C}}$ for ${\rm{C}}\in \{{\rm{B}},\rm{R}\}$.  Note that in the latter, the islands are filled in.}
\label{f.regdef1}
\end{figure}
Next, given a vertex subset $A\subset \U_n,$ we define two types of boundaries (outer and inner) of $A,$

\begin{eqnarray}\label{bdry1}
\partial_{out}A &:=&\{y \in \U_n\setminus A:\,\,\exists\,\, x\in A\,\,\text{such that }  x\sim y\}\\
\nonumber
\partial_{in}A & := &\{x \in A:\,\,\exists\,\, y\in \U_n \setminus A\,\,\text{such that }  x\sim y\}.
\end{eqnarray}

\begin{definition}\label{conndef1}
Let $\U_n^*$ denote the graph $\U_n$ along with all the diagonals of the squares all of whose four sides are in $\U_n.$ We will call connected subsets of $\U_n^*$ as $*-$ connected subsets of $\U_n.$
\end{definition}
\begin{remark}\label{connrem}

Since both $\mathring{R}_{\rm{B}}$ and $\U_n \setminus \mathring{R}_{\rm{B}}$ are connected by definition, it follows by \cite[Lemma 2]{at1}, that both $\partial_{out}\mathring{R}_{\rm{B}}$ and $\partial_{in}\mathring{R}_{\rm{B}}$ are $*-$connected. 
\end{remark}

\subsection{Energy and flows on finite graphs and some potential theory}\label{pota100} In this section we recall some well known results and notation from potential theory that we will need. 
For a graph $G=(V,E),$ with vertex set $V$ and edge set $E,$ as before, let 
$w \sim v$ denote that $w$ is a neighbor of $v$.
Also, let $\vec{E}$ be the set of directed edges, where each edge in $E$ corresponds
to two directed edges in $\vec{E}$, one in each direction
(except for self-loops, which correspond to just one directed edge in $\vec{E}$).
A flow is an antisymmetric  function $f: \vec{E} \rightarrow \mathbb{R}$
(i.e.,\ a function satisfying $f(w,v)=-f(v,w)$). Note that by definition the value of a flow on a self loop is $0.$
Define the energy of a flow $f$ by,
\begin{equation}\label{enerno}
\mathcal{E}(f)=\frac12 \sum_{(v,w)\in \vec{E}}f(v,w)^2.
\end{equation} 
For any flow $f: \vec{E} \rightarrow \mathbb{R},$
 define the divergence, $\div f: V \rightarrow \mathbb{R}$
by, 
\begin{equation}\label{divdef}
\div f (v) = \sum_{w\sim v} f(w,v).
\end{equation}
Note that for any flow, 
\begin{equation}\label{divsum}
\sum_{x\in V}\div f=\sum_{{x,y \in V}\atop {y \sim x}}f(x,y)+f(y,x)=0,
\end{equation}
 since $f$ is antisymmetric.
For disjoint subsets $A,B \subset V$ and a flow $f,$ we say that the flow is from $A$ to $B,$ if $\div f(z)=0$ for all vertices $z\notin A\cup B,$ while $\sum_{x\in A}\div f \ge 0,$ and $\sum_{x\in B}\div f\le 0.$ For more about flows, see \cite[Chapter 9]{lpw}.
For any function $F: V \rightarrow \mathbb{R},$
define the gradient $\nabla F : \vec{E} \rightarrow \mathbb{R}$ by,
$$\nabla F (v,w) = F(w)-F(v).$$
Recall the definition of Laplacian from \eqref{lapnot1}. Thus clearly for any $F: V\rightarrow \mathbb{R},$
\begin{equation}\label{relation12}
\div (\nabla F)(v)=d_v. (\Delta F) (v),
\end{equation} 
where $d_v$ is the degree of the vertex $v$.
The next result is a standard summation by parts formula whose proof we omit.
\begin{lem}\label{byparts}For any function $F: V\rightarrow \mathbb{R},$
$$\mathcal{E}(\nabla F)=\sum_{v\in V}F(v)d_v. \Delta F(v).$$
\end{lem}
We now discuss a standard interpretation of the Green function $G_{n}(\cdot)$ defined on $\U_n$ in \eqref{dgf} in terms of energy and flows. 
\begin{lem}\cite[Theorem 9.10, (Thomson's principle)]{lpw}\label{var} For all flows $\theta$ on $\U_n$ such that $\div \theta =\frac{4}{|\U_{{\rm{B}},n}|}\bigl[\mathbf{1}(\U_{{\rm{B}},n})-\mathbf{1}(\U_{{\rm{R}},n})\bigr],$ we have  $\mathcal{E} (\nabla G_n) \leq \mathcal{E}(\theta).$ \end{lem}

First observe that by \eqref{laplacian} and \eqref{relation12} the flow $\nabla G_n$ satisfies the above divergence condition. Note that the $4$ appears since every vertex in $\U_{{\rm{B}},n}$ and $\U_{{\rm{R}},n}$ has degree $4.$
The result now follows by standard arguments which can be found in \cite{lpw}. 

For notational brevity let,
\begin{eqnarray}
\label{not1}
{\rm{Ind}}_{\rm{B}} (\cdot)& := & \frac{1}{|\U_{{\rm{B}},n}|}\mathbf{1}(\U_{{\rm{B}},n})(\cdot),\\
\label{not2}
{\rm{Ind}}_{\rm{R}} (\cdot)& := & \frac{1}{|\U_{{\rm{B}},n}|}\mathbf{1}(\U_{{\rm{R}},n})(\cdot),\\
\label{not3}
{\rm{Ind}}(\cdot)& := & {\rm{Ind}}_{{\rm{B}}}(\cdot)-{\rm{Ind}}_{\rm{R}}(\cdot).
\end{eqnarray}
Thus rewriting \eqref{laplacian} in this notation we get,\begin{equation}\label{reword}
\Delta G_{n}={\rm{Ind}}.
\end{equation}
\subsection{Stopped Green functions.}\label{stopdef1002}
For ${\rm{C}}\in \{{\rm{B}},\rm{R}\},$ and $\sigma\in \Omega \cup \Omega',$ define 
\begin{equation}\label{stoppedgreen}
G_{{\rm{C}},n}(x): = G_{{\rm{C}},n}(\sigma)(x):= 2n^2\E_x [\int_{0}^{\sT_{({\rm{C}},\sigma)}} {\rm{Ind}}(X(t))dt], 
\end{equation}
where $X(t)$ is the continuous time random walk on $\U_n$ as defined in Subsection \ref{techass}, and for ${\rm{C}}\in \{{\rm{B}},\rm{R}\}$, the stopping time  $\sT_{({\rm{C}},\sigma)}$ was defined in \eqref{stoptime} as the hitting time of $\U_n \setminus R_{\rm{C}}(\sigma),$ for $X(t).$ 

Unlike the expression in \eqref{dgf} the above integral is only up to the stopping time $\sT_{({\rm{C}},\sigma)},$ and hence we name it as the stopped Green function. 
Similar to \eqref{dgf} we again suppress the dependence on $\dd$ in the notation $G_{{\rm{C}},n}$.
 Now notice that
\begin{eqnarray}\label{lapl1}
\Delta G_{{\rm{C}},n}(\sigma)(x) & = & {\rm{Ind}}(x) \mbox{ for } x\in R_{{\rm{C}}}(\sigma),\\
\label{lapl123}
G_{{\rm{C}},n}(\sigma)\mid_{\U_n\setminus R_{\rm{C}}(\sigma)} & = & 0, 
\end{eqnarray}
where following standard notation $\mid_{\cdot}$ will be used to denote the restriction of a function.

Recalling notation for the energy of a flow from \eqref{enerno}, we state the following key lemma,
which translates Theorem \ref{negdrift123} to a statement about difference in energies of flows.

\begin{lem}\label{energyint} Under Assumption \ref{ass2},
\begin{eqnarray}
\label{keyineq}
\E_{\sigma}({\rm{W}}(\varsigma_{1})-{\rm{W}}(\varsigma_{0}))  &\ge & \frac{1}{4}[\mathcal{E}(\nabla G_{n})- \mathcal{E}(\nabla G_{{\rm{B}},n}(\varsigma_0))-\mathcal{E}(\nabla G_{{\rm{R}},n}(\varsigma_0))].
\end{eqnarray}
\end{lem}
Thus proving Theorem \ref{negdrift123} is  reduced to lower bounding the RHS above.
Before proving Lemma \ref{energyint} we make a few more observations about the functions $G_{{\rm{C}},n}(\cdot)$, which will be used.
First notice that if $\{X_k\}_{k\ge0}$ was a discrete time\footnote{At every discrete time $i\in \{1,2, \ldots \}$ the random walk jumps to a uniform neighbor of the location at time $i-1.$} random walk on $\U_n$ then for ${\rm{C}}\in \{{\rm{B}},\rm{R}\},$ the integral in \eqref{stoppedgreen} is exactly the same\footnote{This is because, by definition, the mean of the exponential delays in the continuous time random walk is $\frac{1}{2n^2}.$} as 
\begin{equation}\label{sumrep1}
\E_x \left[\sum_{k=0}^{\sT_{({\rm{C}},\sigma)}-1} {\rm{Ind}}(X_{k})\right],
\end{equation} where $\sT_{({\rm{C}},\sigma)}$ is now the hitting time of $\U_n \setminus R_{\rm{C}}(\sigma)$ for the discrete time random walk. 

\textbf{
Thus for the rest of the section for notational convenience we will switch to the sum notation and work with the discrete time random walk}. 

As emphasized in Sec \ref{techass}, to avoid introducing new notation we will use $\tau(\cdot)$ to denote hitting times for the discrete time random walk as well, and in particular as above, $\sT_{({\rm{C}},\sigma)}$ will denote the hitting time of $\U_n\setminus R_{\rm{C}}(\sigma)$.
\begin{remark}\label{interprop} By \eqref{nointer123}  under Assumption \ref{ass2} $R_{\rm{B}}$ does not intersect $\U_{{\rm{R}},n}$ which is the support of ${\rm{Ind}}_{\rm{R}}$ and similarly $R_{\rm{R}}$ does not intersect $\U_{{\rm{B}},n}$ which is the support of ${\rm{Ind}}_{\rm{B}}$.
\end{remark}
Thus from \eqref{sumrep1} we get,
\begin{align}\label{stoppedgreenold}
G_{{\rm{B}},n}(\sigma)(x)&= \mathbb{E}_x\left[\sum_{k=0}^{\sT_{({\rm{B}},\sigma)}-1}{\rm{Ind}}_{\rm{B}}(X_k)\right],\\
\nonumber
G_{{\rm{R}},n}(\sigma)(x)&= \mathbb{E}_x\left[\sum_{k=0}^{\sT_{({\rm{R}},\sigma)}-1}{\rm{Ind}}_{\rm{R}}(X_k)\right].
\end{align}
Also,  \eqref{lapl1} and \eqref{lapl123} become,
\begin{eqnarray}\label{newlapl1}
\Delta G_{{\rm{C}},n}(\sigma)(x) & = & {\rm{Ind}}_{\rm{C}}(x) \mbox{ for } x\in R_{\rm{C}}(\sigma), \\
\label{newlapl2}
G_{{\rm{C}},n}(\sigma)\mid_{\U_n\setminus R_{\rm{C}}(\sigma)} & = & 0,
\end{eqnarray}
respectively.
An easy but useful observation is the following:
\begin{lem}\label{easecrucial1} 
Under Assumption \ref{ass2}, for any $\sigma\in \Omega$,
\begin{equation}\label{easycrucial}
\E_{\sigma}[(G_{{\rm{R}},n}(\varsigma_{1/2})(x)]\le G_{{\rm{R}},n}(\sigma)(x)
\end{equation}
for all $x \in \U_n,$ where the expectation in the first term is over the random intermediate state $\varsigma_{1/2}$ (see \eqref{fullpro1}).
\end{lem}
\begin{proof} The proof immediately follows from \eqref{stoppedgreenold} since $\sT_{({\rm{R}},\varsigma_{1/2})}\le \sT_{({\rm{R}},\varsigma_{0})}$ as $\varsigma_{1/2}$ has one more blue particle than $\varsigma_0.$
\end{proof}

We are now ready to prove Lemma \ref{energyint}.
\begin{proof}[Proof of Lemma \ref{energyint}]Let us denote the discrete time random walks starting uniformly from $\U_{{\rm{B}},n}$ and $\U_{{\rm{R}},n}$ by  $X^{({\rm{B}})}_k$ and  $X^{({\rm{R}})}_k$ respectively. Thus \eqref{alternate} rewritten becomes $$\E_{\sigma}({\rm{W}}(\varsigma_{1})-{\rm{W}}(\varsigma_{0})) =  \E (G_{n}(X^{({\rm{B}})}_{\sT_{({\rm{B}},\varsigma_0)}}) - G_{n}(X^{({\rm{R}})}_{\sT_{({\rm{R}},\varsigma_{1/2})}})).$$ 
 We start with the following telescopic sum,
 $$G_{n}(X^{({\rm{B}})}_{\sT_{({\rm{B}},\varsigma_0)}})=G_{n}(X^{({\rm{B}})}_0)+[G_{n}(X^{({\rm{B}})}_1)-G_{n}(X^{({\rm{B}})}_0)]+\ldots +[G_{n}(X^{({\rm{B}})}_{\sT_{({\rm{B}},\varsigma_0)}})-G_{n}(X^{({\rm{B}})}_{\sT_{({\rm{B}},\varsigma_0)}-1})].$$
Now clearly $$\E(G_{n}(X^{({\rm{B}})}_{i+1})-G_{n}(X^{({\rm{B}})}_{i})\mid \mathcal{F}_{i})=-\Delta G_{n}(X^{({\rm{B}})}_i)\mathbf{1}(\sT_{({\rm{B}},\varsigma_0)}> i),$$ where $\mathcal{F}_i$ is the filtration generated by  $\varsigma_0$ and $\{X^{({\rm{B}})}_j\}_{0\le j\le i}$. 
Taking expectation on both sides\footnote{The sum and the expectation can be interchanged, since all the hitting times in this paper have exponential tails. We omit the details.} we get 
$$\E \bigl(G_{n}(X^{({\rm{B}})}_{\sT_{({\rm{B}},\varsigma_0)}})\bigr)=\E \left[G_{n}(X^{({\rm{B}})}_0) - \sum_{k=0}^{\sT_{({\rm{B}},\varsigma_0)}-1} \Delta G_{n}(X^{({\rm{B}})}_k) \right].$$ A similar equation holds for $\E (G_{n}(X^{({\rm{R}})}_{\sT_{({\rm{R}},\varsigma_{1/2})}})).$ Subtracting one from the other we get
\begin{align}
\label{e.fourterms}
\E \left[G_{n}(X^{({\rm{B}})}_{\sT_{({\rm{B}},\varsigma_0)}}) - G_{n}(X^{({\rm{R}})}_{\sT_{({\rm{R}},\varsigma_{1/2})}})\right]	&= \E  \left[G_{n}(X^{({\rm{B}})}_0) - \sum_{k=0}^{\sT_{({\rm{B}},\varsigma_0)}-1} \Delta G_{n}(X^{({\rm{B}})}_k)\right. \\
\nonumber
			&\quad\left. - G_{n}(X^{({\rm{R}})}_0) + \sum_{k=0}^{\sT_{({\rm{B}},\varsigma_{1/2})}-1} \Delta G_{n}(X^{({\rm{R}})}_k)\right].
			\end{align}

Now
\begin{equation}\label{enrdiff1}
 \E (G_{n}(X^{({\rm{B}})}_0) - G_{n}(X^{({\rm{R}})}_0)) = \sum_{x \in \U_n} G_{n}(x) ({\rm{Ind}}_{\rm{B}}(x) - {\rm{Ind}}_{\rm{R}}(x)) = \sum_{x\in \U_n}G_n(x)\Delta G_n(x)= \frac{1}{4}\mathcal{E}(\nabla G_n), 
 \end{equation}
where all the equalities but the last  are by definition. The last equality is by Lemma \ref{byparts} and the fact that all vertices in $\U_{{\rm{B}},n}\cup \U_{{\rm{R}},n}$ have degree $4$.
Also
 $$\mathbb{E}_{x}\left[\sum_{k=0}^{\sT_{({\rm{B}},\varsigma_0)}-1} \Delta G_{n}(X_k)\right]=\mathbb{E}_{x}\left[\sum_{k=0}^{\sT_{({\rm{B}},\varsigma_0)}-1}{ \rm{Ind}}_{\rm{B}}(X_{k})\right]=G_{{\rm{B}},n}(\varsigma_0)(x),$$
where the first equality is by \eqref{newlapl1} and the second equality is by  \eqref{stoppedgreenold}. 
Similarly for any $\varsigma_{1/2}$ we get
  $$\mathbb{E}_{y}\left[\sum_{k=0}^{\sT_{({\rm{R}},\varsigma_{1/2})}-1} \Delta G_{n}(X_k)\right]=-G_{{\rm{R}},n}(\varsigma_{1/2})(y).$$
Thus we have
$
\E[\sum_{k=0}^{\sT_{({\rm{B}},\varsigma_0)}-1} \Delta G_{n}(X^{({\rm{B}})}_k)]=\E[G_{{\rm{B}},n}(\varsigma_0)(X^{({\rm{B}})}_0)],
$
and 
$\E[\sum_{k=0}^{\sT_{({\rm{R}},\varsigma_{1/2})}-1} \Delta G_{n}(X^{(\rm{R})}_k)]=\E[G_{{\rm{R}},n}(\varsigma_{1/2})(X^{(\rm{R})}_0)].$

Now putting things together we get,
\begin{align}\label{fincrucial}
\E_{\sigma}({\rm{W}}(\varsigma_1)-{\rm{W}}(\varsigma_{0}))&=  \sum_{x\in\U_{{\rm{B}},n}}G_{n}(x){\rm{Ind}}_{\rm{B}}(x) -\sum_{x\in \U_{{\rm{R}},n}}G_{n}(x){\rm{Ind}}_{\rm{R}}(x) \\
\nonumber
&\quad - \sum_{x\in \U_{{\rm{B}},n}}G_{{\rm{B}},n}(\varsigma_0)(x){\rm{Ind}}_{\rm{B}}(x) -\E[\sum_{x\in \U_{{\rm{R}},n}}G_{{\rm{R}},n}(\varsigma_{1/2})(x){\rm{Ind}}_{\rm{R}}(x)] , 
\end{align}
\begin{align*}
\quad  \ge   \sum_{x\in \U_{{\rm{B}},n}}G_{n}(x){\rm{Ind}}_{{\rm{B}}}(x) -\sum_{x\in \U_{{\rm{R}},n}}G_{n}(x){\rm{Ind}}_{\rm{R}}(x) 
\end{align*}
\begin{align*}
 \quad \quad\quad\quad\quad\quad\quad\quad\quad-\sum_{x\in \U_{{\rm{B}},n}}G_{{\rm{B}},n}(\varsigma_0)(x){\rm{Ind}}_{\rm{B}}(x) -\sum_{x\in \U_{{\rm{R}},n}}G_{{\rm{R}},n}(\varsigma_{0})(x){\rm{Ind}}_{\rm{R}}(x), 
\end{align*}
where for the inequality we replace the last term using Lemma \ref{easecrucial1}.
Now by Lemma \ref{byparts} and \eqref{newlapl1},
\begin{eqnarray*}
\sum_{x \in \U_{{\rm{B}},n}}G_{{\rm{B}},n}(x){\rm{Ind}}_{\rm{B}}(x)= \frac{1}{4}\mathcal{E}(\nabla G_{{\rm{B}},n}),\\
\sum_{x \in \U_{{\rm{R}},n}}G_{{\rm{R}},n}(x){\rm{Ind}}_{\rm{R}}(x) = \frac{1}{4}\mathcal{E}(\nabla G_{{\rm{R}},n}).
\end{eqnarray*}
Thus we are done applying the above and \eqref{enrdiff1} in \eqref{fincrucial}.
\end{proof}

\subsection{Weaker version of Theorem \ref{negdrift123}}\label{weak102}
As already mentioned the proof of Theorem \ref{negdrift123} follows from Lemma \ref{energyint} by lower bounding $[\mathcal{E}(\nabla G_{n})- \mathcal{E}(\nabla G_{{\rm{B}},n}(\varsigma_0))-\mathcal{E}(\nabla G_{{\rm{R}},n}(\varsigma_0))],$ by a positive number independent of $n.$ Towards achieving that, in this section we prove the weaker result that this quantity is at least $-\frac{C}{n^{\kappa}}$ for some constants $C, \kappa$ and hence is almost non-negative. Subsequently, to complete the proof of Theorem \ref{negdrift123} we will refine ideas from this section.
\begin{lem}\label{weak201}There exists $C,\kappa>0$ depending on $\delta,\U$ such that if $\varsigma_0=\sigma$ satisfies Assumption \ref{ass2}, then for all large enough $n> N_0(\delta),$
\begin{equation} 
[\mathcal{E}(\nabla G_{n})- \mathcal{E}(\nabla G_{{\rm{B}},n}(\varsigma_0))-\mathcal{E}(\nabla G_{{\rm{R}},n}(\varsigma_0))] \ge -\frac{C}{n^{\kappa}}.
\end{equation}
\end{lem}
To proceed, we need a few more definitions. 
Recall the standard gluing operation on any multigraph  $G=(V,E),$ where certain subsets of vertices are identified (``glued") and they act as a single vertex. Any edge between two identified vertices now act as a self loop in the glued graph.  For more on glued graphs see \cite{Lp1}.
In the sequel, we often specify a graph by just referring to the set of vertices which will be a subset of the vertex set of $\U_n.$ The implicit graph in that case would be the sub-graph induced by the graph structure on  $\U_n.$ 

Since the blue random walk is stopped whenever it exits $R_{\rm{B}},$ one can think of it as the random walk on the graph obtained by gluing all the outer boundary points of $R_{\rm{B}}$ to a single point, and the random walk is stopped whenever it hits this point. Similarly, one can consider a glued graph corresponding to the red random walk.  

Given the above discussion the basic approach to prove Lemma \ref{weak201} is the following: Using Lemma \ref{var}, we  consider a flow $\theta$ on the entire graph $\U_n$ corresponding to $\mathcal{E}(\nabla G_{n})$ and then restrict it to the graphs obtained by gluing outer boundary points of $R_{\rm{B}}$ and $R_{\rm{R}}$ respectively. The restrictions will be legitimate flows satisfying appropriate divergence conditions on the reduced graphs, and hence by Lemma \ref{var}, the energy of the restrictions will be upper bounds on $\mathcal{E}(\nabla G_{{\rm{B}},n}(\varsigma_0))$  and $\mathcal{E}(\nabla G_{{\rm{B}},n}(\varsigma_0))$ respectively. 
At this point, it would have been rather convenient if the glued graphs corresponding to the blue and red walkers were disjoint, because then the energy of $\theta$ would be at least as large as the sum of the energies of the restrictions and the RHS in Lemma \ref{weak201} could be replaced by zero.  Unfortunately,  in some cases the boundary of $R_{{\rm{B}}}$ can intersect $R_{\rm{R}}$ and vice versa. However, this issue does not arise if the inner boundary of $R_{\rm{B}}$ is glued instead  (see Remark \ref{nonneg200} for a more formal discussion). We provide two definitions of glued graphs below corresponding to gluing the outer and inner boundaries respectively.  We then show that working with either definition is essentially the same  up to an error of $O(\frac{1}{n^\kappa})$ which completes the proof.
 
Also, for technical reasons we glue slightly different subsets than described above.   The formal definitions are provided below for which we need to recall the definitions of $\mathring{R}_{\rm{B}},\mathring{R}_{\rm{R}}$  from \eqref{fillholes} and the inner and outer boundaries  $\partial_{in}$ and $\partial_{out}$  from \eqref{bdry1}.
\vspace{.1in}

\begin{definition}\label{crucdefbdry}
(i).
For ${\rm{C}}\in \{{\rm{B}},\rm{R}\},$ let $\mathring{R}^{out}_{\rm{C}}$ denote the graph with vertex set $\mathring{R}_{\rm{C}} \cup \partial_{out} \mathring{R}_{\rm{C}}$ with the vertex subset 
$\left\{\mathring{R}_{\rm{C}} \cup \partial_{out} \mathring{R}_{\rm{C}}\right \}\setminus R_{\rm{C}}$
glued.\\
(ii).
For ${\rm{C}}\in \{{\rm{B}},\rm{R}\},$ let $\mathring{R}^{in}_{\rm{C}}$ denote the graph with vertex set 
$\mathring{R}_{\rm{C}}$  with the vertex subset \\
$\{\mathring{R}_{\rm{C}}\setminus R_{\rm{C}}\} \cup \partial_{in} \mathring{R}_{\rm{C}}$
glued.
\end{definition}

\vspace{.1in}

For both the definitions above we denote the set of glued vertices now acting as a single vertex  by $\mathfrak{v}_{out}$ and $\mathfrak{v}_{in}$ respectively (whether the underlying graph  corresponds to ${\rm{C}}={\rm{B}}$ or ${\rm{C}}={\rm{R}},$ would be clear from context, and hence we  suppress the dependence on ${\rm{C}}$).  
With the above definitions, for ${\rm{C}}\in \{{\rm{B}},\rm{R}\},$ we can think of $\nabla(G_{{\rm{C}},n})$ as a flow on $\mathring{R}^{out}_{\rm{C}}.$ Moreover, by \eqref{newlapl1} and \eqref{newlapl2},
\begin{eqnarray}\label{newlapl-1}
\Delta({G}_{{\rm{C}},n})(x)&=& {\rm{Ind}}_{\rm{C}}(x) \mbox{ for all } x \in {R}_{\rm{C}},  \\
\label{newlapl-2}
{G}_{{\rm{C}},n}({\mathfrak{v}_{out}})&=& 0.
\end{eqnarray}
It would be notationally convenient to define $
 \mathcal{E}:=\mathcal{E}(\nabla G_{n}),$
and similarly  for ${\rm{C}}\in \{{\rm{B}},\rm{R}\},$ let,
\begin{equation}
\label{abbrev1}
 \mathcal{E}_{\rm{C}}:=\mathcal{E}(\nabla G_{{\rm{C}},n}).
\end{equation}
In terms of the above abbreviations, the statement of Lemma \ref{weak201} becomes, \begin{equation}\label{abbrev4}
 \mathcal{E}- \mathcal{E}_{\rm{B}}-\mathcal{E}_{\rm{R}} \ge -\frac{C}{n^\kappa}.
\end{equation}
For the proof, we  need another variant of the Green functions.
Analogous to $G_{{\rm{C}},n},$ for ${\rm{C}}\in \{{\rm{B}},\rm{R}\},$ 
we define ${G}^{*}_{{\rm{C}},n}$ to be the function on $\mathring{R}^{in}_{\rm{C}}$ such that, 
\begin{equation}\label{stoppedgreennew}
G^{*}_{{\rm{C}},n}(x) = \E_x \left[\sum_{k=0}^\infty {\rm{Ind}}_{\rm{C}}(X_{k})\mathbf{1}(k<\tau(R^{c}_{\rm{C}}\cup \partial_{in} \mathring{R}_{\rm{C}}))\right],  
\end{equation}
i.e., the Green function for the random walk stopped on hitting $\U_n\setminus R_{\rm{C}}$ or the boundary $\partial_{in} \mathring{R}_{\rm{C}}.$
It is easy to verify that, 
\begin{eqnarray}\label{lapl3}
\Delta({G}^{*}_{{\rm{C}},n})(x)&=& {\rm{Ind}}_{\rm{C}}(x), \mbox{ for all } x \in {R}_{\rm{C}} \setminus \partial_{in} \mathring{R}_{\rm{C}} ,\\
\label{lapl4}
{G}^{*}_{{\rm{C}},n}({\mathfrak{v}_{in}})&=& 0.
\end{eqnarray}

Notice that  by \eqref{divsum}, the above information determines the Laplacian  at $\mathfrak{v}_{out}$ and $\mathfrak{v}_{in}$ in the glued graphs $\mathring{R}^{out}_{\rm{C}}$,  $\mathring{R}^{in}_{\rm{C}}$ respectively. Also let, 
\begin{equation}
\label{abbrev3}
\mathcal{E}^*_{\rm{C}}:=\mathcal{E}(\nabla {G}^{*}_{{\rm{C}},n}).
\end{equation}

\begin{remark}\label{enermin1} Notice that Thompson's principle \cite[Theorem 9.10]{lpw} stated in  Lemma \ref{var}, implies that for ${\rm{C}}\in \{{\rm{B}},\rm{R}\}$:
\begin{itemize}
\item [i.]$\displaystyle{\mathcal{E}_{\rm{C}} = \min_{\theta} \mathcal{E}(\theta)},$ where the minimum is taken over all flows $\theta$ on $\mathring{R}^{out}_{\rm{C}}$  such that $\div \theta(x) = 4{\rm{Ind}}_{\rm{C}}(x),$ for all $x \in R_{\rm{C}}$. 
\item [ii.]$\displaystyle{\mathcal{E}^*_{\rm{C}}= \min_{\theta} \mathcal{E}(\theta)}$ where the minimum is taken over all flows $\theta$ on $\mathring{R}^{in}_{\rm{C}}$ such that $\div \theta (x) = 4{\rm{Ind}}_{\rm{C}}(x),$ for all $x \in R_{\rm{C}} \setminus  \partial _{in}\mathring{R}_{\rm{C}} $.
\end{itemize}
 \end{remark}
With the preceding preparation, the proof of Lemma \ref{weak201}  follows from a series of lemmas that we state next.
We first state an easy monotonicity result.

\begin{lem}\label{mon1}If $\sigma\in \Omega$ satisfies Assumption \ref{ass2}, then for ${\rm{C}}\in \{{\rm{B}},\rm{R}\},$
$$\mathcal{E}^*_{\rm{C}}\le {\mathcal{E}}_{\rm{C}} ,$$
where $\mathcal{E}_{\rm{C}},  \mathcal{E}^*_{\rm{C}}$ are defined in \eqref{abbrev1} and \eqref{abbrev3} respectively. 
\end{lem}
\begin{proof}
Without loss of generality, we take ${\rm{C}}={\rm{B}}$. By Lemma \ref{byparts} and \eqref{newlapl-1} we have
${\mathcal{E}}_{\rm{B}}=4\sum_{x \in \mathring{R}^{out}_{\rm{B}}} G_{{\rm{B}},n}(x) {\rm{Ind}}_{\rm{B}}(x)$, and similarly by \eqref{lapl3}, we have 
${\mathcal{E}}^*_{\rm{B}}=4\sum_{x \in \mathring{R}^{in}_{\rm{B}}}G^*_{{\rm{B}},n}(x) {\rm{Ind}}_{\rm{B}}(x).$
The proof is now complete as, 
\begin{equation}\label{mon23}
G^*_{{\rm{B}},n}\le G_{{\rm{B}},n},
\end{equation}
 which immediately follows from the definitions of $G_{{\rm{B}},n},G^*_{{\rm{B}},n}$ and the simple fact that,
$\tau((\U_n \setminus R_{\rm{B}}) \cup \partial_{in}\mathring{R}_{\rm{B}})\le \tau(\U_n \setminus R_{\rm{B}}).$
\end{proof}

The next lemma shows that  even though $\mathcal{E}_{\rm{C}}\ge \mathcal{E}^*_{\rm{C}},$  they are close to each other. 

\begin{lem}\label{overlap}There exists $\kappa>0,$ such that if $\sigma\in \Omega$ satisfies Assumption \ref{ass2}, then for ${\rm{C}}\in \{{\rm{B}},\rm{R}\},$ $$|{\mathcal{E}}_{\rm{C}}-\mathcal{E}^*_{\rm{C}}|=O(\frac{1}{n^{\kappa}}),$$
where the constant in the $O(\cdot)$ depends on $\dd$ appearing in Section \ref{para}.
\end{lem}
The next lemma states that the sum of the voltage differences across the glued graphs is at most that of the original graph.
\begin{lem}\label{res1}If $\sigma\in \Omega$ satisfies Assumption \ref{ass2}, then,
$$\mathcal{E}-\mathcal{E}^*_{\rm{B}}-\mathcal{E}_{\rm{R}}\ge 0.$$ 
\end{lem} 
The proof of Lemma \ref{weak201} now follows easily from the above results.
\begin{proof}[Proof of Lemma \ref{weak201}]
From the above statements it follows that
\begin{equation}\label{loss}
\mathcal{E}-{\mathcal{E}}_{\rm{B}}-\mathcal{E}_{\rm{R}}\ge\mathcal{E}-\mathcal{E}^*_{\rm{B}}-\mathcal{E}_{\rm{R}}-\frac{C}{n^{\kappa}} \ge -\frac{C}{n^{\kappa}},  
\end{equation}
for some constant $C=C(\delta),$ and hence we are done. 
\end{proof}

 We now provide the proofs of Lemmas \ref{overlap} and \ref{res1}. We start with  the latter which has a shorter proof.  
\begin{proof}[Proof of Lemma \ref{res1}] Let us take the optimal flow $\theta$ for $\mathcal{E}$ on $\U_n,$ i.e., 
\begin{equation}\label{optimal1}
\theta=\nabla G_n.
\end{equation} 
For ${\rm{C}}\in \{{\rm{B}},\rm{R}\}$, let,  $
\tilde \theta_{\rm{C}} :=  \theta\mid_{\mathring{R}^{out}_{\rm{C}}}$  and $\tilde \theta^*_{\rm{C}} := \theta\mid_{\mathring{R}^{in}_{\rm{C}}}$
be the restrictions of $\theta$ on $\mathring{R}^{out}_{\rm{C}}$ and $\mathring{R}^{in}_{\rm{C}}$ respectively.
Now by \eqref{reword} and Remark \ref{interprop}  it follows that, 
\begin{eqnarray}
\div(\tilde \theta_{\rm{C}})(x)  &=& 4 {\rm{Ind}}_{\rm{C}}(x) \mbox{ for all }x \in {R}_{\rm{C}},\\ 
\div(\tilde \theta^*_{\rm{C}})(x) &=& 4 {\rm{Ind}}_{\rm{C}}(x) \mbox{ for all } x \in R_{\rm{C}}\setminus  \partial_{in} \mathring{R}_{\rm{C}},
\end{eqnarray}
for ${\rm{C}}\in \{{\rm{B}},\rm{R}\}$, where ${\rm{Ind}}_{\rm{C}}(\cdot)$ was defined in \eqref{not1} and \eqref{not2}.
Thus the flows $\tilde\theta_{\rm{C}}$ and $\tilde\theta^*_{\rm{C}}$ on $\mathring{R}^{out}_{\rm{C}}$ and $\mathring{R}^{in}_{\rm{C}}$ have the same divergence as $\nabla G_{{\rm{C}},n}$ and $\nabla G^*_{{\rm{C}},n}$ respectively.
Hence by Remark \ref{enermin1}
 $${\mathcal{E}}_{\rm{C}}\le \mathcal{E}(\tilde \theta_{\rm{C}}) \text{ and } {\mathcal{E}}^*_{\rm{C}}\le \mathcal{E}(\tilde \theta^*_{\rm{C}}).$$

Recall $C_{{\rm{B}},m}$ from \eqref{fillholes}. Since $\mathring{R}_{\rm{B}}=\U_{n}\setminus C_{{\rm{B}},m}$ is connected and 
$\{R_{\rm{R}}\bigcup A_{{\rm{R}},n}\}\subseteq C_{{\rm{B}},m},$ it follows by definition that $\mathring{R}_{\rm{R}}\subset C_{{\rm{B}},m}.$
Thus, the edge set of the graphs  $\mathring{R}^{in}_{\rm{B}},\mathring{R}^{out}_{\rm{R}}$ are disjoint. 
Therefore,
\begin{equation}\label{nonneg}
\mathcal{E}=\mathcal{E}(\theta)\ge \mathcal{E}(\theta\mid_{\mathring{R}^{in}_{\rm{B}}})+\mathcal{E}(\theta\mid_{\mathring{R}^{out}_{\rm{R}}})=\mathcal{E}(\tilde \theta^*_{\rm{B}})+\mathcal{E}(\tilde \theta_{\rm{R}})\ge {\mathcal{E}}^*_{\rm{B}}+ {\mathcal{E}}_{\rm{R}}.
\end{equation}
Hence we are done.
\end{proof} 
\begin{remark}\label{nonneg200}Notice that \eqref{nonneg} would not have been true if we worked with  $\mathring{R}^{out}_{\rm{B}}$ and $\mathring{R}^{out}_{\rm{R}}$, since they do not necessarily have disjoint edges. This is the reason for gluing the inner boundary in one of the sets. \end{remark}

To prove Lemma \ref{overlap}, we need a few preliminary lemmas first. Let
\begin{eqnarray}
\label{opt12} 
\theta_* & = &\nabla ({G}^{*}_{{\rm{B}},n}).
\end{eqnarray}
Hence by Remark \ref{enermin1}, ii., $\theta_*$ has the minimum energy among all flows on ${\mathring{R}^{in}_{\rm{B}}}$ with the same divergence as itself.
Even  though ${\mathring{R}^{in}_{\rm{B}}}$ is obtained from a subgraph of $\U_n$ by identifying some vertices, we  denote the edges of this graph by the corresponding edges of $\U_n.$  
The next lemma is a conservation of flow result. It says that the total flow across $\partial_{in} \mathring{R}_{\rm{B}}$ is at most a constant.
\begin{lem}\label{sumbound} For $\theta_*$ as above,
$$\sum_{{x\in\partial_{in} \mathring{R}_{\rm{B}}}}\sum_{ {
y\sim x}\atop {y \in {R}_{\rm{B}}\setminus \partial_{in} \mathring{R}_{\rm{B}}}}{\theta}_{*}(x,y)\le 4.
$$
\end{lem}
\begin{proof}
We start with the observation that for all such $x,y$ as in the sum, $G^*_{{\rm{B}},n}(x)=0$ and $G^*_{{\rm{B}},n}(y)\ge 0$ which follows by definition. Hence ${\theta}_{*}(x,y)\ge 0$ for all such $(x,y)$. 
Now by \eqref{divsum}  we have 
 $\sum_{z \in \mathring{R}^{in}_{\rm{B}}} \div(\theta_*)(z)=0.$
 By \eqref{lapl3} for any $z \in \mathring{R}^{in}_{\rm{B}}$ such that $z \in R_{\rm{B}}\setminus \partial_{in}\mathring{R}_{\rm{B}}$ we have $\div(\theta_*)(z)=4{\rm{Ind}}_{\rm{B}}(z).$
 Note that even though the graph is not regular, the $4$ appears since every $z$ as above, in the support of ${\rm{Ind}}_{\rm{B}}(\cdot)$ has degree $4$. 
Moreover, by definition,  $$\sum_{z \in \mathring{R}^{in}_{\rm{B}}} \div(\theta_*)(z)=\sum_{z\in R_{\rm{B}}\setminus \partial_{in}\mathring{R}_{\rm{B}}}4{\rm{Ind}}_{\rm{B}}(z)+ \sum_{x\in\partial_{in}\mathring{R}_{\rm{B}}\bigcup (\U_n \setminus R_{\rm{B}})}\sum_{ {
y\sim x}\atop {y \in {R}_{\rm{B}} \setminus \partial_{in} \mathring{R}_{\rm{B}}}}{\theta}_{*}(y,x).$$ The LHS is $0$ and the first sum on the RHS is at most $4$. Since $\theta_*$ is anti-symmetric, and all the terms in the second sum above are non-positive, we are done .
\end{proof}

We state two more lemmas.  The first result says that for any connected subset $A$ of $\U_n$ with large enough diameter which is at a certain distance away from $\U_{{\rm{B}},n},$ the probability that a random walk started from a neighboring site of $A$ hits $\U_{{\rm{B}},n}$ before hitting $A$ decays as a power law in $n$. The result is well known on the whole lattice. The proof in our case with the  necessary adaptions is provided in \cite{tc1}. Let  ${\rm{diam}}(\cdot)$ denote the diameter in terms of the Euclidean metric.
\begin{lem}\label{hitprob2}\cite[Lemma 5.1]{tc1} Fix $c>0$. Let $A\subset \U_n$ be $*-$connected (Definition \ref{conndef1}) and suppose that $\min({\rm{diam}}(A),d(\U_{\mathrm{B},n},A))\ge c$.
Then for all large $n$,  $$\sup_{x\sim A}\mathbb{P}_{x}(\tau(\U_{{\mathrm{B}},n})\le \tau(A))\le \frac{C}{n^{\kappa}},$$ for some positive $\kappa,C$ depending only on $c$ and $\U.$ Here $x \sim A$ means that $x \notin A$ and there exists $y \in A$ such that $x$ is a neighbor of $y.$
\end{lem}

The next result is a basic isoperimetric inequality saying that the boundary of the set $\mathring{R}_{\rm{B}}$ is not small. 
\begin{lem}\label{diamlb}There exists $c=c(\dd,\U)>0,$ such that $$\min ({\rm{diam}}(\partial_{in} \mathring{R}_{\rm{B}}), {\rm{diam}}(\partial_{out} \mathring{R}_{\rm{B}}))>c.$$ 
\end{lem}
\begin{proof} 
Clearly it suffices to just show that ${\rm{diam}}(\partial_{in} \mathring{R}_{\rm{B}})>c$ since every vertex in $\partial_{in} \mathring{R}_{\rm{B}}$ has a neighbor in $\partial_{out} \mathring{R}_{\rm{B}}.$
We  consider only the disc $\D,$ since the  general proof now follows by using the bi-Lipschitz nature of the conformal map $\phi$ in \eqref{confmap1}. 
However the proof for the disc follows by standard isoperimetric inequalities on the plane once we observe that, $\D_{{\rm{B}},n}\subset \mathring{R}_{\rm{B}}$ and  $\D_{{\rm{R}},n}\subset \D_n \setminus \mathring{R}_{\rm{B}}$ which imply that both $\mathring{R}_{\rm{B}}$ and its complement contain at least $\delta^{2}n^2$ vertices. We omit the details.
\end{proof}

To prove Lemma \ref{overlap}, the following will be our strategy: Without loss of generality, we take ${\rm{C}}={\rm{B}}$. Recall from  \eqref{opt12}, that $\theta_*$ is a flow on $\mathring{R}^{in}_{\rm{B}}$.  
We derive a flow $\theta_{\rm{B}}$ on $\mathring{R}^{out}_{\rm{B}}$  from $\theta_*,$ by  specifying the value of the flow on the edges from $\partial_{in} \mathring{R}_{\rm{B}}$ to $\partial_{out} \mathring{R}_{\rm{B}}$  (the extra edges in $\mathring{R}^{out}_{\rm{B}}$ not in $\mathring{R}^{in}_{\rm{B}}$) and keeping the value of the flow on all other edges same as $\theta_{*}.$ By construction
$\theta_{\rm{B}}$ will have the same divergence as $\nabla G_{{\rm{B}},n}$. Thus 
\begin{equation}\label{ineq22}
\mathcal{E}^*_{\rm{B}}\le\mathcal{E}_{\rm{B}}\le \mathcal{E}(\theta_{\rm{B}}),
\end{equation} where the first inequality is by Lemma \ref{mon1} and the second inequality by Remark \ref{enermin1} i.
We  then show that, $$\mathcal{E}(\theta_{\rm{B}})-\mathcal{E}^*_{\rm{B}}=O(\frac{1}{n^{\kappa}}),$$ for some $\kappa>0,$ which  completes the proof.
Before jumping into the proof of Lemma \ref{overlap}, we need  a crude upper bound for the stopped Green functions $G_{{\rm{C}},n}$. 
\begin{lem}\label{remupbound} For $\sigma\in \Omega\cup \Omega',$ for all small enough $\delta$ (depending on $\U$),
$$\limsup_{{n=2^m}\atop{m\to \infty}}\sup_{x\in \U_{{\mathrm{B}},n}}G_{{\rm{B}},n}(x)=O\left(\frac{1}{\dd^2}\right).$$
A similar result holds for $G_{2,n}$ as well. 
\end{lem}
A sharper upper bound of $O(\log(\frac{1}{\dd}))$ can be obtained with a little more work. However, since we do not try to obtain optimal dependence of our bounds on $\delta,$  this will suffice. 
\begin{proof}
Recall the definitions from \eqref{region} and $\sT_{({\rm{B}},\sigma)}$ from \eqref{stoptime}. Since $R_{\rm{B}} \subseteq S_{\rm{B}},$ it follows that,  $\sT_{({\rm{B}},\sigma)})\le \tau(\U_n \setminus S_{\rm{B}}). $ 
As a simple consequence of the fact that the total variation mixing time of the discrete time random walk on $\U_n$ is $O(n^2),$ (for a proof see \cite[Lemma 5.4]{tc1}) and the fact that both $\rm{S}_{\rm{B}}$ and $\rm{S}_{\rm{R}}$ have $\Theta(n^2)$ vertices,  it follows that,
$\sup_{x\in \U_n}\E(\tau (S_{{{\rm{C}}}}))=O(n^2)$  for ${\rm{C}}\in \{{\rm{B}},\rm{R}\}$ and hence for ${\rm{C}}\in \{{\rm{B}},\rm{R}\},$
$$\sup_{x\in \U_n,\sigma \in \Omega \cup \Omega'}\E(\sT_{({\rm{C}},\sigma)})=O(n^2).$$ Thus the lemma follows using the crude bound, $$G_{{\rm{B}},n}(x)\le \frac{1}{|\U_{{\rm{B}},n}|}\E_x(\sT_{({\rm{B}},\sigma)}).$$
\end{proof}

\begin{proof}[Proof of Lemma \ref{overlap}]
We  construct a flow $\theta_{\rm{B}}$ on $\mathring{R}^{out}_{\rm{B}}$:
For each $y$ in $R_{\rm{B}} \cap \partial_{in} \mathring{R}_{\rm{B}}$ choose $x \in \partial_{out} \mathring{R}_{\rm{B}}$ such that $y\sim x$ (such an $x$ exists by definition of $\partial_{in} \mathring{R}_{\rm{B}}$).
Let, 
\begin{align*}
\theta_{\rm{B}}(y,x)&=\sum_{\stackrel{z \in {R}_{\rm{B}}}{z\sim y}}\theta_{*}(z,y), \text{ for each } y \in R_{\rm{B}} \cap \partial_{in} \mathring{R}_{\rm{B}},\\
\theta_{\rm{B}} (y',x')&= 0\, \text{ for all other edges with } y' \text { in } \partial_{in} \mathring{R}_{\rm{B}}, x' \in \partial_{out} \mathring{R}_{\rm{B}},\\
&= \theta_{*} \text{ everywhere else}.
\end{align*}
We first claim that 
\begin{equation}\label{trivial}
\mathcal{E}_{\rm{B}}\le \mathcal{E}(\theta_{\rm{B}}). 
\end{equation}
By Remark \ref{enermin1} i. this will follow if we can show for all $y \in R_{\rm{B}}$
\begin{equation}\label{verify1}
\div(\theta_{\rm{B}})(y)=d_y(\Delta G_{{\rm{B}},n})(y)=4{\rm{Ind}}_{\rm{B}}(y),
\end{equation}
where $d_y$ is the degree of $y$ in $\mathring{R}_{\rm{B}}^{out}$.
We know by definition that 
$$\div(\theta_*)(y)=d_y(\Delta {G}^{*}_{{\rm{B}},n})(y)=4{\rm{Ind}}_{\rm{B}}(y) \mbox{ for all }y \in {R}_{\rm{B}}\setminus  \partial_{in} \mathring{R}_{\rm{B}}.$$
Also, by construction $\theta_{\rm{B}}=\theta_*$ on all the edges except the boundary edges of $ \partial_{in} \mathring{R}_{\rm{B}}.$ 
Now by \eqref{fillholes}
\begin{equation}\label{distcond}
d(\U_{{\rm{B}},n},\partial_{in} \mathring{R}_{\rm{B}})\ge \dd_1/4.
\end{equation}
 Thus,
 ${\rm{Ind}}_{\rm{B}}\mid_{\partial_{in} \mathring{R}_{\rm{B}}}=0$. Hence to verify \eqref{verify1}, it suffices to show that 
$\div(\theta_{\rm{B}})\mid_{\partial_{in} \mathring{R}_{\rm{B}}\cap R_{\rm{B}}}=0.$
Let $y\in \partial_{in}\mathring{R}_{\rm{B}}\cap R_{\rm{B}}.$
By construction, we know that there exists exactly one $x\in\partial_{out}\mathring{R}_{\rm{B}}$ such that $\theta_{\rm{B}}(y,x)\neq 0.$
Also,  $$\theta_{\rm{B}}(y,x)=\mathop{\sum_{z\in {R}_{\rm{B}}}}_{ z\sim y}\theta_{*}(z,y)=\mathop{\sum_{z\in {R}_{\rm{B}}}}_{ z\sim y}\theta_{\rm{B}}(z,y),$$ where the first equality is by definition and for the second inequality we use the fact that $\theta_{\rm{B}}=\theta_*,$ for all the edges in the sum by construction.
Thus, $$\div(\theta_{\rm{B}})(y)=\mathop{\sum_{w\in \mathring{R}^{out}_{\rm{B}}}}_{ w\sim y}\theta_{\rm{B}}(w,y)=0,$$ 
and \eqref{trivial} is verified. By \eqref{ineq22}, the proof of the lemma will be complete once we show that $$\mathcal{E}(\theta_{\rm{B}})-\mathcal{E}^*_{\rm{B}} = O(\frac{1}{n^{\kappa}}).$$
We claim that 
\begin{equation}\label{conclude}
 \sup_{y\in \partial_{in}\mathring{R}_{\rm{B}}}\sum_{{z\in R_{\rm{B}}}\atop{z\sim y}}\theta_{*}(y,z)=O(\frac{1}{n^{\kappa}}).
\end{equation} 
 First notice that $\partial_{in}\mathring{R}_{\rm{B}} $ satisfies the hypotheses of Lemma \ref{hitprob2}: 
\begin{itemize}
\item By \eqref{distcond}, $d(\partial_{in}\mathring{R}_{\rm{B}},\U_{\rm{B}})\ge \dd_1/4.$
\item The connectedness hypothesis is satisfied by Remark \ref{connrem}.
\item The diameter lower bound follows from Lemma \ref{diamlb}.
\end{itemize}
Thus by Lemma \ref{hitprob2} for any $y\in \partial_{in}\mathring{R}_{\rm{B}}$ and $z \sim y,$ $$\mathbb{P}_z(\tau(\U_{{\rm{B}},n})< \tau(\partial_{in}\mathring{R}_{\rm{B}}))=O \left( \frac{1}{n^{\kappa}}\right),$$ where the constant in the $O(\cdot)$ term depends on $\dd_1$ and $\U.$  
Since ${\rm{Ind}}_{\rm{B}}(\cdot)$ is positive only on $\U_{{\rm{B}},n}$ and $0$ everywhere else, by \eqref{stoppedgreennew} it follows that
\begin{eqnarray*}
G^*_{{\rm{B}},n}(z)&\le & \mathbb{P}_z \left(\tau(\U_{{\rm{B}},n})<  \tau(\partial_{in}\mathring{R}_{\rm{B}})\right)\sup_{x \in \U_{{\rm{B}},n}}G^*_{{\rm{B}},n}(x),\\
& \le & \mathbb{P}_z\left(\tau(\U_{{\rm{B}},n})< \tau(\partial_{in}\mathring{R}_{\rm{B}})\right)\sup_{x \in \U_{{\rm{B}},n}}G_{{\rm{B}},n}(x).
\end{eqnarray*}
The last inequality follows since for all $x \in \U_n,$  by \eqref{mon23}, $G^*_{{\rm{B}},n}(x)\le G_{{\rm{B}},n}(x).$
Thus by Lemma \ref{remupbound} we get
$$\theta_{*}(y,z)=G^*_{{\rm{B}},n}(z)= O(\frac{1}{n^{\kappa}}),$$
where the first equality follows from the fact that $\theta_*=\nabla G^*_{{\rm{B}},n}$ and $G^*_{{\rm{B}},n}(y)=0.$
Since $\theta_{\rm{B}}=\theta_*$ on all but the boundary edges we have
\begin{eqnarray*}
\mathcal{E}(\theta_{\rm{B}})-\mathcal{E}(\theta_*)&=&O(\sum_{ y\in \partial_{in}\mathring{R}_{\rm{B}}}\bigl(\sum_{z\sim y}|\theta_{*}(z,y)|\bigr)^{2})\\
&\le & \bigl(\sup_{y\in \partial_{in}\mathring{R}_{\rm{B}}}\sum_{z\sim y}|\theta_{*}(z,y)|\bigr)\bigl(\sum_{y\in \partial_{in}\mathring{R}_{\rm{B}}}\sum_{z\sim y}|\theta_{*}(z,y)|\bigr)\\
& = & O(\frac{1}{n^\kappa}),
\end{eqnarray*}
where the last equality is due to Lemma \ref{sumbound} and \eqref{conclude}. 
Hence ${\mathcal{E}}_{\rm{B}}- \mathcal{E}^{*}_{\rm{B}}\stackrel{ \eqref{trivial}}{\le} \mathcal{E}(\theta_{\rm{B}})-\mathcal{E}(\theta_*)=O(\frac{1}{n^{\kappa}}).$
\end{proof}

\subsection{Quantitative version of Lemma \ref{res1}}\label{quant102}
The proof of Theorem \ref{negdrift123} demands the following stronger quantitative version of Lemma \ref{weak201} or equivalently  Lemma \ref{res1}. 

\begin{lem}\label{resquant1}For all $\sigma\notin \Omega_{(\e)},$ satisfying Assumption \ref{ass2}, 
$$\mathcal{E}-\mathcal{E}^*_{\rm{B}}-\mathcal{E}_{\rm{R}}\ge c,$$ for some $c=c(\e,\U)>0.$ 
\end{lem} 
The proof of Theorem \ref{negdrift123} is now immediate.
\begin{proof}[Proof of Theorem \ref{negdrift123}]
Under Assumption \ref{ass2}, the proof follows from Lemma \ref{energyint}, \eqref{loss} and  \eqref{resquant1}. The remaining cases were proved in Section \ref{efn}.
\end{proof}

We now prepare for the proof of Lemma \ref{resquant1}.
Let $\overset{wired}{\U_n}$ be the same graph as $\U_n$ but with the following set of vertices: 
\begin{equation}\label{glueset1} 
\partial_{in}\mathring{R}_{\rm{B}} \cup \{ x \in{\mathring{R}_{\rm{B}}} :\sigma(x)={\rm{R}}\} \cup \{ x \in{\U_n \setminus \mathring{R}_{\rm{B}}} :\sigma(x)={\rm{B}}\},
\end{equation}
glued. The points that are glued are all the red vertices in $\mathring{R}_{\rm{B}}$ (which cause the blue random walk to stop) and similarly all the blue vertices in $\U_n \setminus \mathring{R}_{\rm{B}}$.
Let $\mathfrak{w}$ denote the set of glued vertices which act as a single vertex in  $\overset{wired}{\U_n}$.
In the proof of Lemma \ref{res1} we took the flow $\nabla G_n$ on $\U_n$ and restricted it to $\mathring{R}^{in}_{\rm{B}}$ and $\mathring{R}^{out}_{\rm{R}}.$ To prove Lemma \ref{resquant1}, instead of restricting $\nabla G_n,$ we  construct another flow $\overset{wired}{\theta}$ on $\overset{wired}{\U_n}$ and then restrict it.  As a consequence, similar to Lemma \ref{res1} it will follow that $\cE(\overset{wired}{\theta})-\mathcal{E}^*_{\rm{B}}-\mathcal{E}_{\rm{R}}\ge0.$ Moreover, due to a  quantitative version of Rayleigh's monotonicity principle which is proved in Lemmas \ref{gain1} and \ref{disjointpath23}, it will follow that $\cE(\nabla G_n)-\cE(\overset{wired}{\theta})>c$.  Combining the above two, we would be done.
To proceed, recall the functions ${\rm{Ind}},{\rm{Ind}_{\rm{B}}},$ and ${\rm{Ind}_{\rm{R}}}$ from \eqref{not3}.
\begin{lem}\label{unimini}There exists a unique flow $\overset{wired}{\theta}$ on $\overset{wired}{\U_n},$ with the following properties:\\
\begin{align*}
\div(\overset{wired}{\theta})\mid_{\{R_{\rm{B}}\cup R_{\rm{R}}\}\setminus \mathfrak{w}} & =  4{\rm{Ind}}\mid_{\{R_{\rm{B}}\cup R_{\rm{R}}\}\setminus \mathfrak{w}},\\
\div (\overset{wired}{\theta})(\cdot) & =  0 \mbox{ for all other vertices except }  \mathfrak{w},\\
\overset{wired}{\theta}&=\argmin_{g} \mathcal{E}(g),
\end{align*}
 where in the last display the infimum is taken over all flows $g$ on $\overset{wired}{\U_n}$ which have the same divergence as $\overset{wired}{\theta}$.
\end{lem}

Note that using the above data and \eqref{divsum}, $\div (\overset{wired}{\theta})(\mathfrak{w})$ is determined exactly.

\begin{proof} For brevity, let us call  the set of flows $g$ on $\overset{wired}{\U_n}$ satisfying the above divergence conditions as $\mathcal{F}.$ That the set is non empty follows, since  by \eqref{reword} the flow $\nabla(G_n)\mid_{ \overset{wired}{\U_n}}$ is in this set. 
Now let
\begin{equation}\label{minenergy}
\overset{wired}{\mathcal{E}}:=\inf_{g \in \mathcal{F}} \mathcal{E}(g).
\end{equation}
By standard arguments using compactness (see proof of \cite[Theorem 9.10]{lpw}), there exists a unique minimizer $\overset{wired}\theta$ of the energy minimization problem.
\end{proof}

Under Assumption \ref{ass2}, the restrictions
\begin{eqnarray*}
  \overset{wired}{\theta_{\rm{B}}} &= & \overset{wired}{\theta}\mid_{\mathring{R}^{in}_{\rm{B}}},\\
  \overset{wired}{\theta_{\rm{R}}} &= & \overset{wired}{\theta}\mid_{\mathring{R}^{out}_{\rm{R}}},
\end{eqnarray*}  
 satisfy the following divergence conditions:
\begin{eqnarray}\label{resdiv}
  \div(\overset{wired}{\theta_{\rm{B}}})(x) &= & 4{\rm{Ind}}_{\rm{B}}(x) \text{ for all }x\in R_{\rm{B}}\setminus \partial_{in}\mathring{R}_{\rm{B}}, \\
  \nonumber
  \div(\overset{wired}{\theta_{\rm{R}}})(y) &= & 4{\rm{Ind}}_{\rm{R}}(y) \text{ for all }y\in R_{\rm{R}} . 
\end{eqnarray}  
To see this,  note that no vertex in either $R_{\rm{B}}\setminus \partial_{in}\mathring{R}_{\rm{B}}$ or $R_{\rm{R}}$ was in the set of glued vertices in \eqref{glueset1}.
Now by Remark \ref{interprop}, for ${\rm{C}}\in \{{\rm{B}},\rm{R}\},$ we have ${\rm{Ind}}(\cdot)\mid_{ R_{\rm{C}}}={\rm{Ind}}_{\rm{C}}(\cdot).$ Thus \eqref{resdiv} follows, since $\overset{wired}{\theta}$ by definition, satisfies the divergence conditions.

Moreover by Remark \ref{enermin1} we have  ${\mathcal{E}}^*_{\rm{B}}\le \mathcal{E}(\overset{wired}{\theta_{\rm{B}}}),\text{ and } {\mathcal{E}}_{\rm{R}}\le  \mathcal{E}(\overset{wired}{\theta_{\rm{R}}}).$

\begin{remark}\label{rem4}
Similar to \eqref{nonneg} we have, 
$$
\overset{wired}{\mathcal{E}}=\mathcal{E}(\overset{wired}{\theta})\ge \mathcal{E}(\overset{wired}{\theta}\mid_{\mathring{R}^{in}_{\rm{B}}})+ \mathcal{E}(\overset{wired}{\theta}\mid_{\mathring{R}^{out}_{\rm{R}}})=\mathcal{E}(\overset{wired}{\theta_{\rm{B}}})+\mathcal{E}(\overset{wired}{\theta_{\rm{R}}})\ge {\mathcal{E}}^*_{\rm{B}}+ {\mathcal{E}}_{\rm{R}}.
$$
Therefore, 
\begin{eqnarray*}
\mathcal{E}-{\mathcal{E}}_{\rm{B}}-\mathcal{E}_{\rm{R}}  \ge  \mathcal{E}-{\mathcal{E}}^*_{\rm{B}}-\mathcal{E}_{\rm{R}}-O(\frac{1}{n^{\kappa}})& = & \mathcal{E}- \overset{wired}{\mathcal{E}}+\left[\overset{wired}{\mathcal{E}}-\mathcal{E}_{\rm{B}}^* -{\mathcal{E}}_{\rm{R}}\right]-\frac{C}{n^{\kappa}},\\
&\ge & \mathcal{E}- \overset{wired}{\mathcal{E}}-\frac{C}{n^{\kappa}} ,
\end{eqnarray*}
where the first inequality follows from \eqref{loss}. 
Hence to prove Lemma  \ref{resquant1} by \eqref{abbrev4}  it suffices to show that 
\begin{equation}\label{proofstra1}
\mathcal{E}- \overset{wired}{\mathcal{E}}\ge c,
\end{equation}
for some constant $c=c(\e).$ 
\end{remark} 

Notice that $\mathcal{E}- \overset{wired}{\mathcal{E}}$ is nothing but the drop in voltage when points at various voltages are glued together to get $\overset{wired}{\U_n}$ from $\U_n$. The fact that the quantity is positive is the well known \emph{Rayleigh's  monotonicity principle}. For more details see \cite[Chap 2]{Lp1}. We now make it quantitative. To estimate such voltage drops we state and prove two technical lemmas. 
Let $\tilde\U_n$ be a graph obtained from $\U_n$ by gluing certain pairs of points. 
 Let $A$ be a set of $k$ pairs of points, $(z_1,z'_1),\ldots (z_k,z'_k),$ which are among the pairs glued to obtain $\tilde \U_n$. 
Recall from \eqref{optimal1}, that $\theta=\nabla G_n.$ Consider the restricted flow $\nabla G_n \mid_{\tilde \U_n}$ (the flow that $\nabla G_n$ induces on $\tilde \U_{n}$). 
Let $\tilde \theta$ denote the flow on $\tilde \U_{n}$ such that 
\begin{equation}\label{energymini}
\mathcal{E}(\tilde\theta)=\inf_{g}\{\mathcal{E}{(g)}\},
\end{equation}
 where the infimum is taken over all flows $g$ on $\tilde \U_n$ which has the same divergence as $\theta\mid_{\tilde \U_n}.$ Existence of $\tilde \theta $ follows by the argument already used in the proof of Lemma \ref{unimini}.

\begin{lem}\label{gain1} Let $A\subset \U_n,\,\, \theta$  and $\tilde \theta$ be as defined above.  Suppose that, for each $1\le i\le k,$  there exists simple paths in $\U_n$ joining $(z_i,z'_i)$ of length $d_i$, and let $D$ be the maximum number of those paths that intersect any edge, then
$$\mathcal{E}(\tilde \theta)\le \mathcal{E}(\theta)- \frac{[\sum_{i=1}^{k}G_n(z_i)-G_n(z'_i)]^2}{D^2\sum_{i=1}^{k} d_i}.$$ 
\end{lem} 

Thus the lemma measures the amount of voltage drop when points at different potentials are glued together.

\begin{proof} Without loss of generality assume for all $i, $ that  $G_n(z_i)\ge G_n(z'_i).$ Take a path joining $z'_i$ to $z_i$ of length $d_i,$ and think of it as oriented towards  $z_i.$  Since on $\tilde \U_{n}$ the points $z_i$ and $z'_i$ are glued, the oriented path becomes a directed cycle. Let us denote it by $\overset{\rightarrow}{\cP_i}.$  Also, let $\overset{\leftarrow} {\cP_i}$ be the same cycle in the reverse orientation. We first create a new flow ${\theta}_{A}$ on $\U_n$ by sending an additional amount of flow $p$ along $\overset{\rightarrow}{\cP_i}$ for all $i=1,\ldots,k.$  Thus,  $${\theta}_{A}(\vec{e})=\theta(\vec{e}) + \sum_{i=1}^{k}p \mathbf{1}(\vec{e}\in\overset{\rightarrow}{ \cP_i})+\sum_{j=1}^{k}-p \mathbf{1}(\vec{e}\in\overset{\leftarrow} {\cP_j})$$ for all directed edges $\vec{e}$, where every undirected edge in $\U_n$ appears with both orientations.
Now the additional flow along an edge is at most $Dp$ since by hypothesis every edge is in at most $D$ many cycles. 
An easy computation by expanding the squares now shows that,  $$\mathcal{E}({\theta}_{A})\le \mathcal{E}(\theta)+2p \sum_{i=1}^{k}(G_n(z_i)-G_n(z'_i))+ D^2 p^2\sum_{i}^{k} d_i. $$ 
Optimizing the RHS over $p$ we see that the corresponding $\theta_A$ satisfies, 
$$\mathcal{E}({\theta}_{A})\le \mathcal{E}(\theta)- \frac{[\sum_{i=1}^{k}(G_n(z_i)-G_n(z'_i))]^2}{D^2\sum_{i}^{k} d_i} .$$
Notice that since all the paths $\overset{\rightarrow} {\cP_i} $ become cycles in $\tilde{\U}_n,$ and sending  additional mass along cycles does not change the divergence of the flow,  $\div(\theta_{A}\mid _{\tilde \U_n})=\div(\theta \mid_{\tilde \U_n}),$ (note that $\theta_A$ and $\theta$ do not have the same divergence on $\U_n$).
Hence by \eqref{energymini}, we have $\mathcal{E}(\tilde \theta)\le \mathcal {E}(\theta_{A}\mid _{\tilde \U_n})\le \mathcal{E}({\theta}_{A})$
and the proof is complete.
\end{proof}

By the above discussion and \eqref{proofstra1}, to use Lemma \ref{gain1} in the proof of Lemma \ref{resquant1}, one has to estimate the $d_i$'s and $D$ appearing when the glued set is given by \eqref{glueset1}.
In the next lemma we provide such estimates, which are not optimal but would suffice for our purposes. 
Recall the definition of the sets $\U^{(i),n}$ from \eqref{newnot1}.
\begin{lem}\label{disjointpath23} Given a  domain $\U,$ for any small enough $d,\e>0$ (depending on $\U$), there exists $C,a>0$ such that for all $\dd< \dd_0(d)$ and all $n=2^{m}>N(\dd)$ the following is true:
 For any set $A\subset \U_n$ containing $d n$ points, each of which is in a distinct $\U^{(i),n}$, and all of which are at Euclidean distance of at least $\e$ from the sources $\{x_{\rm{B}},x_{\rm{R}}\}$,
 there exists at least $\frac{d n}{20}$  disjoint pairs $(z_1,z'_1),\ldots (z_{\frac{d n}{20}},z'_{\frac{d n}{20}})$ of elements of $A$  such that: 
\begin{itemize}
\item [i.]  $|G_{n}(z_j)-G_{n}(z'_j)| > a$ for $j=1, \ldots \frac{d n}{20},$ 
\item [ii.]
There exist simple paths in  the graph $\U_n$ between $z_j$ and $z'_j$ of length at most $Cn$, such that no edge in $\U_n$ intersects  more than one path.
\end{itemize}
\end{lem}

\begin{remark}\label{suff1} It will be useful later to note that given $\e>0$, any set $A\subset \U_n$ with $|A|\ge cn^2$ for some constant $c$ (independent of $n$) much larger than $\e$  satisfies the hypothesis of the above lemma for a certain value of $d=d(c,\U).$ 
\end{remark}

\begin{proof}[Proof of Lemma \ref{disjointpath23}]
We first prove the second part.  Note that the statement is not difficult to prove, when the underlying graph is a rectangle and  the shells $\U^{(i),n}$ are horizontal strips (see for  e.g., \cite[Sec 4.3]{cyl}). 
A similar argument works on a general domain but there are certain subtle issues one needs to be careful about which we point out below.
 Recall from \eqref{newnot1}, that all the $\U^{(i),n}$'s have width  $\frac{100}{n}.$ Now, let $z_i \in A \cap \U^{(i),n}. $  
On $\U_n,$ the only source of concern could be if the points lie close to the boundary where the graph does not look like $\Z^2.$ Note that by the smoothness assumption on  the boundary of $\U,$  the domain near the boundary  looks like a half plane (see \eqref{localhalf}) and hence is similar to the situation in a rectangle.
Hence, in that case for each  such $z_i$ one can take a path in $\U_n$ starting from $z_i$ ending at say $y_i,$ so that the path  lies entirely in $\U^{(i-1),n} \cup \U^{(i),n} \cup \U^{(i+1),n},$ and $y_i$ lies away from the boundary.  The construction on the rectangle can now be applied. 
To avoid intersection of these paths we only consider a subset of the  points $z_{i},$ such that the shells $\U^{(i),n}$ containing them, are separated from each other by at least three shells. 
The proof is now completed by pairing the points $y_{i}$ into pairs $(y_j,y'_j)$ such that corresponding  $\U^{(i),n}$'s for the elements of a pair are at least $\frac{d}{20}$ apart, (see Figure \ref{f.connpath1} for an illustration). 

Now  by arguments similar to the case of a rectangle one can see that  the elements of each such pair  $(y_j,y'_j)$ can be joined using disjoint lattice paths of length at most $Cn,$ for some $C$ depending only on $\U$. Thus, the path between $z_j$ and $z_j'$ is obtained by concatenating the  paths from $z_j$ to $y_j$,  $y_j$ to $y_j',$ and  $y_j'$ to $z_j'$. 
\begin{figure}[hbt]
\centering
\includegraphics[scale=.8]{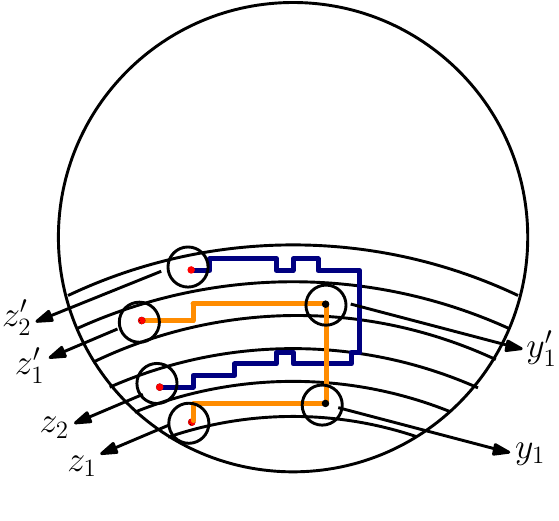}
\caption{Illustrating the proof of Lemma \ref{disjointpath23} in the case of the disc where points in different $\D^{(i),n}$ are connected by edge disjoint paths. To avoid cluttering in the figure, we do not draw too many shells, and hence do not follow the above prescribed rule of ignoring at least three shells between points.} 
\label{f.connpath1}
\end{figure}

For the first part, observe that by the above construction for any pair $z_j$ and $z'_j$ the corresponding $\U^{(i),n}$'s are at least $\frac{d}{20}$ apart.  Thus,  by Theorem \ref{convergence} and Lemma \ref{unifcon}, there exists $\dd_0$, such that given $\dd< \dd_0,$
$$\liminf_{{n=2^m}\atop{ m \to \infty}}[\inf_{j}(G_{n}(z_j)-G_{n}(z'_j)] >  c,$$ for some constant $c$ depending only on $d,\e,$ and $\U.$ Hence we are done.
\end{proof}
We now have all the necessary ingredients to finish the proof of Lemma \ref{resquant1}.
\begin{proof}[Proof of Lemma \ref{resquant1}]\label{finalcase}  We will show: for all $\sigma \in \Omega \setminus \Omega_{(\e)}$ there is a set $A$ which is a subset of the set in \eqref{glueset1} (hence is glued to get $\overset{wired}{\U_n}$ from $\U_n$) such that: $A$ satisfies the hypothesis of Lemma \ref{disjointpath23}.
Before showing the above we first discuss why it suffices.
Immediately using Lemma \ref{disjointpath23} we have pairs 
$(z_1,z'_1),\ldots (z_k,z'_k)$  with $k=\Omega (n)$ which are glued and have paths connecting them of length $O(n)$ such that no edge appears in more than a single path.
Also, $$G_{n}(z_j)-G_{n}(z'_j)\ge c',$$ for $j=1\ldots k.$ Recalling $\overset{wired}{\mathcal{E}}$ from \eqref{minenergy}, we see that by  Lemma \ref{gain1} 
$\overset{wired}{\mathcal{E}}\le \mathcal{E}-d,$ for some constant $d=d(\e).$ 
Hence we are done  by  \eqref{proofstra1}.

We now proceed to show existence of such sets $A.$ 
Recall $\beta$ from \eqref{areaparam} 
and let $
\beta_{-4}>\beta_{-3}>\beta_{-2}>\beta_{-1}>\beta_0=\beta>\beta_1>\beta_2>\beta_3>\beta_4,$ be such that for all $i=-4,-3,\ldots ,3$, we have $d({\rm{Geo}}_{\beta_{i}},{\rm{Geo}}_{\beta_{i+1}})= \e_1 \,(\e_1 \mbox{ is specified soon}).$ 
\begin{figure}[hbt]
\center
\includegraphics[scale=.5]{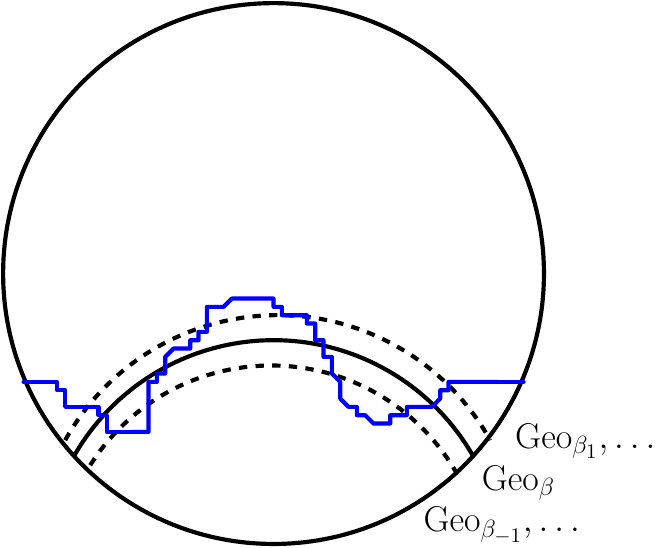}
\caption{${\rm{Geo}}_{\beta_i}$'s on either side of ${\rm{Geo}}_{\beta}$ measure how far the boundary of $\mathring{R}_{\rm{B}}$, i.e., the blue curve, deviates from ${\rm{Geo}}_{\beta}$. The first case is when the blue curve is below ${\rm{Geo}}_{\beta_{-2}}$ and similarly the second case considers the case when it is above ${\rm{Geo}}_{\beta_{2}}$. The third case deals with the situation when the blue curve which is $*$-connected, intersects both ${\rm{Geo}}_{\beta_{i}}$ and ${\rm{Geo}}_{\beta_{i+1}}$ for some $i,$ and the last case is when it is contained in a band around ${\rm{Geo}}_{\beta}.$}
\label{f.bfar1}
\end{figure}
Recalling the notations from \eqref{newnot1} the proof involves considering the following cases:
\begin{enumerate}[i.]
\item $\partial_{in} \mathring{R}_{\rm{B}}$ is a subset of  $\U_{\ge\beta_{-2},n},$
\item $\partial_{in} \mathring{R}_{\rm{B}}$ is a subset of $\U_{< \beta_{2},n},$
\item  $\partial_{in} \mathring{R}_{\rm{B}}$ intersects $\U_{\ge\beta_i,n}$ and $\U_{<\beta_{i+1},n}$ for some $i \in \{-4,\ldots, 3\},$ 
\item  $\partial_{in} \mathring{R}_{\rm{B}}$ is a subset of $\U_{<\beta_{-3},n}\cap \U_{\ge \beta_{3},n},$
\end{enumerate}
(see Figure  \ref{f.bfar1} for an illustration).
We choose $\e_1= \frac{\e}{100}$ and hence it follows that,
 \begin{equation}\label{choice1}
 {\rm{area}}(\U_{<\beta_{j-1}}\cap \U_{\ge\beta_{j}}) \le \frac{C\e}{100},  
 \end{equation} for all $j=-4,\ldots 3,$ for some universal constant $C>0$ depending on $\U$ 
($100$ is just  a large enough number and is nothing special). Also let $\e'>0$ be such that the LHS in \eqref{choice1} is at least $\e'.$ It is easy to see that $\e'\ge c\e_1$ for some $c=c(\U)>0.$
The arguments for cases i., and ii. are  symmetric and hence to avoid repetition we provide arguments only for case i.

By hypothesis $\partial_{in} \mathring{R}_{\rm{B}}$ lies entirely in $\U_{\ge\beta_{-2},n}$. Now by arguments similar to the proof of Lemma \ref{disjointpath23} ii.  for all large enough $n,$ any two points in $ \U_{<\beta_{-1},n}$ are connected by a path in $\U_n$ which lies entirely in $\U_{<\beta_{-1}+\frac{\e_1}{2},n}$ and hence does not intersect $\partial_{in}\mathring{R}_{\rm{B}}$. Thus by definition
$ \U_{<\beta_{-1},n}$ is in $\U_n \setminus \mathring{R_{\rm{B}}}.$ Now recall that $${\rm{area}}(\U_{<\beta_0})=(1-\alpha){\rm{area}}(\U).$$ 
Thus
${\rm{area}} (\U_{<\beta_{-1}})\ge (1-\alpha){\rm{area}} (\U)+\e'.$
Hence the number of blue vertices in  $\U_{<\beta_{-1},n}$ is at least $\e' n^2$.  
Recall from \eqref{glueset1} that both  $\sigma^{-1}({\rm{R}})\cap \mathring{R}_{\rm{B}}$ and $\sigma^{-1}({\rm{B}})\cap \{\U_n \setminus \mathring{R}_{\rm{B}}\}$ are glued to obtain $\overset{wired}{\U_n}$ from $\U_n$. 
Thus at least $\e' n^2$ vertices are glued, and hence we are done by Remark \ref{suff1}. 

 iii.  By hypothesis  in this case $\partial_{in}{\mathring{R}_{\rm{B}}}$ (which is a $*-$ connected set by Remark \ref{connrem}) intersects both ${\U}_{\ge \beta_i,n}$ and ${\U}_{<\beta_{i+1},n}.$ Since by construction $d({\rm{Geo}}_{\beta_i}, {\rm{Geo}}_{\beta_{i+1}})= \e_1$, it follows that $\partial_{in}{\mathring{R}_{\rm{B}}}$ intersects at least $d n$ many $\U^{(i),n}$'s for some $d=d(c,\e_1)$ in the region $\U_{<\beta_i,n}\cap \U_{\ge\beta_{i+1},n}$. Thus $A=\partial_{in} \mathring{R}_{\rm{B}}$ satisfies the hypothesis of  Lemma \ref{disjointpath23} and  is also glued by \eqref{glueset1}. 
Therefore we are done.

 iv. Recall that by hypothesis $\sigma \in \Omega^c_{(\e)}$ and $\partial_{in} \mathring{R}_{\rm{B}}$ is a subset of $\U_{<\beta_{-3},n}\cap \U_{\ge\beta_{3},n}.$  Since $\sigma \in \Omega^c_{(\e)}$ and $\e_1 = \frac{\e}{100}$, there exists at least $\frac{\e n}{20}$ many $\U^{(i),n}$'s  containing at least one red vertex all of which are in $\U_{\ge\beta_{-3},n}$ or  there exists at least $\frac{\e n}{20}$ many $\U^{(i),n}$'s containing at least one blue vertex, all of which are in $\U_{<\beta_{3},n}$. We are done by taking the set of such points to be $A$.
\end{proof}
\section{Proof of Theorem \ref{mainresult}}\label{pmr}

Given the preceding results, the proof of the main theorem follows from the following two hitting time results for the competitive erosion chain $\{\varsigma_t\}$. 
\begin{lem}\label{linehit2} Given any $\e>0,$ there exist
a positive constant $\e_1,$ such that for all small enough $\dd$ (blob size), there exist
positive constants $c=c(\e,\dd),d=d(\e,\dd),$ such that 
for all large enough $n=2^m,$ and $\sigma\in \Omega_{(\e)}\cap \mathcal{A}_{\e_1},$ 
\begin{eqnarray}\label{ht11}
\mathbb{P}_{\sigma}(\tau(\mathcal{G}_{{\e}^{1/4}}) >dn^2) & \le &e^{-cn}.
\end{eqnarray}
\end{lem}

\begin{lem}\label{main1}Given $\e>0,$ there exist
a positive constant $\e',$ such that for all small enough $\dd$ (blob size), there exist positive constants $b=b(\e,\dd),d=d(\e,\dd),$  such that 
for all large enough $n=2^m$ and  $\sigma\in \mathcal{G}_{\e}\cap \Gamma_{\frac{\e'}{2}},$
\begin{equation*}\mathbb{P}_{\sigma}
( \tau (\mathcal{G}^c_{{\e}^{1/4}}) >e^{d n})>1-e^{-b n}.
\end{equation*}
\end{lem}
The proof of the above two lemmas require developing robust arguments for IDLA on $\U_n$. These are presented in the next section. 
We now show how to finish the proof of Theorem \ref{mainresult} using the above. We start by proving the following lemma.
Given any positive $\e_1,$ and $\e_2,$ 
define the set, $$\mathcal{C}_{\e_1,\e_2}=\mathcal{G}_{\e_1}\cap \Gamma_{\e_2}.$$ 

\begin{lem}
\label{t.quickhit}
Given small enough $\e,\e'>0$, for all small enough $\delta,$ 
there exists  constants $D,D'>0$ (depending on $\delta$) such that for all large enough $n=2^m$ and any $\sigma\in \Omega,$
\begin{align}\label{hittingestimate1}
\mathbb{P}_{\sigma}\bigl(\tau(\mathcal{C}_{{\e}^{1/4},\frac{\e'}{2}})  \ge Dn^2 \bigr)
& \le e^{-D'n},\\
\label{hitest56}
 \E_{\sigma}\bigl(\tau(\mathcal{C}_{{\e}^{1/4},\frac{\e'}{2}})\bigr)& \le 2 Dn^2. 
\end{align}
\end{lem}
\begin{proof}
The bound on the expected hitting time  follows from  \eqref{hittingestimate1}, as the latter holds for any starting $\sigma$ and hence the expected hitting time is clearly dominated by $Dn^2$ times a geometric variable with mean $2$. 
Given $\e,\e'$, recalling the notation \eqref{stopnot}, let
\begin{eqnarray*}
\tau' = \tau (\Gamma_{\frac{\e'}{4}}), \tau'' = \tau({\Gamma_{\frac{\e'}{2}}^\mathsf{c}},\tau'), \\
\tau''' = \tau({\Omega_{(\e)}}\cap \cA_{\e_1},\tau'), \tau'''' = \tau({\mathcal{G}_{{\e}^{1/4}}},\tau'''),
\end{eqnarray*}
where the $\e_1$ in the definition of $\tau''',$ is the same as the one in the hypothesis of Lemma \ref{linehit2}. 
Now to show \eqref{hittingestimate1}, 
we notice the following containment of events for any positive $A$:
$$\{\tau'\le An^2\} \cap\{\tau'' \ge e^{cn^2}\}
\cap \{\tau''' \le An^2 \}
\cap \{\tau'''' \le An^2 \}
\subset \{\tau({\mathcal{C}_{{\e}^{1/4},\e'/2}})\le 3An^2\}.$$ 
To see why this containment holds, we first notice that $\Gamma_{\frac{\e'}{4}}\subset \Gamma_{\frac{\e'}{2}}.$
Thus, the first two events imply that the process 
hits the set $\Gamma_{\frac{\e'}{2}}$ in $An^2$ steps, 
and stays inside, for an exponential (in $n^2$) amount of time, 
and in particular stays in $\Gamma_{\frac{\e'}{2}}$ 
from time $An^2$ through time $3An^2$. 
In addition, the third and fourth events together imply that 
regardless of where the chain is at time $An^2$,
the chain enters $\Omega_{(\e)}\cap \cA_{\e_1}$ by time $2An^2,$
and then enters $\mathcal{G}_{{\e}^{1/4}}$ by time $3An^2$.
Hence, in the intersection of the four events,
the hitting time of 
$\mathcal{C}_{{\e}^{1/4},\frac{\e'}{2}}=\mathcal{G}_{{\e}^{1/4}}\cap \Gamma_{\frac{\e'}{2}}$ 
is at most $3An^2$.  Let $D=3A$.
Now, for a large enough constant $A$,  there exists  $D'>0,$ such that the probabilities 
 of all the four events on the left hand side are at least $1-e^{-D'n}$.
This follows from Lemmas  \ref{hitting}, \ref{goodbad},  and \ref{linehit2}  respectively.
Hence by union bound, \eqref{hittingestimate1} is true for a slightly smaller value of $D'$. 
\end{proof}
\subsection{Proof of Theorem \ref{mainresult}}\label{pom1}
Note that \eqref{hittingestimate1} is true for all small enough choices of $\e$ and $\e'.$ 
However, given $\e>0,$ by Lemma \ref{main1} there exist $\e',b,d,$ such that for all large enough $n=2^m$, and  $\sigma \in \mathcal{C}_{{\e},\frac{\e'}{2}},$
\begin{equation}\label{last1}\mathbb{P}_{\sigma}
( \tau (\mathcal{G}^{c}_{\e^{1/4}}) >e^{dn})>1-e^{-bn}.
\end{equation}
For such an $\e',$ the proof of Theorem \ref{mainresult} follows by using Lemma \ref{hitstation} with the following choices of parameters:
\begin{eqnarray*}
A_1= \mathcal{C}_{{\e},\frac{\e'}{2}}, A_2=\mathcal{G}^c_{{\e}^{1/4}}, t_1=2Dn^2, t_2=e^{dn}.
\end{eqnarray*}
The above choice of parameters satisfy the hypotheses of Lemma \ref{hitstation} by  \eqref{hittingestimate1} and \eqref{last1}.
\qed

\section{Hitting time results using IDLA estimates}\label{IDLArobust}
In this section we provide the proofs of Lemmas \ref{linehit2} and \ref{main1} that were used in Section \ref{pmr}. 
We start with a generalized IDLA estimate on $\U_n,$ (recall the discussion about IDLA from Section \ref{idla1}).  
To ensure brevity, we will omit stating the standard definitions regarding IDLA and only include the new definitions that we will need. 
 For the standard definitions and other results about IDLA on $\Z^d$ see \cite{lbg}. 
 We now precisely define the IDLA process in our setting. To proceed, we first define
adding one particle started at $x$ to an existing aggregate $S\subset \U_n$. For $x \in \U_n$, denote by $A(S; x),$
the IDLA aggregate obtained as follows:  

Let $X= \{X_0, X_1, \ldots\}$ be the discrete time simple random walk on $\U_n,$
with $X_0 = x$ and recall that $\tau(\U_n\setminus S),$ denotes the first exit time from $S$ for the walk. Define,
\begin{align}\label{IDLAcrucnot}
A(S; x) := S \cup \{X_{\tau(\U_n\setminus S)}\}. 
\end{align}
Below we consider the setting when the starting point of the random walk $X_0$ is distributed uniformly over $\U_{{\rm{R}},n}.$ 
The IDLA aggregate is now a growing cluster which at time $t$ will be denoted by $A(t)$ with $A(0)$ denoting the initial aggregate.
One of the key differences between our setting and \cite{lbg} is that the initial aggregate $A(0)$ is not an empty set and consists of all the sites in $\U_{n} \setminus \U_{(\alpha),n}$  along with all the sites in an arbitrary $\e' n $ many $\U^{(i),n}$'s for some small $\e' >0$ (for the definitions of $\U^{(i),n}$ and $\U_{(\alpha),n}$ see \eqref{newnot1} and \eqref{optimalconf} respectively).
 The next result will bound the size of the cluster after $\e n^2$ particles have been released, where  $\e \ge D\e'$ for some constant $D>0.$   
\begin{thm}\label{IDL1}(IDLA upper bound) There exists positive constants $D,C,c$ depending only on $\U,$ and $\alpha,$ such that for all small enough $\e',\e>0,$ with $\e>D\e',$
 and starting with any initial condition $A(0)$ as described above, for all large enough $n,$ with probability at least $1-e^{-c {\e} n},$
$$A(\e n^2)\setminus A(0) \subset \U_n \setminus \U_{(\alpha-C\sqrt \e),n}.$$ 
\end{thm}
\begin{figure}[hbt]
\centering
\includegraphics[scale=.58]{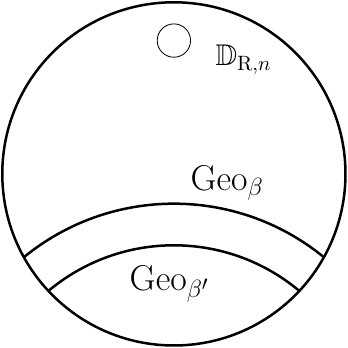}\\
\caption{Illustrating Theorem \ref{IDL1} in case of the disc. The geodesics $\rm{Geo}_\beta$ and ${\rm{Geo}}_{\beta'}$ denotes the boundaries of $\D_{(\alpha)}$ and $\D_{(\alpha-C\sqrt{\e})}$ respectively. The above theorem says that if the initial aggregate contains all of the region above $\rm{Geo}_\beta$, i.e., $\D\setminus \D_{(\alpha)}$, and at most $\e' n$ (where $\e' \ll \e$) many thin strips $\D^{(i),n}$ (see \eqref{newnot1}), then in $\e n^2$ rounds of IDLA on $\D_n,$ where the random walks start from $\D_{{\rm{R}},n}$, the aggregate does not occupy any new site below ${\rm{Geo}}_{\beta'},$ i.e., the new part of the aggregate stays in $\D \setminus \D_{(\alpha-C\sqrt \e)}$.}
\label{idla12}
\end{figure}

\begin{remark}Note that since $\alpha$ is fixed, for all small $\e,$ the geodesics which act as boundaries of $\U_{(\alpha)}$ and $\U_{(\alpha-\e)}$ are up to  constants, at distance $\e$ apart. Thus, the above result says that if the initial aggregate consisted of all sites in $\U_n \setminus \U_{(\alpha),n}$ and at most  $\e'n$ many shells, $\U^{(i),n},$ then in the next $\e n^2$ rounds, the aggregate formed by the new particles is confined within a distance of  $\sqrt{\e}$ from $\U_n \setminus \U_{(\alpha),n}$. Note that one should expect the correct bound to be $O(\e)$ instead of $O(\sqrt{\e}).$ However the weaker bound is sufficient for our purpose.
\end{remark}

\begin{remark}\label{domination123}
The IDLA process $A(t)$ (where only the red cluster grows) and the competitive erosion process $\varsigma_t,$
 can be coupled where the same red random walk is used for the IDLA  and the competitive erosion process. Under the above coupling 
if $A(0)=S_{\rm{R}}(\varsigma_0)$ i.e., the initial set of red particles are same in both the processes, then for all subsequent times, $S_{\rm{R}}(\varsigma_t)\subset A(t)$.
\end{remark}
Moreover, in the subsequent arguments we also need to consider a similar IDLA process by the blue particles. Clearly the above results continue to hold for such a process as well. 

\subsection{Proofs of Lemmas \ref{linehit2} and \ref{main1}} Before proving Theorem \ref{IDL1}, we finish all the remaining  proofs. 
Note that, to prove Lemma \ref{linehit2}, we have to prove a lower bound on the growth of the IDLA cluster. 
We follow the approach in \cite {lbg} (the reader is encouraged to refer to \cite{lbg}, for a more detailed explanation of some of the arguments appearing in the proof). Similar statements for IDLA, on the entire lattice $\Z^2,$ have cleaner proofs owing to symmetry of $\Z^2.$ However, in our setting, we need more robust approaches and hence  will rely on certain heat kernel bounds for the random walk on $\U_n$ from \cite{wtf}. 
 %which are summarized in \cite[Theorem 5.2]{tc1}. 
We first quote the following hitting time estimate for the random walk on $\U_n$.
\begin{lem}\label{gflower}\cite[Lemma 5.3]{tc1} Given a small enough $\e> 0$,  for all large enough $n$, for all $z \in \U_{(\alpha-\sqrt{\e}),n}$,  and $y\in \U_n,$ such that $d(y,z) \le \e^2$, 
$$\mathbb{P}_{y}(\tau(z)< \tau(\U_n\setminus  \U_{(\alpha-\e),n}))= \Theta(\frac{\log(1/d(y,z))}{\log n}),$$
where the constant in the $\Theta(\cdot)$ notation, depends on $\e, \alpha, \U.$
\end{lem}
\begin{proof}[Proof of Lemma \ref{linehit2}] Since by hypothesis $\sigma \in \Omega_{(\e)}\cap \cA_{\e_1},$ there are at most $ \e_1 n^2$ red vertices in $\U_{(\alpha),n}$ and similarly there are at most $\e_1 n^2$ blue vertices in $ \U_n\setminus \U_{(\alpha),n}$.  To bound the hitting time of $\cG_{\e^{1/4}},$ we need to bound the time it takes for the red particles in $\U_{(\alpha-\e^{1/4}),n}$ to turn blue and also the time taken by  blue particles in $\U_n\setminus\U_{(\alpha-\e^{1/4}),n}$ to turn red.
Because of the obvious similarity in the situations, we  provide arguments only for the former case.
The first step in the proof compares the IDLA process $A(t)$ (where only the red cluster grows) and the competitive erosion process $\varsigma_t.$
Now, by Remark \ref{domination123} and Theorem \ref{IDL1}, starting from $\sigma,$ with probability at least $1-e^{c {\e}n},$ no red walker reaches $\U_{(\alpha-C_2\sqrt \e),n}$ in the next $ C_1 \e n^2$ rounds for the IDLA process and hence for the competitive erosion process, where $C_2$ depends on $C_1$. 

Crucially, observe that on this event, the set of blue sites inside $\U_{(\alpha-C_2\sqrt{\e}),n},$ i.e., $\{S_{\rm{B}}(\varsigma_t): t \le C_1 \e n^2\}$ dominates a `killed' blue IDLA process where the blue walkers start uniformly from $\U_{{\rm{B}},n}$ and perform IDLA i.e., occupy the first vacant (red) site that they hit, but are killed on hitting $\partial_{out} \U_{(\alpha-C_2\sqrt \e),n},$ the outer boundary (see \eqref{bdry1}) of the set $ \U_{(\alpha-C_2\sqrt \e),n}$. This is because, on the aforementioned event,  no  red random walker reaches the region $\U_{(\alpha-C_2\sqrt \e),n}$. 

From the above discussion, it suffices to show that in the next  $ C_1 \e n^2$ rounds of the `killed' blue IDLA process, the blue walkers will occupy and hence wipe off all the red vertices in $\U_{(\alpha-{\e}^{1/4}),n}$, (by  assumption on the starting configuration $\sigma$  there are at most $ \e_1 n^2$ such red vertices) which amounts to proving a lower bound for the `killed' IDLA aggregate for the blue walkers. We adapt the proof in \cite{lbg}.   For the purposes of certain definitions, following the same strategy from \cite{lbg}, even if a blue random walk stops on hitting a red particle, we sample the entire trajectory of the walk till it gets killed on exiting $\U_{(\alpha-C_2\sqrt \e),n}.$
Let us denote by ${\rm{IDLA}}_{\rm{B}}(t),$ the blue aggregate for the killed IDLA process after the $t^{th}$ round.

Assuming this for any positive integer $i,$ we associate the following stopping times to the $i^{th}$ blue walker:  
\begin{itemize}
\item $\tau^{i}_{\rm{IDLA}}$: the stopping time in the killed IDLA process (i.e.,  when the blue walker either occupies a red site or gets killed by hitting $\partial_{out} \U_{(\alpha-C_2\sqrt \e),n},$ 
\item $\tau^{i}_z=\tau(z)$: the hitting time of $z \in \U_n,$
\item $\tau^{i}_{\rm{exit}}=\tau(\U_n \setminus \U_{(\alpha-C_2\sqrt \e),n})$:  the exit time of the set 
$\U_{(\alpha-C_2\sqrt \e),n}.$
\end{itemize}   
For $z \in  \U_{(\alpha-\e^{1/4}),n},$ we define the random variables,
\begin{eqnarray*}
N&=&\sum_{i=1}^{{C_1\e n^2}}\mathbf{1}{(\tau^{i}_z<\tau^{i}_{\rm{IDLA}})}\mbox{: the number of blue walkers that visit $z$ before stopping},\\
M&=&\sum_{i=1}^{{C_1\e n^2}}\mathbf{1}{(\tau^{i}_z<\tau^{i}_{\rm{exit}})}\mbox{: the number of blue walkers that visit $z$ before exiting $\U_{(\alpha-C_2\sqrt \e),n}$},\\
L&=&\sum_{i=1}^{{C_1\e n^2}}\mathbf{1}{(\tau^{i}_{\rm{IDLA}}<\tau^{i}_z <\tau^{i}_{\rm{exit}})}  \mbox{: the number of blue walkers that visit $z$ before exiting} \\
& & \U_{(\alpha-C_2\sqrt \e),n}, \mbox{ but after stopping}.
\end{eqnarray*}
%Note that in the killed process, a blue random walk can stop in two ways: either by occupying an `empty' (red) vertex or by exiting $\U_{(\alpha-C_2\sqrt \e),n}.$ 
Thus $N\ge M-L.$ 
Hence, 
\begin{equation}\label{keybound}
\mathbb{P}\Bigl(z\notin {{\rm{IDLA}}_{\rm{B}}}( {C_1 \e n^2})\Bigr)=\mathbb{P}(N=0)\le \mathbb{P}(M<a)+\mathbb{P}(L>a),
\end{equation}
where the last inequality holds for any $a.$
Now, by definition,  
\begin{equation}\label{inter1}
\E(M)= C_1 \e n^2 \mathbb{P}_{\mu_{\rm{B}}}(\tau(z)<\tau(\U_n \setminus \U_{(\alpha-C_2\sqrt \e),n})),
\end{equation}
where ${\mathbb{P}}_{\mu_{\rm{B}}}(\cdot)$ denotes the random walk measure started uniformly from $\U_{{\rm{B}},n}.$
We now bound the expectation of $L$.  Define the following quantity: let independent random walks start from each `empty' (red) site i.e $S_{\rm{R}}(\sigma)\cap \U_{(\alpha-C_2 \sqrt \e),n}$ and let, $$\tilde{L}=\sum_{{v\in S_{\rm{R}}(\sigma) \cap \U_{(\alpha-C_2\sqrt \e),n}}}\mathbf{1}{(\tau(z)<\tau(\U_n \setminus \U_{(\alpha-C_2 \sqrt \e),n}), \mbox{ for the walker starting at } v)}.$$
Clearly $L$ is stochastically dominated by $\tilde{L}.$ Hence the RHS of \eqref{keybound} can be upper bounded by $\mathbb{P}(M<a)+\mathbb{P}(\tilde{L}>a)$.
Now
$$\E(\tilde{L})= \sum_{{v\in S_{\rm{R}}(\sigma) \cap \U_{(\alpha-C_2\sqrt \e),n}}} \mathbb{P}_v(\tau(z)<\tau(\U_n \setminus \U_{(\alpha-C_2\sqrt \e),n})).$$
Note that by hypothesis there exists at most $ \e_1 n^2$ empty (red) vertices in $\U_{(\alpha-C_2\sqrt \e),n}.$ 
Since the above sum is over a set of size at most $ \e_1 n^2,$ by Lemma \ref{gflower} we have, 
\begin{align*}
\E(\tilde L) & \le C'\sqrt {\e_1} \frac{n^{2}}{\log n},\\
\E(M) & \ge  \e n^2 \frac{C''}{\log n},
\end{align*}
for some constants $C',C'',$ depending on $\e.$  
The first bound follows by summing the bound in Lemma \ref{gflower}, over a set of red points of size $\e_1 n^2$ (this is maximized, when all the points are in a ball of radius $\sqrt{\e_1}$ around $z$).
The latter bound follows since uniformly for any point $z\in  \U_{(\alpha-\e^{1/4}),n},$  the blue random walk starting uniformly from $\U_{{\rm{B}},n}$ has a chance $C_1=C(\e,\U)$ of reaching a point $\e^2$ close to $z$ before exiting $\U_{(\alpha-C_2\sqrt \e),n}.$\footnote{ One can take a tube of  small enough width joining a point in the `interior of'  $\U_n$ (say $10\e$ away from the boundary) to any point $z$ and see that the random walk stays in that tube and hits the end with the point  $z$ with constant probability $C=C(\e,\U)$, independent of $n$. That the random walk starting from $\U_{{\rm{B}},n},$ reaches such an interior point before exiting $\U_{(\alpha-C_2\sqrt \e),n},$ with constant probability, is a straightforward consequence of the heat kernel estimates \cite[Theorem 5.2]{tc1}.}  Now, once the random walker reaches such a point, using Lemma \ref{gflower}, we get a lower bound of $\frac{C''}{\log n}$ of hitting $z$ before exiting $\U_{(\alpha-C_2\sqrt \e),n}$.  
We now choose $\e_1$ small enough so that $\E(\tilde L) \le \E(M)/4.$
Choosing $a=\E(M)/2$, and using Azuma's inequality for sums of indicator variables we get,
\begin{eqnarray*}
\mathbb{P}(\tilde{L}>a) \le  \exp(-dn)\,\,\text{and}\,\,\mathbb{P}(M<a)\le  \exp(-dn),
\end{eqnarray*}
for some constant $d>0$ (independent of $n$).
Thus in \eqref{keybound} we get,
$$\mathbb{P}(M<a)+\mathbb{P}(L>a)\le 2\exp(-dn).$$ 
The proof of the lower bound now follows by taking an union bound over $O(\e_1 n^2)$ red vertices $z \in  \U_{(\alpha-{\e}^{1/4}),n}.$ 
\end{proof}
\subsection{Proof of Lemma \ref{main1}}
For any set $A\subset \Omega$ define the positive return time $\tau^{+}(A)$ to be $\inf\{t\ge 1 : \varsigma_t \in A\}.$
We break the proof into a couple of lemmas. 
\begin{lem}\label{hitlemma1} Given positive numbers $\e,c,$ there exists  $\e'>0,$ such that for all small enough $\dd$ (blob size), there exists $D=D(\e,c, \dd)>0,$ such that for any $\sigma\in \Omega_{(\e)}\cap \Gamma_{\e'},$
\begin{equation}\mathbb{P}_{\sigma}(\tau^+(\Omega_{(\e)})\ge c n^2)<e^{-Dn^2},
\end{equation}
for all large enough $n=2^m.$
\end{lem}

\begin{proof} 
The proof follows from Lemma \ref{lemm:azuma1} i. Recalling that $\varsigma_t$ denotes the competitive erosion chain with $\varsigma_0=\sigma$, 
if $\varsigma_1 \in \Omega_{(\e)}$ then $\tau^+(\Omega_{(\e)})=1$. 
Otherwise consider the process $\chi(t)=\varsigma_{t+1}$.
and  let $g= {\rm{W}}(\cdot)$ be the weight function defined in \eqref{wf}. 
We now make the following choice of parameters to apply Lemma \ref{lemm:azuma1} i:
\begin{align*}
\cB &= \Omega_{(\e)},\, a_1 = {\rm{W}}_{\max}, \,a_2 = \log(1/\dd),\\
a_3 &= a(\e)\mbox{ appearing in Theorem }\ref{negdrift123},\\
a_4  & = \e' n^2 \mbox{ with }\e' = \frac{a_3c}{2}  \mbox{ where } c \mbox{ appears in the hypothesis},\\
T &= cn^2.
\end{align*}
The choice of $a_4$ works since $\sigma\in \Gamma_{\e'}$ by hypothesis.
Also the choice of $\e'$ ensures that $a_4-a_3T <0$ and hence the hypothesis of Lemma \ref{lemm:azuma1} i. is satisfied. 
Thus for $\e'=\frac{a_3c}{2},$ by Lemma \ref{lemm:azuma1} i. for all $\sigma \in \Omega_{(\e)}\cap \Gamma_{\e'}$ 
$$\mathbb{P}_{\varsigma_1}(\tau(\Omega_{(\e)})\ge T)\le \exp\left({-D n^2}\right),$$ for some $D=D(\e,c,\dd).$
Since $\tau^{+}(\Omega_{(\e)})$ starting from $\sigma$ is one more than $\tau(\Omega_{(\e)})$ starting from $\varsigma_1,$ we are done.
\end{proof}

\begin{lem}\label{finalargument1}For all small $\e>0,$ there exists $\e'>0,$ such that for all small enough $\dd$ (blob size), there exists a constant $b(\dd,\e),$  such that for large enough $n=2^m,$ and any $\sigma\in \Omega_{(\e)}\cap \Gamma_{\e'}\cap \mathcal{G}_{{\e}^{1/4}},$
\begin{equation*}\mathbb{P}_{\sigma}(\tau^+(\Omega_{(\e)})\le \tau(\Omega \setminus \mathcal{G}_{{\e}^{1/4}}) )>1-e^{-bn}.
\end{equation*}
%n^{1/3}
\end{lem}

\begin{proof} Since $\sigma \in \Omega_{(\e)}\cap \mathcal{G}_{{\e}^{1/4}},$ by Theorem \ref{IDL1} (since IDLA dominates competitive erosion as mentioned before),
\begin{equation*}
\mathbb{P}_{\sigma}(\tau(\Omega\setminus \mathcal{G}_{{\e}^{1/4}})\ge   d n^2)\ge 1-e^{-hn},
\end{equation*}
%n^{1/3}
for some $d,h$ depending on $\e,\dd.$
Now using $c=d/2$ in Lemma \ref{hitlemma1},
we get that we can choose $\e'$ such that for any $\sigma\in \Omega_{(\e)}\cap \Gamma_{\e'}\cap \mathcal{G}_{{\e}^{1/4}},$
\begin{equation}\label{quadreturn}
\mathbb{P}_{\sigma}(\tau^+(\Omega_{(\e)})\ge \frac{d n^2}{2})<e^{-Dn^2},
\end{equation}
for some positive constant $D$.
Thus for such an $\e'$ 
\begin{eqnarray*}
\mathbb{P}_{\sigma}(\tau^+(\Omega_{(\e)})\le \tau(\Omega \setminus \mathcal{G}_{{\e}^{1/4}}) )& \ge & \mathbb{P}_{\sigma}(\tau^+(\Omega_{(\e)})\le \frac{d n^2}{2} )-\mathbb{P}_{\sigma}(\tau(\Omega\setminus \mathcal{G}_{{\e}^{1/4}})\le d n^2),\\
 &\ge & 1-e^{-hn}-e^{-Dn^2}.
\end{eqnarray*}
%n^{1/3}
Hence we are done by choosing $b=\frac{h}{2}$.
\end{proof}

\begin{proof}[Proof of Lemma \ref{main1}]%\label{secmain1}
We will specify $\e'$ later. However notice that for any small enough $\e',$ if
 $\sigma\in \Gamma_{\e'/2},$ then by Lemma \ref{goodbad}, there exist  positive constants $c,d$ depending on $\e'$ and the blob size $\delta,$ such that for large enough $n=2^m,$ 
\begin{equation}\label{ht33}
\mathbb{P}_{\sigma}(\tau(\Omega \setminus \Gamma_{\e'}) >e^{cn^2})  \ge  1- e^{-{d}n^2}.
\end{equation}

Now notice that by hypothesis $\sigma\in \mathcal{G}_{\e},$ 
and hence $\sigma \in \Omega_{(c\e)},$ for some constant $c=c(\U),$ by Lemma \ref{rem2}. 
Let $\tau_{(1)},\tau_{(2)}\ldots,$ 
be successive hitting times to $\Omega_{(c\e)}$, of the chain $\varsigma_t$ with $\varsigma_0=\sigma$
i.e.,\ $\tau_{(1)}=0$ and for all $i\ge 0$, 
$$\tau_{(i+1)}=\inf\{t:t>\tau_{(i)},\varsigma_{t} \in \Omega_{(c\e)}\}.$$  

The following containment is true for any positive $b'$: Let $s:=e^{b'n}.$
%n^{1/3}
Then
$$\{\tau(\cG^c_{{\e}^{1/4}})\le s\}\subset \left\{\exists \,i\le s \mbox{ such that }\varsigma_{i}\notin \Gamma_{\e'}\right\}\cup\left(\bigcup_{j=0}^{s}\left\{\bigl\{\tau_{(j)}< \tau(\cG^c_{{\e}^{1/4}})\le \tau_{(j+1)}\bigr\}\cap \bigl\{ \varsigma_{\tau_{(j)}}\in \Gamma_{\e'}\bigr\}\right\}\right).$$
The above follows by first observing whether there exists  any $i<e^{b'n}$ 
%n^{1/3}
such that $\varsigma_{i}\notin \Gamma_{\e'}$. Also, let $j$ be the first index such that  $\tau_{(j)}< \tau({\cG ^c_{{\e}^{1/4}}})\le \tau_{(j+1)}.$
Notice that on the event $\{\tau(\cG^c_{{\e}^{1/4}})\le s\}$ we have $j\le s.$ 
Now by definition on the event $\bigl\{\varsigma_{i}\in \Gamma_{\e'} \mbox{ for all }  i \le s\bigr\}$, we have $\varsigma_{\tau_{(j)}}\in \Gamma_{\e'}\cap \Omega_{(c\e)}\cap \mathcal{G}_{{\e}^{1/4}}.$
Thus by the union bound
\begin{eqnarray*}\mathbb{P}_{\sigma}\bigl(\tau(\cG^c_{{\e}^{1/4}})\le e^{b'n}\bigr)&\le & \mathbb{P}(\exists \,i\le s \mbox{ such that }\varsigma_{i}\notin \Gamma_{\e'})\\ &+&   \sum_{j=1}^{e^{b'n}}\mathbb{P}_{\sigma}\left(\bigl\{\varsigma_{\tau_{(j)}}\in\Gamma_{\e'}\cap \Omega_{(c\e)}\cap \mathcal{G}_{{\e}^{1/4}}\bigr\}\cap \bigl\{\tau(\cG^c_{{\e}^{1/4}})\le \tau_{(j+1)}\bigr\}\right).
\end{eqnarray*}
%n^{1/3}
Recall that by hypothesis $\sigma \in \Gamma_{\e'/2}.$ Thus by \eqref{ht33}, for any small enough $\e'$ there exists $d=d(\e',\dd)>0$ such that for any $b'>0$ the first term is less than $e^{-dn^2}$ for large enough $n$.
Also notice that by Lemma \ref{finalargument1} we can choose $\e'$ such that every term in the sum is at most $e^{-bn}$ for some constant $b>0$. 
%n^{1/3}
Thus for such an $\e'$, putting everything together we get that for any $b'>0$ and large enough $n,$ 
$$\mathbb{P}_{\sigma}(\tau(\cG^c_{{\e}^{1/4}}) \le e^{b'n}) \le e^{-dn^2}+\sum_{i=1}^{e^{b'n}}e^{-bn}.$$
%n^{1/3}
Hence by choosing $b'<b$ we get that for all large enough $n,$ $$\mathbb{P}(\tau(\cG^c_{{\e}^{1/4}})\le e^{b'n})\le 2e^{(b'-b)n}.$$
The proof is thus complete.
\end{proof}

Finally we prove Theorem \ref{IDL1}. The argument will be a  modification of the argument in \cite{hugo}, adapted to our setting (we will follow the same definitions and notation as in \cite{hugo}, to which we refer the interested reader to, for further  details).
Recalling the notation from \eqref{IDLAcrucnot}, for the purposes of this section, we consider a slightly more general process, where the growth of the
aggregate is stopped at certain stopping times,  e.g., upon exiting a set $T$. Starting with an aggregate $S\subset \U_n,$ denote by
$A(S; x\rightarrow T)$, the aggregate obtained by letting a particle doing a random walk starting from $x$, and stopping if it exits $S$ or
pausing it, if it exits $T$, i.e.,  
$A(S; x \rightarrow T) := S \cup \{X_{\tau(\U_n\setminus S) \wedge (\tau(\U_n\setminus T) - 1)}\}$.   To keep track of the position of a paused particle we define
$$P(S; x \to T) = \left\{\begin{array}{cc}
X_{\tau(\U_n\setminus T)} & \text{if } \tau(\U_n\setminus T) \le \tau(\U_n\setminus S), \\
\perp & \text{otherwise}.
\end{array}
\right.
$$
Thus, $\perp$ means that the particle is already €œabsorbed€ in the aggregate.
Given vertices $x_1, \ldots , x_k$ in $\U_n$ and a set $T$, define $A(S; x_1,\ldots , x_k \to T)$ to be the IDLA aggregate formed from an existing aggregate $S,$ by $k$ particles, started at $x_1,x_2,\ldots x_k$ and paused upon exiting $T$. Formally we use induction: $$S_0=S,\quad S_j=A(S_{j-1};x_j \to T),$$ for $j\in\{1,2,\ldots k\},$ and we have, $$A(S;x_1,\ldots,x_k \to T):=S_k.$$  Again, to keep track of paused particles, define $P(S;x_1,\ldots,x_k \to T)$ to be the sequence of particles paused in this process, i.e., if, $$p_j:=P(S_{j-1};x_j \to T) \text{ for } j \in \{1,\ldots, k\},$$ then $$ P(S;x_1,\ldots,x_k \to T) \text{ is the sequence } \{p_j: p_j \neq \perp\}.$$ We recall the \textit{Abelian property} of IDLA: $$A(S;x_1,\ldots,x_k) \text{ has the same distribution as }A(A(S;x_1,\ldots,x_k \to T );P(S;x_1,\ldots,x_k \to T)).$$ For more details about the Abelian property see \cite{lbg} and the references therein. 

To prove Lemma \ref{IDL1}, we need the following  two lemmas from \cite{hugo} adapted to our setting.
Recall the notations  $\U_{\ge a,n}$,  and $\U^{(i),n}$ from \eqref{newnot1}. 
\begin{lem}\label{hugo1}\cite[Lemma 5]{hugo} Fix a compact interval $I \in \R $. There exists positive constants $\e,\e_1,$  such that for all  large enough $n,$ and positive integer $r,$ with $n^{1/6} < r < \e n,$ and $a \in I$, the
following holds: Let $x\in  \U_{<a,n}$ and let $S \subset  \U_{<a+\frac{r}{n},n}$ be such that $|S\cap   \U_{\ge a,n}| \le \e_1 r^2.$ Moreover, let $R$ be a set which intersects at most $\e_1$ fraction of the $\U^{(i),n}$'s in $\U_{\ge a,n} \cap \U_{< a +\frac{r}{n},n}.$ 
Then if  $X=\{X_{0},X_{1},\ldots\}$ is a random walk started at $x,$ and stopped upon hitting $ \U_{\ge a+\frac{r}{n},n}$, then,
$$ \P( {\rm{Trace}}(X) \cap  \left(\U_{\ge a ,n} \setminus \{S \cup R \cup \U_{\ge a +\frac{r}{n},n}\}\right)
\neq \emptyset) \ge  \eta,
$$
for $\eta=\eta(I,\U),$ where ${\rm{Trace}}(X)$ denotes the trace of the random walk path. 
\end{lem}

Observe that $\U_{\ge a'}\subset \U_{\ge a}$ for any $ a' \ge  a.$
Thus the above lemma says that if the size of the intersection of the set $S,$ with $\U_{\ge a,n} \cap \U_{< a+\frac{r}{n},n},$ is not large, then there is a constant chance that the random walk starting from $\U_{< a,n},$ will visit sites in $\U_{\ge a,n} \cap \U_{< a+\frac{r}{n},n},$ which are neither in $S$ nor in $R,$ before hitting $\U_{\ge a+\frac{r}{n},n}.$ 

Note that on the above event, in the IDLA process the random walk will occupy a site in $\U_{\ge a,n} \cap \U_{< a+\frac{r}{n},n},$. This is the type of result which allows us to then bound the growth rate of the IDLA cluster.

After analyzing the behavior of a single particle, we can analyze the behavior of the
whole aggregate. The following lemma says that with high probability, a constant fraction
of the aggregate is absorbed in a wide enough shell. 
\begin{lem}\label{prob123}\cite[Lemma 6]{hugo} Let $I$ be as above. Then, there exists positive constants $c,c_1,c_2,$ and $0<p < 1,$ such that for all $n$ large enough, for all $n^{1/3} < k < c n^2,$ and $x_1, \ldots, x_k \in \U_{< a,n}$, and for all $S \subset  \U_{< a,n}$ where $ a \in I,$ the following is true: for any set $R,$ which intersects at most $c_1 $ fraction of the $\U^{(i),n}$'s in  $\U_{\ge a ,n} \cap \U_{<\gamma,n},$ then, 
$$\P( |A\bigl (S\cup R; x_1,\ldots, x_k \rightarrow \U_{<\gamma,n}\bigr) \setminus S\cup R| \le  c_2 k)
\le  p^k
$$
where $\gamma= a +\frac{\sqrt k}{n}.$
%$$\P( |A\bigl (S\cup B; x_1, . . . , x_k \rightarrow \U_{<\beta+\frac{\sqrt{k}}{n},n}\bigr) \setminus S\cup B| \le  \e_2 k)
%\le  p^k
%$$
\end{lem}

The above lemma says that if at any time the aggregate is a subset of $\U_{< a,n},$ along with some  $c_1$ fraction of the $\U^{(i),n}$'s in $\U_{\ge a,n} \cap \U_{<\gamma,n}$ then in the next $k$ rounds (as long as $k<cn^2$) a constant fraction of the particles fill up sites in $ \U_{<\gamma,n}$ with high probability. 
Our  hypothesis about the initial cluster in Theorem \ref{IDL1} will allow us to apply the above lemmas.
Before proving the above lemmas we finish the proof of Theorem \ref{IDL1}.
\begin{proof}[Proof of Theorem \ref{IDL1}] Note that $\U_{\ge\beta}=\U_{(\alpha)}$ (see \eqref{areaparam}) and $ \U_{(\alpha-C\sqrt{\e})} \subset \U_{\ge\beta+{C_1 \sqrt{\e}}}$ for some $C_1$ depending on $C$ (due to the same reasoning as in \eqref{order6}).  Thus it suffices to prove that the growth cluster will not touch any new points in $\U_{\ge\beta+{C_1\sqrt{\e}},n}$ for some $C_1.$

 The proof consists of inductively constructing
a sequence of aggregates $A_{(j)}$ by pausing the particles at different distances $m_j$ from the
starting set $\U_{{\rm{R}},n}$. If $k_j$ is the number of paused particles, we choose the next distance $m_{j+1}$, at which
we pause the particles again, in terms of $m_j$ and $k_j.$ We iterate this procedure until there are less than $\e n$
 %$n^{1/3}$%
paused particles. Since there are too few particles now, we naively bound the growth of the cluster from this point on.

Define $A_{(j)},m_j,P_j,k_j$ as follows: $A_{(0)}$ is the initial set of filled sites as in the hypothesis of Theorem \ref{IDL1}. Let
 $m_1=\sqrt{\e} n$ and $A_{(1)}$ be the cluster after $\e n^2$ random walks were released uniformly from $\U_{{\rm{R}},n}$ and stopped on hitting $\U_{\ge\beta_1,n}$ where $\beta_1=\beta+\frac{m_1}{n}$. Correspondingly let $P_1$ be the sequence of stopped particles and $k_1=|P_1|$.
We now define $A_{(j)}, P_j$ for all $j>1$. 
For $j>1$ define $m_{j+1}$ in the following way:
(this is where we make the adaptations from \cite{hugo} since in our setting the initial cluster has some particles and hence not all the sites in $\U_{\ge\beta_j,n}$ are empty.) 

For $i=1,2,\ldots$, let $b_i=\beta+\frac{m_{j}+ i\sqrt{k_{j}}}{n}.$
%$$ \frac{m_{j}+2 \sqrt{k_{j}}}{n}, \frac{m_{j}+3 \sqrt{k_{j}}}{n},\ldots).$$ 
%Recall $\U^{(i),n}$ from \eqref{newnot1}. 
Let $\ell$ be the smallest integer such that  at most $c_1$ fraction of the shells $\U^{(i),n}$ that intersect $\U_{\ge b_{\ell-1},n}\cap \U_{<b_{\ell},n}$, contain particles from the initial cluster $A_{(0)}$ ($c_1$ is the same as in the hypothesis of Lemma \ref{prob123}). We define $m'_j$ and $m_{j+1}$ to be $m_{j}+ (\ell-1) \sqrt{k_{j}}$ and $m_{j}+ \ell \sqrt{k_{j}}$ respectively. 

%SGThat is, we define $m_{j+1}$ to be the smallest number of the right scale, such that there are enough empty sites in $\U_{\le \beta+{m_{j+1}}/n,n}\cap \U_{\ge\beta+{m'_{j}}/n,n}.$

Let \begin{align*}
A_{(j+1)}&=A(A_{(j)};P_j \rightarrow  \U_{<b_{\ell},n}), \\
P_{j+1} &=P(A_{(j)}; P_j \rightarrow  \U_{<b_{\ell},n}),
\end{align*} and let $k_{j+1}=|P_{j+1}|.$
Also let $J$ be the (random) first time at which $k_j \le \e n$.
%n^{1/3}%
We do not stop the particles after this point and hence by construction $A_{(j)}=A_{(J+1)},$ for any $j \ge J+1.$ As mentioned earlier, the Abelian property guarantees that $A_{(J+1)}$  has the same law as the IDLA cluster with $m_0=\e n^2$ particles. We use the crude bound,  $$A_{(J+1)}\setminus A_0 \subset   \U_{<\beta+\beta',n},$$
where $\beta'=\frac{m_J}{n}+ C\frac{(\e'n+\e n)}{n},$
%n^{1/3}
 for some constant $C=C(\U)$. This follows because there are at most $\e n$ 
%$n^{1/3}$
particles left and initially at most $\e'n$ many shells $\U^{(i),n}$  filled up, and hence the growth from $J$ onwards cannot hit more than  $\e'  n+\e n$ shells $\U^{(i),n}$ (note that the initial cluster $A_0$ by hypothesis can intersect $\U_{\ge\beta+\beta',n},$ but the above argument implies that  no new site added to the aggregate intersects it).
%n^{1/3}
By Lemma \ref{prob123}, for some $c_2>0,$ and $p>0$ (depending only on $\alpha$) 
$$\P[\exists\,\,2 \le j\le J \wedge n\,:\, k_j \ge (1-c_2)^j  k_1 ] \le \P[\exists\,2 \le j\le J \wedge n \,:\, k_j \ge (1-c_2)  k_{j-1} ] = O(n) p^{\e n}.$$
%n^{1/3}
Since $k_1 \le \e n^2$ by hypothesis, for all small enough $\e$ this implies that with probability at least $1-Cn p^{\e n}$ (for some constant $C$ independent of $n$) we have  $J \le n,$  and $k_j$ 
%n^{1/3}
decreases geometrically at each step. Hence,
\begin{align}\label{boundarg2}
m_J& =m_1+ (m'_1-m_1)+\sqrt{k_{1}}+(m'_2-m_2)+\ldots + (m'_{J-1}-m_{J-1})+ \sqrt{k_{J-1}},\\
\nonumber
 & \le m_1+ \sqrt{k_{1}}\frac{1}{1-\sqrt{(1-c_2)}}+ \sum_{\ell=1}^{J-1}(m'_{\ell}-m_{\ell}).
\end{align}
Now we finish by bounding $\sum_{\ell=1}^{J-1}(m'_{\ell}-m_{\ell})$. Notice that by definition at least $c_1(m'_{\ell}-m_{\ell})$ shells intersect the initial cluster. Also by assumption on the initial cluster, number of such shells is at most $\e' n\le \frac{\e n}{D} $. 
Thus $c_1 \sum_{\ell=1}^{J-1}(m'_{\ell}-m_{\ell}) \le  \frac{\e}{D} n,$ and hence choosing $D=\frac{1}{c_1},$ we get that $\sum_{\ell=1}^{J-1}(m'_{\ell}-m_{\ell}) \le  {\e} n.$
Since $k_1 \le \e n^2,$ putting everything together it follows that with  probability at least  $1- np^{\e n}$ we have
%n^{1/3}
$
m_J\le C m_1=C\sqrt{\e}n$ for some absolute constant $C.$ 
Thus $A_{(J+1)}\setminus A_0$ does not intersect $\U_{\ge\beta+C_1 \sqrt{\e},n}$ and we are done.
\end{proof}
We now provide the proofs of Lemmas \ref{hugo1} and \ref{prob123}.
 \begin{proof}[Proof of Lemma \ref{hugo1}] Let $y$ be the first point that the random walk hits, as it exits $\U_{< a+\frac{r}{2n},n}$. Now using the smoothness of $\U,$  uniformly in $y,$ and all small enough $c,$ we have $|B_n(y,\frac{cr}{n})|\ge  \Theta(c^2 r^{2}).$   Now, if $\e_1=\e_1(c) $ is small enough, then
$|B_n(y,\frac{cr}{n}) \setminus \{S \cup A\}|\ge \Theta(c^2 r^{2}).$ 
The result now follows from the heat kernel estimate in \cite[Theorem 5.2]{tc1} which says that, the probability that the random walk started at $y,$ hits the set $B_n(y,\frac{c r}{n})\setminus \{S \cup A\}$ before exiting $B_n(y,\frac{r}{2n})$ is lower bounded by a constant independent of $r,n$.
\end{proof}

\begin{proof}[Proof of Lemma \ref{prob123}] The proof of  \cite[Lemma 6]{hugo} uses \cite[Lemma 5]{hugo}. The same  arguments as in \cite[Lemma 6]{hugo} using  Lemma \ref{hugo1} above, instead of \cite[Lemma 5]{hugo},  completes the proof.
\end{proof}

\appendix\label{app1}
\section{Proofs of some earlier statements}\label{pol}  
In this section we provide  proofs of 
 Lemmas \ref{unibounded}, \ref{closelemm}, \ref{maxbnd1} \ref{rem2}, \ref{unifcon} and \ref{step1} that were omitted in the main text.
\begin{proof}[Proof of Lemma \ref{unibounded}] 
Clearly by \eqref{transformed} and Lemmas \ref{unifcon}, and \ref{step1}, we see that $G_*(z)$ is a bounded continuous function on $\U$ (where the bound is $O(|\log(\dd)|))$.
Approximating integral by Riemann sums, we get, $$\int_{\U}|G_*(z)| \,d\xi d\eta=\lim_{n\rightarrow \infty} \sum_{z_n\in \U_n}\frac{1}{n^2}|G_*(z_n)|,$$ where $z=\xi+i\eta$. 
The proof will then immediately follow from Theorem \ref{convergence}, if we can show that,
$$\limsup_{\dd\rightarrow 0}\int_{\U}|G_*(z)|\, d\xi d\eta < \infty,$$
which follows since,
\begin{eqnarray} \label{unifbound1}
\lim_{\dd\rightarrow 0}\int_{\U}|G_*(z)|\, d\xi d\eta &=&\int_{\U}\frac{64}{\pi}\left|\log \bigl|\frac{\psi(z)-i}{\psi(z)+i}\bigr|\right| d\xi d\eta,\\
\nonumber 
&=& \frac{64}{\pi} \int_{\DD}|\phi'(z)|^2\left|\log \bigl|\frac{z-i}{z+i}\bigr|\right| d\xi d\eta, \\
\nonumber 
&<& \infty.
\end{eqnarray} 
The first equality is a consequence of Lemmas \ref{unifcon} and \ref{step1}. The second equality is just the  change of measure formula. The last inequality holds since  $|\phi'(z)|$ is bounded and the function 
$\log \bigl|\frac{z-i}{z+i}\bigr|$ is an integrable function on the disc.
Thus, we are done.
\end{proof}
\begin{proof}[Proof of Lemma \ref{closelemm}]
Let us fix a configuration where every vertex  in $\U_{(\alpha),n},$ except possibly $O(n)$ many, is colored blue and call it $\sigma_*$. The $O(n)$ correction is needed since every configuration has exactly $\lfloor{\alpha|\U_n|}\rfloor$  many blue vertices.
By Lemma \ref{unifcon}, for all $z \in \overline\U,$  such that $\min(d(z,x_{\rm{B}}),d(z,x_{\rm{R}}))\ge \dd^{a},$ we have, 
\begin{equation}\label{imperr3}
G_*(z)=\frac{64}{\pi}\log\bigl|\frac{\psi(z)-i}{\psi(z)+i}\bigr|+o(1).
\end{equation}
By Lemma \ref{step1}, on the remaining set of measure $O(\dd^{a}),$ one has $|G_*|=O(|\log(\dd)|).$
This along with Theorem \ref{convergence}, implies that given any number $c,$ for small enough $\dd,$ and $n=2^{m}>N(\dd),$  we have, 
\begin{eqnarray}
\nonumber
\left|{\rm{W}}_{\max}-n^2\frac{64}{\pi}\int_{\U_{(\alpha)}}\log\bigl|\frac{\psi(z)-i}{\psi(z)+i}\bigr|\,d\xi d\eta \right| & \le & cn^2,\\
\label{close11}
|{\rm{W}}(\sigma_*)-{\rm{W}}_{\max}|& \le & cn^2.
\end{eqnarray}
From Lemmas \ref{unifcon}, and \ref{step1}, it follows that the function $G_*(z)$ (as $\dd$ goes to $0$) is uniformly integrable i.e., given any $a,$ there exists $b=b(a)$ ($b(a) \to 0$ as $a$ goes to $0$) such that,
$$\limsup_{\dd \to 0} \sup_{A}\int_{A}G_{*}(z)d\xi d\eta\le b,$$ where the supremum is taken over all subsets $A$ of $\U$ of Lebesgue measure at most $a.$ 
Using this and Theorem \ref{convergence}, it follows that:  Given any constant $a,$ there exists a constant $b=b(a)$ such that for all small enough $\dd$ and $n=2^{m}>N(\dd)$:
\begin{equation}\label{error}
\sup_{{A\subset\U_n}\atop |A|\le an^2}\sum_{z_n\in A}|G_n(z_n)|\le bn^2,
\end{equation} where $b=b(a)$ tends to $0,$ as $a$ tends to $0$.     
Thus for any $\sigma\in \cA_{\e},$ by \eqref{error}, 
$|{\rm{W}}(\sigma)-{\rm{W}}(\sigma_{*})|\le b(2\e)n^2.$
Hence we are done by \eqref{close11}.
\end{proof}
\begin{proof}[Proof of Lemma \ref{maxbnd1}] It follows immediately from Lemmas \ref{unifcon}, \ref{step1} and Theorem \ref{convergence}.
\end{proof}

\begin{proof}[Proof of Lemma \ref{rem2}.]
The upper containment follows from \eqref{close}. The lower containment follows from \eqref{imperr3} and Theorem \ref{convergence}.  That $\Omega_{(\e_3)} \subset 
\cA_{\e}$ for some $\e_3=O(\e)$ follows immediately from the definitions. 
\end{proof}
\begin{proof}[Proof of Lemma \ref{unifcon}] Recall that the map $\phi$ is bi-Lipschitz.  Thus it suffices to show that $$\lim_{\dd \rightarrow 0}\sup_{{{z \in \overline \DD:}\atop {d(z,-i)\ge \dd^a}}\atop {d(z,i)\ge \dd^a}}\left|G_*\circ \phi(z)-\frac{64}{\pi}\log \left|\frac{z-i}{z+i}\right|\right|=0.$$
%Recall the sets $A_{\rm{B}},A_{\rm{R}}\subset  \DD$ from \eqref{mappedsource}.
Note that by hypothesis, $\U_{\rm{B}} \subset B(x_{\rm{B}},2\dd)\cap \U,$ and hence, 
\begin{equation}\label{lip12}
\psi(\U_{\rm{B}}) \subset B(-i,C\dd) \cap \DD, 
\end{equation} 
for some constant $C$ (where $C/2$ is the Lipschitz constant of $\psi$). Similar containment holds for $\U_{\rm{R}}$ and $\psi(\U_{\rm{R}}).$ 
Now, notice that by the change of variable formula for ${\rm{C}}\in \{{\rm{B}},\rm{R}\},$
\begin{equation}\label{area1}
\left|\frac{1}{\rm{ area}(\U_{\rm{B}})}\int_{A_{\rm{C}}} \tilde f\circ\phi(\zeta)|\phi'(\zeta)|^2 d\xi d\eta\right|=16.
\end{equation}
These facts along with \eqref{transformed} clearly imply that 
\begin{equation}\label{imperr}
G_*\circ \phi(z)=\frac{64}{\pi}\log\left[\frac{|z-i|+O(\dd)}{|z+i|+O(\dd)}\right],
\end{equation}
 where the constant in the $O(\cdot)$ term depends only on $C$ in \eqref{lip12}. 
 Thus, $$\left|G_*\circ\phi (z)-\frac{64}{\pi}\log \left|\frac{z-i}{z+i}\right|\right|\le \frac{64}{\pi}\left(\left|\log\left(1+\frac{O(\dd)}{|z-i|}\right)\right|+\left|\log\left(1+\frac{O(\dd)}{|z+i|}\right)\right|\right).$$
Since $a<1,$ we have $\displaystyle{\lim_{\dd\to 0}\frac{\dd}{\dd^a}=0}.$ Thus  sending $\dd$ to $0,$ we see that the RHS  uniformly converges to $0$ on the set $\{z\in \overline \DD: \min(|z-i|,|z+i|)\ge \dd^a\}.$ Hence we are done.
\end{proof}

\begin{proof}[Proof of Lemma \ref{step1}] 
We only consider the first case, since the argument for the other case is similar.  Since $a$ will be fixed throughout the proof,  $C=C(a)$ will be a universal constant independent of everything. Assuming  $d(x_{\rm{B}},z)\le \dd^{a},$ and using the bi-Lipschitz property of $\phi,$ it suffices to show that for all $z \in \overline \DD$ such that $d(z,-i)\le \dd^{a},$ $$\frac{a}{C}|\log(\dd)|\le  (G_*\circ \phi)(z) \le C|\log(\dd)|.$$
We first show the above bounds for, 
 \begin{align} \label{eq:step1}
 \frac{1}{\rm{ area}(\U_{\rm{B}})}\frac{1}{\pi}\int_{|\zeta|<1}\bigl(\mathbf{1}(\psi(\U_{\rm{R}}))-\mathbf{1}(\psi(\U_{\rm{B}}))\bigr)(\zeta)|\phi'(\zeta)|^2\log|\zeta-z|d\zeta. 
 \end{align}
Clearly, $$\frac{1}{\rm{ area}(\U_{\rm{B}})}\frac{1}{\pi}\int_{|\zeta|<1}\mathbf{1}(\psi(\U_{\rm{R}}))(\zeta)|\phi'(\zeta)|^2\log|\zeta-z|d\zeta =O(1),$$ since $d(z,\psi(\U_{\rm{R}}))>\frac{d(x_{\rm{B}},x_{\rm{R}})}{2}$ for all small enough $\dd$. 
Since $\psi$ is bi-Lipschitz,  $\psi(\U_{\rm{B}})$ is contained in a ball of radius $C \dd,$ and contains a ball of radius $c\dd$ for some constants $C,c,$ which do not depend on $\dd.$ This implies that ${\rm{area}} (\psi(\U_{\rm{B}}))=\Theta(\dd^2).$  Also,  $|\phi'|$ is  bounded away from $0$ and $\infty$.
Thus, for any $z\in \overline\DD$ we have,  
\begin{eqnarray*}
\left|\frac{1}{\rm{ area}(\U_{\rm{B}})}\frac{1}{\pi}\int_{|\zeta|<1}\mathbf{1}(\psi(\U_{\rm{B}}))(\zeta)|\phi'(\zeta)|^2\log|\zeta-z|d\zeta\right| &=& 
 O\left(\left|\frac{1}{\dd^2}\int_{0}^{C\dd}\int_{|\zeta-z|=r}\log|\zeta-z| d\theta dr\right|\right), \\
     &=& O\left(\left|\frac{1}{\dd^2}\int_{0}^{C\dd}r\log(r)dr\right|\right),\\
     &=& O\left(|\log(\dd)|\right).
\end{eqnarray*}
The lower bound follows immediately since  $d(z,\psi(\U_{\rm{B}}))\le \dd^{a}.$

 We now prove the same bound for
 \begin{align}\label{step2}
 \frac{1}{\rm{ area}(\U_{\rm{B}})}\frac{1}{\pi}\int_{|\zeta|<1}\bigl(\mathbf{1}(\psi(\U_{\rm{B}}))-\mathbf{1}(\psi(\U_{\rm{R}}))\bigr)(\zeta)|\phi'(\zeta)|^2\log|1- \bar \zeta z|d\zeta. 
 \end{align}
Again, it is easy to check that, $$\frac{1}{\rm{ area}(\U_{\rm{B}})}\frac{1}{\pi}\int_{|\zeta|<1}\mathbf{1}(\psi(\U_{\rm{R}}))(\zeta)|\phi'(\zeta)|^2\log|1-  \bar \zeta  z|d\zeta =O(1).$$
To find the order of the term,  $$\frac{1}{\rm{ area}(\U_{\rm{B}})}\frac{1}{\pi}\int_{|\zeta|<1}\mathbf{1}(\psi(\U_{\rm{B}}))(\zeta)|\phi'(\zeta)|^2\log|1- \bar \zeta z|d\zeta,$$
recall from Section \ref{para}, that $d(\U_{\rm{B}},\partial \U)=\Theta(\dd)$ and hence $d(\psi(\U_{\rm{B}}),\partial \DD)=\Theta(\dd).$ 
Thus, for any $\zeta \in \psi(\U_{\rm{B}})$, we have $C\dd \le |1-\bar \zeta  z|.$  Also, since  $d(z,x_{\rm{B}})\le \dd^a$ by hypothesis, $|1-\bar \zeta  z|=O(\dd^a).$ 
Therefore, the sought bounds follow by plugging in the above estimates, and hence by  \eqref{transformed}, we are done.
\end{proof}

\section{ Basic geometric properties of $\U$} \label{RWE}

We finish by including a short discussion of some basic geometric properties of $\U_n$ used before in the article.
\begin{figure}
\centering
\begin{tabular}{cc}
\includegraphics[scale=.3]{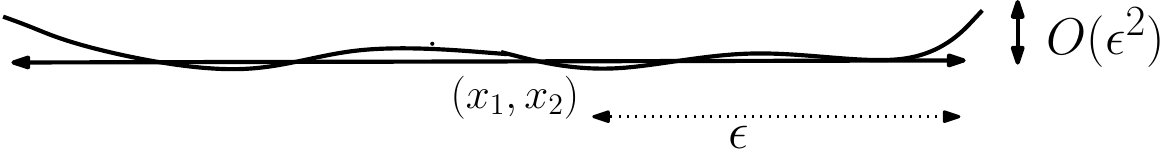}& 
\includegraphics[scale=.4]{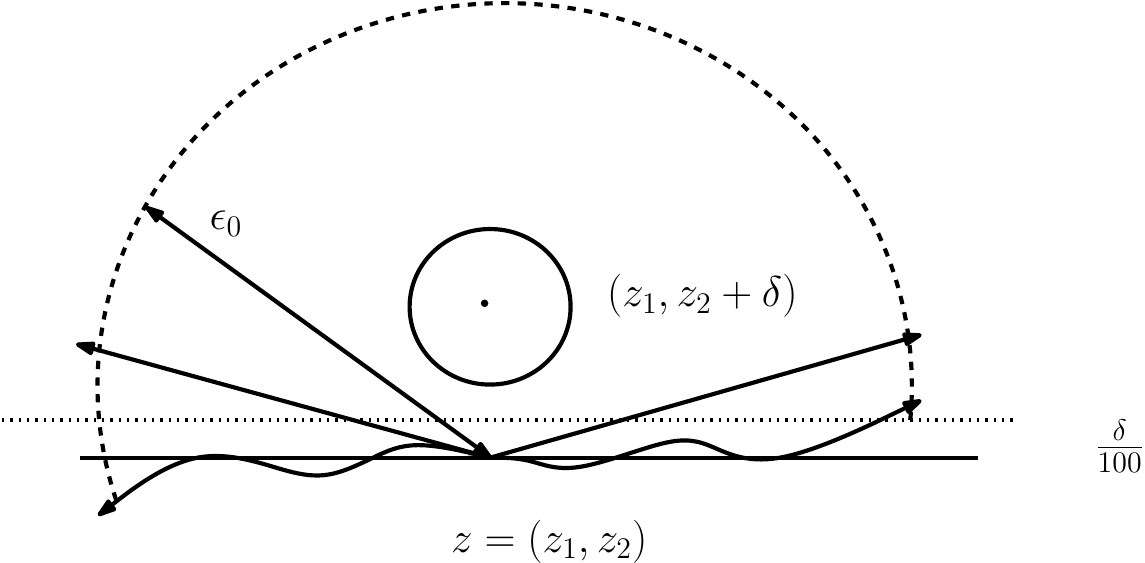} \\
i. & ii.
\end{tabular}
\caption{i. Locally the region near the boundary looks like a half plane. ii. Illustrating the proof of Corollary \ref{ias1}.}
\label{fig.iap}
\end{figure}
The following is a well known locally half plane like property of the boundary of $\U$: since the boundary $\partial \U$ is analytic, there exists a $C>0,$ and an $\e_0,$ such that for all $x=(x_1,x_2)\in \partial \U,$ there exists an orthogonal system of coordinates centered at $x$, such that for all $\e\le \e_0:$ 
\begin{equation}\label{localhalf}
B(x,\e)\cap \U=\left\{ (x_1',x_2') \in B(x,\e): x_1'\in (x_1-\e,x_1+\e),x_2'\ge f(x_1') \right\},
\end{equation}  
 and $|f(x_1')-x_2|\le C|x_1'-x_1|^2.$ The above  is a simple consequence of Taylor expansion up to second order of the curve, locally near $x$.  See Figure \ref{fig.iap} i.
As a simple corollary of the above fact, we see that $\U$ satisfies the following property, which shows that the $y_{\rm{C}}$'s  in Section \ref{para} can indeed be chosen. 
\begin{cor}\label{ias1}
Let $\U$ be as in Section \ref{para}. Then there exists $\dd_0=\dd_0(\U)$ such that for all $x\in \overline{\U}$ and $\dd<\dd_0$ there exists $y\in \U$ such that 
$ d(y,x)\le \dd\,\,\text{and}\,\, B(y,\frac{\dd}{2})\subset  \U.$
\end{cor}
\begin{proof}Choose  $\dd_0\le \frac{\e_0}{4},$ such that $C\dd_0^{2}\le \frac{\dd_0}{100},$ where $\e_0$ and $C$ appear in \eqref{localhalf}. For any $\dd< \dd_0,$ the lemma is immediate if $d(x,\partial \U)>\frac{\dd}{2}$; since, then we can choose $y=x.$ 
Otherwise, let $z=(z_1,z_2)\in \partial \U$ be the closest point on the boundary to $x$. 
Now, in the local coordinate system centered at $z$ as in \eqref{localhalf}, choose $y=(z_1,z_2+\frac{\dd}{2})$.
Then, $d(x,y)\le d(x,z)+d(z,y)\le \dd.$ Also, clearly $B(y,\frac{\dd}{2})\subset \U,$  and hence we are done. See Figure \ref{fig.iap} ii.
\end{proof}

\section*{Acknowledgments}
We are grateful to the anonymous referees for the detailed comments that helped improve the exposition of the paper.
We thank  Krzysztof Burdzy, Zhen-Qing Chen, Wai-Tong Fan, Lionel Levine, and Steffen Rohde for valuable discussions at various stages of this work.
We also thank Gerandy Brito, Christopher Hoffman, Matthew Junge and Brent Werness  for several useful comments.
The work was initiated when S.G. was an intern with the Theory Group at Microsoft Research, Redmond. He thanks the group for its hospitality.

\bibliographystyle{plain}
\bibliography{GFF}

\end{document}